\documentclass[11pt, a4paper, final]{amsart}
\usepackage[T1]{fontenc}
\usepackage{lmodern}
\usepackage{mathtools}
\usepackage[utf8]{inputenc}

\usepackage{amsfonts}
\usepackage{amsmath}
\usepackage{amssymb}
\usepackage{amsthm}
\usepackage{mathtools}
\usepackage{paralist, tabularx}
\usepackage{enumitem}
\usepackage[utf8]{inputenc}
\usepackage{mathabx,epsfig}
\usepackage{tikz-cd}
\usepackage{verbatim}

\usepackage{etoolbox}

\usepackage{xcolor}
\definecolor{green}{RGB}{0,127,0}
\definecolor{redd}{RGB}{191,0,0}
\definecolor{red}{RGB}{105,89,205}
\usepackage[colorlinks=true]{hyperref}

\usepackage[notref, notcite]{showkeys}
\usepackage[cmtip,arrow]{xy}

\DeclareMathOperator{\Aut}{Aut}

\newcommand{\R}{{\mathbb{R}}}
\newcommand{\Z}{{\mathbb{Z}}}

\DeclareMathOperator{\error}{{error}}
\DeclareMathOperator{\Mor}{Mor}
\DeclareMathOperator{\ext}{ext}
\DeclareMathOperator{\id}{id}

\newcommand{\C}{\mathfrak{C}}
\newcommand{\M}{{\mathcal M}}

\newcommand{\nref}[2]{\hyperref[#1]{\ref*{#1}$_{#2}$}}

\DeclareMathOperator{\im}{{Im}}

\DeclareMathOperator{\SL}{{SL}}
\DeclareMathOperator{\tp}{{tp}}
\DeclareMathOperator{\cl}{{cl}}
\DeclareMathOperator{\dcl}{{dcl}}

\newtheorem{theorem}{Theorem}
\numberwithin{theorem}{section}
\newtheorem{lemma}[theorem]{Lemma}

\newtheorem{fact}[theorem]{Fact}
\newtheorem{proposition}[theorem]{Proposition}

\newtheorem{conjecture}[theorem]{Conjecture}

\newtheorem{question}[theorem]{Question}
\newtheorem{corollary}[theorem]{Corollary}

\newtheorem{clm}{Claim}
\newtheorem*{clm*}{Claim}

\theoremstyle{definition}
\newtheorem{definition}[theorem]{Definition}

\theoremstyle{remark}
\newtheorem{remark}[theorem]{Remark}

\AtEndEnvironment{proof}{\setcounter{clm}{0}}
\newenvironment{clmproof}[1][\proofname]{\proof[#1]}{\endproof}

\newcommand{\orcidlogo}{\includegraphics[height=\fontcharht\font`\B]{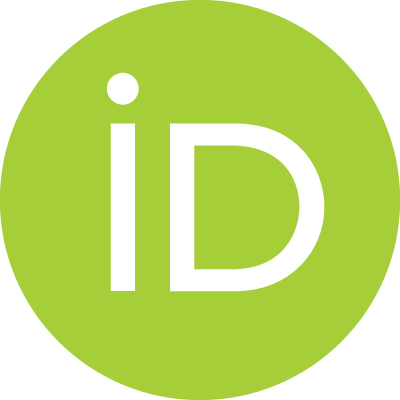}}
\newcommand{\orcid}[1]{\href{#1}{\orcidlogo #1}}

\usepackage[hyphenbreaks]{breakurl}

\usepackage[
backend=bibtex,
url=false,
isbn=false,
backref=true,
citestyle=alphabetic,
bibstyle=alphabetic,
autocite=inline,
maxnames=99,
minalphanames=4,
maxalphanames=4,
sorting=nyt,
]{biblatex}

\title{Generalized locally compact models for approximate groups}

\author{Krzysztof Krupi\'{n}ski}
\thanks{\noindent The first author was supported by the Narodowe Centrum Nauki grant no. 2016/22/E/ST1/00450.}
\address[K.\ Krupi\'{n}ski]{{Instytut Matematyczny, Uniwersytet Wroc{\l}awski\\
pl. Grunwaldzki 2, 50-384 Wroc{\l}aw, Poland}}
\address{\orcid{https://orcid.org/0000-0002-2243-4411}}
\email{Krzysztof.Krupinski@math.uni.wroc.pl}

\author{Anand Pillay}
\thanks{\noindent The second author was supported by NSF grants DMS-1665035, DMS-1760212, and DMS-2054271}
\address[A. Pillay]{{Department of Mathematics, University of Notre Dame\\
	255 Hurley Hall\\
	Notre Dame, IN 46556, USA}}
\email{apillay@nd.edu}


\begin{document}
	
	\begin{abstract}
	We give a proof of the existence of {\em generalized definable locally compact models} for arbitrary approximate subgroups via an application of topological dynamics in model theory. Our construction is simpler and shorter than the original one by Hrushovski \cite{Hru2} and it uses only basic model theory (mostly spaces of types and realizations of types). The main tools are Ellis groups from topological dynamics considered for suitable spaces of types. However, we need to redevelop some basic theory of topological dynamics for suitable ``locally compact flows'' in place of (compact) flows. We also prove that the  generalized definable locally compact model which we constructed is universal in an appropriate category. We note that the main result yields structural information on definable generic subsets of definable groups, with a more precise structural result for generics in the universal cover of $\SL_2(\R)$.
	\end{abstract}
	
\maketitle

\section{Introduction}

%%%Krzys25: Modified the next two paragraphs.

A subset $X$ of a group is called an {\em approximate subgroup} if it is symmetric (i.e. $e \in X$ and $X^{-1}=X$) and $X X \subseteq FX$ for some finite $F \subseteq \langle X \rangle$; if $F$ can be chosen of size $K$, we say that $X$ is a {\em $K$-approximate subgroup}. Approximate subgroups were introduced by Tao in \cite{Tao}. They naturally originate in the classical (in additive number theory and combinatorics) study of finite subsets of groups which have small doubling (or tripling, or $n$-pling) property, where recall that a finite subset $X$ of a group has {\em doubling at most $K$}, if $|XX| \leq K|X|$. It is well known (e.g. see \cite[Proposition 2.3]{Bre1}) that in the abelian context if $X$ has doubling at most $K$, then $X-X$ is a $K^5$-approximate subgroup; a similar result holds for subsets of small tripling in arbitrary groups \cite[Proposition 2.2]{Bre1}. The study of the structure of subsets with small doubling goes back to Freiman's theorem saying that the subsets of $\mathbb{Z}$ with small doubling are ``close'' to generalized arithmetic progressions \cite{Fre}. An important motivation for introducing approximate subgroups also came from connections and applications to expanders, random walks, and spectral gaps \cite{BoGa}, and later to geometric group theory \cite{BGT}. For a historical background on approximate subgroups the reader may consult \cite{Bre2, Toi}. An advantage of Tao's definition is that it is more algebraic than the notion of small doubling, and applies also to infinite subsets. This is needed for example in the theory of approximate lattices (which are certain special approximate subgroups living in locally compact groups) and aperiodic order. The study of approximate lattices goes back to the seminal monograph of Meyer \cite{Mey} from 1972. A short  historical background on approximate lattices is outlined in the introduction to the very recent paper \cite{Mac}.

There is a long list of authors and important papers in the subject, including many applications within and outside mathematics, e.g. to quasicrystals.
In this introduction, we will only mention a few milestone contributions after \cite{Tao} which are relevant for this paper.

%%%%%%%%%%%%%%%%%%%%%%%%%%%%%%%%%%%%%%%%%%%%%%%%%%%%%%%%%%%
\begin{comment}
%Approximate subgroups have became a central topic in additive combinatorics in the last 15 years.
A subset $X$ of a group is called an {\em approximate subgroup} if it is symmetric (i.e. $e \in X$ and $X^{-1}=X$) and $X X \subseteq FX$ for some finite $F \subseteq \langle X \rangle$. Approximate subgroups were introduced by Tao in \cite{Tao} and since then have played a central role in additive combinatorics. However, the study of approximate subgroups in some contexts goes back much earlier to the seminal monograph of Meyer \cite{Mey} from 1972.
%and has many  applications within and outside mathematics, e.g. to quasicrystals. 
There is a long list of authors and important papers in the subject, including many  applications within and outside mathematics, e.g. to quasicrystals. 
%A good historical background is outlined in the introduction to \cite{Mac}; see the introduction in \cite{BGT}.
A good historical background is outlined in the introduction to the very recent paper \cite{Mac}; see also the introduction in \cite{BGT}. Here, we will only mention a few milestone contributions after \cite{Tao} which are relevant for this paper.
\end{comment}
%%%%%%%%%%%%%%%%%%%%%%%%%%%%%%%%%%%%%%%%%%%%%%%%%%%%%%%%%%%%

%%%Krzys: I added "relatively" before compact to extend the context. No, I did not do that.
A symmetric  compact neighborhood of the neutral element in a locally compact group is always an approximate subgroup. Let $X$ be an approximate subgroup and $G := \langle X \rangle$. By a {\em locally compact [resp. Lie] model} of $X$ we mean a group homomorphism $f \colon \langle X \rangle \to H$ for some locally compact [resp. Lie] group $H$ such that $f[X]$ is relatively compact in $H$ and there is a neighborhood $U$ of the neutral element in $H$ with $f^{-1}[U] \subseteq X^m$ for some $m<\omega$. It is easy to show that if $f \colon \langle X \rangle \to H$ is a locally compact model of $X$, then $X$ can be recovered  up to commensurability as the preimage of any compact neighborhood of the identity in $H$.

A breakthrough in the study of the structure of approximate subgroups was obtained by Hrushovski in \cite{Hru}, where a locally compact 
%(and in consequence Lie) 
model for any pseudofinite approximate subgroup (more generally, {\em near-subgroup}) $X$ was obtained by using model-theoretic tools, and in consequence also a Lie model  was found for some approximate subgroup commensurable with $X$ and contained in $X^4$. This paved the way for Breuillard, Green, and Tao to give a full classification of all finite approximate subgroups in \cite{BGT}.

%%%Krzys25: Modified the next paragraph in response to the referee's suggestion.

From the model-theoretic point of view, it is natural to consider {\em definable approximate subgroups} and their {\em definable locally compact models}. This is also useful in applications (e.g. in the proof of the main theorem of \cite{BGT}). These notions are recalled in Definition \ref{definition: definable approximate subgroups}. We would like to emphasize here that in the abstract situation of an arbitrary approximate subgroup $X$, we can always equip the ambient group with the {\em full structure} (i.e. add all subsets of all finite Cartesian powers as predicates), and then $X$ becomes definable and the additional requirement of definability of locally compact models is automatically satisfied. In other words, definable approximate subgroups generalize abstract approximate subgroups.

%%%%%%%%%%%%%%%%%%%%%%%%%%%%%%%%%%%%%%%%%%%%%%%%%%%%%%%%%%%
\begin{comment}
By a {\em definable} (in some structure $M$)  {\em approximate subgroup} we mean an approximate subgroup $X$ of some group such that $X, X^2,X^3,\dots$ are all definable in $M$ and  $\cdot |_{X^n \times X^n} :X^n \times X^n \to X^{2n}$ is definable in $M$ for every positive $n \in \mathbb{N}$. 
%Naming the appropriate parameters, we will be assuming that the definable approximate subgroups are $0$-definable (i.e. without parameters).
If the approximate subgroup $X$ is definable in $M$, then in the definition of a locally compact model one usually additionally requires {\em definability} of $f$ in the sense that for any open $U \subseteq H$ and compact $C \subseteq H$  such that $C \subseteq U$, there exists a definable (in $M$) subset $Y$ of $G$ such that $f^{-1}[C] \subseteq Y \subseteq f^{-1}[U]$. Note that in the abstract situation of an arbitrary approximate subgroup $X$, we can always equip the ambient group with the {\em full structure} (i.e. add all subsets of all finite Cartesian powers as predicates), and then $X$ becomes definable and the additional requirement of definability of locally compact models is automatically satisfied. In other words, definable approximate subgroups generalize abstract approximate subgroups.
\end{comment}
%%%%%%%%%%%%%%%%%%%%%%%%%%%%%%%%%%%%%%%%%%%%%%%%%%%%%%%%%%%%%%

Massicot and Wagner \cite{MaWa} proved the existence of definable locally compact 
%(and in consequence Lie) 
models for all definably amenable definable approximate subgroups, and Wagner conjectured that a locally compact model exists for an arbitrary approximate subgroup. Literally, this conjecture is false; a counter-example can be found for example in \cite[Section 4]{HKP}. However, in another breakthrough paper \cite{Hru2}, Hrushovski weakened the notion of locally compact [and Lie] model of an approximate subgroup $X$ by replacing a homomorphism by a quasi-homomorphism $f \colon \langle X \rangle \to H$ with a compact, normal, symmetric error set $S$ (meaning that $f(y)^{-1}f(x)^{-1}f(xy) \in S$ for all $x,y \in \langle X \rangle$) whose preimage under $f$ is contained in an absolute  (i.e. independent of $X$) power of $X$, and he proved the existence of such {\em generalized definable locally compact models} for arbitrary approximate subgroups (where the notion of definability is also weakened appropriately). This is a structural result on arbitrary approximate subgroups, as each approximate subgroup can be recovered up to commensurability as the preimage of any compact neighborhood of the (compact, symmetric) error set via a generalized locally compact model. This allowed Hrushovski to deduce the existence of suitable generalized Lie models and obtain full classifications of approximate lattices in some contexts, e.g., in $\textrm{SL}_n(\mathbb{R})$ and $\textrm{SL}_n(\mathbb{Q}_p)$. Very recently, Machado wrote an impressive paper \cite{Mac} with a complete structure theorem for approximate lattices in linear algebraic groups over local fields and a uniqueness result for generalized locally compact models with certain extra properties.

The proof in \cite{Hru2} of the existence of generalized definable locally compact (and Lie) models is based on a new theory developed by Hrushovski including {\em definability patterns structures} and {\em local logics}, which is difficult and may be inaccessible to non model theorists.

We prove the existence of generalized definable locally compact models via topological dynamics methods in a model-theoretic context. The main idea is to extend the fundamental theory of Ellis groups to the context of suitable locally compact flows, 
%(working with locally compact left topological semigroups), 
and then the desired generalized definable locally compact model is a certain (explicitly defined) quasi-homomorphism to the canonical Hausdorff quotient of the Ellis group. Our proof is much shorter and uses only standard model theory (e.g. externally definable sets, [external] types, realizations of types). 
%%%Krzys: I replaced the next sentence by a new one.
%In fact, this proof could be rewritten in terms of Boolean algebras avoiding any model theory; however, we find it less natural and more technical,  so we will only comment on how to do that without going into details. 
In fact, if one is not interested in obtaining any definability property of generalized locally compact models, then one can just equip the group $G:=\langle X \rangle$ (generated by the given abstract approximate subgroup $X$) with the full structure and then our proof uses a suitable locally compact subflow of the Stone-\v{C}ech compactification $\beta G$ of $G$, and so there is no need to use externally definable sets and external types.
Our construction of the generalized definable locally compact model is supposed to be fully self-contained. 
%In particular, we will provide all the proofs while developing the theory of Ellis groups for suitable locally compact flows in Section  \ref{section: main theorem}, even of those facts whose proofs are the same as in the classical context of compact flows.
In particular, we will provide almost all the proofs while developing the theory of Ellis groups for suitable locally compact flows in Section  \ref{section: main theorem}; only a few easy proofs that are identical to the proofs in the classical context of compact flows are omitted with precise references to where they can be found.

We also prove universality of our generalized definable locally compact model in a suitable category. As a consequence, we obtain a characterization for a quasi-homomorphism to be a generalized definable locally compact model. While the usual notion of definability of a map from a definable set to a compact space has a characterization coming from continuous logic (namely, a factorization through a suitable space of types), the modified notion of definablity used in generalized definable locally compact models is not so transparent and our characterization explains its nature.

%%%Krzys25: Modified paragraph, as there is a gap in the proof of of Theorem 2.5 in [KP17] very recently found by Andy Zucker. SO we do not known at the moment whether $u\M/H(u\M)$ is  the generalized externally definable Bohr compactification (a least using the definition in terms of cmapctification flows). 
%It is also interesting to consider the special case when the approximate subgroup $X$ in question generates a group $G$ in finitely many steps. Then the target space of our generalized [definable] locally compact model is compact, and it is in fact the classical [resp. externally definable] generalized Bohr compactification of $G$ defined by Glasner (see \cite{Gla} and \cite{KrPi}). This special case can be seen as a structural result on arbitrary definable generic subsets of definable groups. We will discuss it in Section \ref{section: compact case}.
It is also interesting to consider the special case when the approximate subgroup $X$ in question generates a group $G$ in finitely many steps. Then the target space of our generalized locally compact model is compact, and it is in fact the generalized Bohr compactification of $G$ defined by Glasner (see \cite{Gla}); more generally, for definable $X$ and $G$ the target space of our generalized definable locally compact model is a certain canonical compact group associated to $G$ (see \cite{KrPi}). This special case can be seen as a structural result on arbitrary definable generic subsets of definable groups. We will discuss it in Section \ref{section: compact case}.

In Section \ref{section: preliminaries}, we give the necessary preliminaries, including all basic definitions in model theory. Section \ref{section: main theorem} is devoted to our construction of a generalized definable locally compact model of an arbitrary definable approximate subgroup. In Section \ref{section: universality}, we prove universality of our model and discuss related things. In Section \ref{section: compact case}, we focus on the situation when the approximate subgroup in question generates a group in finitely many steps, so in fact the situation of a definable, symmetric, generic subset of a definable group. 
We explain why the main result can be thought of as a structural result on such generic subsets and we use it to obtain more precise structural information on generics in the universal cover of $\SL_2(\mathbb{R})$. 
%%%Krzys: I added the last part of the lest sentence of this paragraph "and also shows ...".
Moreover, our analysis of $\widetilde{\SL_2(\mathbb{R})}$ leads to an answer to some natural question stated at the end of Section \ref{section: universality}, and also shows that the weakening of Newelski's conjecture proposed in Section \ref{section: compact case} holds for $\widetilde{\SL_2(\mathbb{R})}$.

We finish this introduction with a brief history connecting Hrushovski's approach from \cite{Hru2} and our approach via topological dynamics. Topological dynamics methods were introduced to model theory by Newelski in \cite{New}. Since then many papers have appeared in this subject, in particular some connections and applications to model-theoretic components of groups and to strong types were obtained in \cite{KrPi,KrPiRz,KrRz,KrNeSi}. Motivated by this work, Hrushovski developed in \cite{Hru1} a parallel theory of definability patterns structures. Then, in \cite{Hru2}, he redeveloped it in the context of local logics introduced by himself in \cite{Hru2}, and used it to prove the existence of generalized definable locally compact models. In this paper, we return to the topological dynamics approach, but for locally compact flows instead of usual compact flows, and we provide a shorter and simpler proof of Hrushovski's theorem with further information on universality.

\section{Preliminaries}\label{section: preliminaries}

In this section, we recall some basic notions from model theory and topological dynamics to make the main construction self-contained.

\subsection{Model theory}\label{subsection: model theory}

Let us fix a {\em language} (or {\em signature}) $L$, i.e. a collection of relation, function, and constant symbols. Using those symbols together with quantifiers, variables, and logical symbols, one constructs recursively the set of all {\em $L$-formulas}; {\em $L$-sentences} are $L$-formulas without free variables. An {\em $L$-structure} is a set $M$ together with interpretations of all the symbols of $L$. For example, if $L$ consists of just one binary function symbol, then any group is an $L$-structure. Let us fix an arbitrary $L$-structure $M$.

For any $L$-sentence $\varphi$, $M \models \varphi$ means that $\varphi$ is true in $M$. 
For any subset $A$ of $M$ we can expand the language $L$ to $L_A$ be adding constant symbols for the members of $A$, which are then interpreted in $M$ as the corresponding elements of $A$. For an $L_A$-formula $\varphi(x)$, $\varphi(M)$ denotes the {\em set of realizations} of $\varphi(x)$ in $M$, i.e. $\varphi(M):= \{ a \in M^{|x|}: M \models \varphi(a)\}$. By an {\em $A$-definable subset} of $M$ [more generally, of a Cartesian power $M^n$] we mean the set of realizations in $M$ of an $L_A$-formula $\varphi(x)$ with one [resp. $n$] free variables $x$. By a {\em definable subset} we mean an $M$-definable subset. For example, the centralizer of an element of a group is a definable subset of this group. 

An $L$-structure $N$ is an {\em elementary superstructure} of $M$ (symbolically, $M \prec N$) if $M \subseteq N$ and for every $L$-formula $\varphi(x_1,\dots,x_n)$ and tuple $(a_1,\dots,a_n) \in M^n$ we have $M \models \varphi(a_1,\dots,a_n) \iff N \models  \varphi(a_1,\dots,a_n)$. 

By a {\em type} over $A \subseteq M$ in variables $x$ we mean a consistent collection $\pi(x)$ of $L_A$-formulas, where $\pi(x)$ being {\em consistent} means that for any finitely many formulas $\varphi_1(x),\dots,\varphi_n(x) \in \pi(x)$ we have $M \models (\exists x)( \varphi_1(x) \wedge \dots \wedge \varphi_n(x))$. The compactness theorem tells us that this is equivalent to the property that $\pi(x)$ has a realization $a$ in some $N \succ M$, i.e. $N \models \varphi(a)$ for all $\varphi(x) \in \pi(x)$, which will be denoted by $a \models \pi$. A {\em complete type} over $A$ in variables $x$ is a type $p(x)$ over $A$ such that for every $L_A$-formula $\varphi(x)$ we have $\varphi(x) \in p$ or $\neg \varphi(x) \in p$. This is equivalent to saying that $p=\tp(a/A):=\{\varphi(x) \textrm{ an $L_A$-formula}: N\models \varphi(a)\}$ for some tuple $a$ in some $N \succ M$. The set of all complete types over $A$ in variables $x$ is denoted by $S_x(A)$. This is a compact, zero-dimensional topological spaces with a basis of open sets given by the $L_A$-formulas, i.e. any $L_A$-formula $\varphi(x)$ yields a basic open set $[\varphi(x)]:=\{p \in S(A): \varphi(x) \in p\}$. Identifying formulas (modulo equivalence) with the definable sets that they define, complete types over $A$ can be treated as ultrafilters in the Boolean algebra of $A$-definable subsets of $M^{|x|}$, and then  the topology on $S_x(A)$ is just the Stone space topology. We will often omit $x$ in $S_x(A)$. 
%%%Krzys25: Removed sentence, as suggested by the referee.
%The above discussion applies also to any elementary extension of $M$ in place of $M$.

For a given cardinal $\kappa$, we say that $N \succ M$ is {\em $\kappa$-saturated} if for every $B \subseteq N$ of cardinality $<\kappa$, every $p \in S(B)$ has a realization in $N$. Using the compactness theorem, for every $\kappa$ there exists $N \succ M$ which is $\kappa$-saturated. In this paper, we will work with $N \succ M$ which is $|M|^+$-saturated, and since it is very convenient to work with realizations of types from $S(N)$, we will be taking them in an $|N|^+$-saturated $\C \succ N$.  

An {\em externally definable} subset $D$ of $M$ is the intersection of $M$ with a definable subset of $N$ (where $N \succ M$ is $|M|^+$-saturated), that is $D=M \cap \varphi(N)$ for some formula $\varphi(x)$ with parameters from $N$. This definition does not depend on the choice of $N$. 
%%%Krzys: I modified the next sentence according to your suggestion, and moreover I wrote "over $M$" in place of "of $M$". Then I had to change the beginning of the sentence after ";": "all these types".
By a {\em complete external type} over $M$ we mean an ultrafilter on the Boolean algebra of externally definable subsets of $M$; all these types form a Stone space $S_{\ext}(M)$.
It is very convenient to identify $S_{\ext}(M)$ with a space of complete types in the usual sense. In order to do that, take an $|M|^+$-saturated  $N \succ M$. Then $S_{\ext}(M)$ is homeomorphic with the space $S_M(N)$ of all complete types $p \in S(N)$ which are {\em finitely satisfiable} in $M$, i.e. for any $\varphi(x) \in p$, $\varphi(x)$ is realized by some element or tuple of elements of $M$; more precisely,  $S_M(N) \ni p \mapsto \{\varphi(M): \varphi(x) \in p\} \in S_{\ext}(M)$ is a homeomorphism.

%%%Krzys: $S_{\varphi(x)}(N)$ in place of $S_{\varphi(x)}(M)$.
%%%Krzys25: I wrote that $\varphi(x)$ is over $M$ in response to the referee's question.
One can also restrict the context to a given formula $\varphi(x)$ with parameters from $M$ or to the set of realizations $X:=\varphi(M)$. 
%%%Krzys: I found a clash of notation, but I do not want to go through the paper to change any notation, so I added a comment in the footnote.
By $S_{\varphi(x)}(N)$ or $S_X(N)$ we denote the space of complete types $p \in S(N)$ which contain the formula $\varphi(x)$\footnote{There is a clash of notation here, as $S_M(N)$ and $S_X(N)$ have two different meanings when $X=M$. This should not cause any confusion, as the symbol $S_M(N)$ will always denote the space of complete types over $N$ finitely satisfiable in $M$, whereas the symbol $S_X(N)$ (for $X :=\varphi(M)$)  will always denote the space of complete types over $N$ which contain the formula $\varphi(x)$ defining $X$.}; $S_{X,M}(N)$ will stand for the space of complete types over $N$ which contain $\varphi(x)$ and are finitely satisfiable in $M$. Then $S_{X,M}(N)$ is homeomorphic with the space $S_{X,\ext}(M)$ of ultrafilters on the Boolean algebra of  externally definable subsets of $X$. 
%%%Krzys I added $C$ and I added the whole last sentence.
All of it applies also to any superset $C$ of $N$ (contained in $\C$) in place of $N$. In particular, we have the spaces $S_M(C)$ and $S_{X,M}(C)$ homeomorphic with  $S_{\ext}(M)$ and $S_{X,\ext}(M)$, respectively.

%%%Krzys: In the next paragraph, I replaced $A$ by $B$, as I used $A$ to denote a superset of $N$ in the sentence added below Fact 2.1. No, I did not do that.
In this paper, we will need to extend this context to so-called $\bigvee$-definable sets, i.e. unions of possibly infinitely many definable sets. 
%%%Krzys25: Added explanation of upward directed family.
More precisely, let $\{X_i\}_{i \in I}$ be an upward directed family of $A$-definable sets for some $A \subseteq M$ (i.e. $I$ is a directed set with respect to some pre-order $\leq$, and $X_i \subseteq X_j$ whenever $i \leq j$), and let $G := \bigcup_{i \in I} X_i$. Then by $S_{G,M}(N)$ we mean $\bigcup_{i \in I}S_{X_i,M}(N)$ with the topology inherited from $S_{M}(N)$. Since each $S_{X_i,M}(N)$ is clearly an open subset of $S_{M}(N)$, we get that $U \subseteq S_{G,M}(N)$ is open if and only if $U \cap S_{X_i,M}(N)$ is open in $S_{X_i,M}(N)$ for all $i \in I$; so $F \subseteq S_{G,M}(N)$ is closed if and only if $F \cap S_{X_i,M}(N)$ is closed in $S_{X_i,M}(N)$ for all $i \in I$. As each $S_{X_i,M}(N)$ is clearly a clopen subset of $S_{G,M}(N)$ which is a compact (Hausdorff) space, we get

\begin{fact}
$ S_{G,M}(N)$ is a locally compact (Hausdorff) space.
\end{fact}

In this paper, compact and locally compact spaces are Hausdorff by definition.

%%%Krzys25: The next paragraph is extended and made more precise, as externally definable subsets of $G$ (which is not definable) were not defined.
%Note that the space $S_{G,M}(N)$ is homeomorphic with the space $S_{G,\ext}(M)$ of those ultrafilters on the Boolean algebra generated by the externally definable subsets of $G$ which are concentrated on some $X_i$.
Note that the space $S_{G,M}(N)$ is naturally homeomorphic with the space $S_{G,\ext}(M)$ of those ultrafilters on the Boolean algebra of externally definable subsets of $G$ (i.e. subsets which are intersections of $G$ with the sets definable with parameters from $N$) which are concentrated on some $X_i$. It is also homeomorphic with the space of those ultrafilters on the Boolean algebra of subsets of $G$ generated by all externally definable subsets of the $X_i$'s, $i \in I$, which are concentrated on some $X_i$.

%%%Krzys: I added the next sentence.
As before, the above discussion applies also to any superset $C$ of $N$ in place of $N$. In particular, we have the locally compact space $S_{G,M}(C)$ homeomorphic with $S_{G,\ext}(M)$.

%%%Krzys: I added the next sentence, as this notation is used in the proof of part (7) of the proof of Theorem 3.25.
By $S_{G}(M)$ we mean $\bigcup_{i \in I}S_{X_i}(M)$ with the topology inherited from $S(M)$, where $S_{X_i}(M)$ is the space of complete types over $M$ containing a formula defining $X_i$. As above, this is a locally compact space which is witnessed by the clopen compact sets $S_{X_i}(M)$.

If it is not specified, all the parameters and elements are taken from $\C$. Sets of parameters are usually denoted by capital letters, while elements or tuples of elements by lower case letters. For any $a,b,A$, we will write $a \equiv_A b$ to express that $\tp(a/A)=\tp(b/A)$.

%%%Krzys25: "any finite collection of formulas" replaced by "any formula" following the referee.
For any $A \subseteq B$ and $a$ we say that $\tp(a/B)$ is a {\em coheir} over $A$ if it is finitely satisfiable in $A$ (i.e. any formula in $\tp(a/B)$ has a realization in $A$). The following remark will be used many times.

\begin{remark}\label{remark: on coheirs}
If $a \equiv_A b$ and $\tp(c/A,a,b)$ is a coheir over $A$, then $a \equiv_{A,c} b$. 
\end{remark}

\begin{proof}
If not, then there is an $L_A$-formula $\varphi(x,y)$ such that $\C \models \varphi(a,c) \wedge \neg \varphi(b,c)$. Since $\tp(c/A,a,b)$ is a coheir over $A$, there is $c' \in A$ such that $\C \models \varphi(a,c') \wedge \neg \varphi(b,c')$, so $a \nequiv_A b$, a contradiction.
\end{proof}

%%%Krzys25: Below I added "(meaning  that any finite collection of formulas has a common realization in $A$)" as finite satisfiability in $A$ was not defined for partial types.
Note that if a type $\pi(x)$ (over any set of parameters) is finitely satisfiable in $A$ (meaning  that any finite collection of formulas in $\pi(x)$ has a common realization in $A$),
%(i.e. any finite collection of formulas in $\pi(x)$ has a realization in $A$), 
then it extends to a global type $p \in S(\C)$ finitely satisfiable in $A$. For that it is enough to take any ultrafilter $\mathcal{U}$ on the Boolean algebra of all subsets of $A$ such that $\{\varphi(\C) \cap A: \varphi(x) \in \pi\} \subseteq \mathcal{U}$ and to define $p$ as $\{\varphi(x) \in L_{\C}: \varphi(\C) \cap A \in \mathcal{U}\}$. 

\begin{fact}\label{fact: uniqueness of coheirs}
For any type $p \in S_M(N)$ and superset $B$ of $N$ there is a unique extension $\tilde{p} \in S(B)$ of $p$ which is finitely satisfiable in $M$.
\end{fact}

%%%Krzys25: I added the first sentence of the proof, and wrote "Uniqueness" instead of "This" at the beginning of the second sentence.
\begin{proof}
Existence follows from the above paragraph.
%This 
Uniqueness follows from the fact that $\{\varphi(M) : \varphi(x) \in p\}$ is an ultrafilter on the Boolean algebra of externally definable subsets of $M$ (which in turn follows from $|M|^+$-saturation of $N$).
\end{proof}

%%%Krzys25: Definition added.

The main object of interest in this paper will be a definable approximate subgroup.

\begin{definition}\phantomsection\label{definition: definable approximate subgroups}
\begin{enumerate}
\item A {\em definable} (in some structure $M$)  {\em approximate subgroup} is an approximate subgroup $X$ of some group such that $X, X^2,X^3,\dots$ are all definable in $M$ and  $\cdot |_{X^n \times X^n} :X^n \times X^n \to X^{2n}$ is definable in $M$ for every positive $n \in \mathbb{N}$. 
\item If the approximate subgroup $X$ is definable in $M$, then by a {\em definable locally compact model} of $X$ we mean  a locally compact model $f \colon \langle X \rangle \to H$ of $X$ such that for any open $U \subseteq H$ and compact $C \subseteq H$ with $C \subseteq U$, there exists a definable (in $M$) subset $Y$ of $G$ such that $f^{-1}[C] \subseteq Y \subseteq f^{-1}[U]$.
\end{enumerate}
\end{definition}

The model theory context in this paper will be the following: $X$ will be an approximate subgroup definable in a structure $M$, $N \succ M$ an $|M|^+$-saturated elementary extension of $M$, $\C \succ N$ a big (at least $|N|^+$-saturated) elementary extension of $N$ (the so-called {\em monster model}), $G := \langle X \rangle$ --- the group generated by $X$, $\bar X = X(\C)$ --- the interpretation of $X$ in $\C$, $\bar G:=\langle \bar X \rangle$ --- the group generated by $\bar X$.
Thus, we can use the above notation $S_{G,M}(N)$ for the family $\{X_i\}_{i \in I} := \{X^n: n \in \omega\}$.
	
Regarding the monster model $\C$, besides saturation one usually also assumes strong homogeneity with respect to a sufficiently big cardinal. Using the compactness theorem, it is easy to construct $\C \succ N$ which is $|N|^+$-saturated and {\em strongly $|N|^+$-homogeneous} 
%%%Krzys25: Referred to Hodges.
(see \cite[Theorem 10.2.1]{Hod}) which means that for any subset $A\subseteq \C$ of cardinality at most $|N|$, any elementary map $f \colon A \to \C$ (that is $\C \models \varphi(a) \iff \C \models \varphi(f(a))$ for every formula $\varphi(x) \in L$ and finite tuple $a$ from $A$) extends to an automorphism of $\C$. Although the arguments in this paper do not require strong $|N|^+$-homogeneity of $\C$, it is convenient to assume it and use (without even mentioning) the fact that then for every $A$ of cardinality at most $|N|$ and finite tuples $a,b$ we have that $a \equiv_Ab$ if and only if $b=f(a)$ for some $f \in \Aut(\C/A)$ (the pointwise stabilizer of $A$).

\subsection{Topological dynamics}\label{subsection: topological dynamics}

Topological dynamics studies {\em flows}, that is pairs $(G,Y)$ where $Y$ is a compact space and $G$ is a topological group acting continuously on $Y$. We focus on the case when $G$ is discrete; then continuity of the action just means that the action is by homeomorphisms.

In this paper, we will have to extend the context to the case when $Y$ is a certain special locally compact space on which $G$ acts by homeomorphisms, namely $Y:=S_{G,M}(N)$ from the end of the last subsection. We will develop all the necessary theory in this context providing all the details (including proofs) in Section \ref{section: main theorem}. So here we only briefly recall some notions and facts in the classical context of (compact) flows.  They will not be used in the main construction (except Fact \ref{Ellis theorem}), and we give them only to show what is well-known in topological dynamics. This classical context of (compact) flows is however sufficient when the approximate subgroup $X$ generates $G$ in finitely many steps, and this is the context of Section \ref{section: compact case}.

In the rest of this subsection, $(G,X)$ will be an arbitrary flow (so $X$ and $G$ have nothing to do with the approximate subgroup $X$ considered above).
Classical references for Ellis semigroups and groups are \cite{Aus,Gla}. A very good concise exposition with proofs can be found in Appendix A of \cite{Rz}.
%%%Krzys: In the literature (Auslander, Glasner) the tau-topology is considered on the Ellis group of the universal ambit, whereas we use it on Ellis groups of any flows. We already were using it in this wider context in our earlier papers. In his thesis, Rzepecki collected in Appendix A all the necessary facts with full proofs in this wider context. I do not know a better reference. I just do not know alternative references for particular facts on the tau-topology in this wider context. 

	\begin{definition}
		The {\em Ellis semigroup} of the flow $(G,X)$, denoted by $E(X)$, is the closure of the collection of functions $\{\pi_g : g \in G\}$ (where $\pi_g\colon X \to X$ is given by $\pi_g(x):=gx$) in the space $X^X$ equipped with the product topology, with composition as the semigroup operation.
	\end{definition}

%%%Krzys25: We added ``meaning that for any fixed $p$ the function $x \mapsto xp$ is continuous''.
$E(X)$ is a compact left topological semigroup (i.e. the semigroup operation is continuous in the left coordinate, meaning that for any fixed $p$ the function $x \mapsto xp$ is continuous). 
%%%Krzys25: Added sentence.
We will usually notionally identify $\pi_g$ with $g$ for $g \in G$, although the map $g \mapsto \pi_g$ need not be injective.

	The following fundamental fact was proved by Ellis (e.g. see Corollary 2.10 and Propositions 3.5 and 3.6 of \cite{Ell}, or Fact A.8 of \cite{Rz}).
	
	\begin{fact}\label{Ellis theorem}
%		Let $S$ be a semigroup equipped with a compact, Hausdorff topology so that the semigroup operation is continuous in the left coordinate. Then there is a minimal left ideal $\M$ in $S$ (i.e. a minimal set such that $S\M=\M$), and every such $\M$ satisfies the following.
Let $S$ be a semigroup equipped with a quasi-compact $T_1$ topology such that for any $s_0 \in S$ the map $s \mapsto s s_0$ is a continuous and closed mapping (the latter follows immediately from continuity and compactness if $S$ is Hausdorff). 
%%%Krzys25: Added "nonempty".
Then there is a minimal left ideal $\M$ in $S$ (i.e. a minimal nonempty set such that $S\M=\M$), and every such $\M$ satisfies the following.
		\begin{enumerate}[label=\roman*), nosep]
			\item
			For any $p \in \M $, $Sp=\M p=\M $ is closed.
			\item
			$\M $ is the disjoint union of the sets $u\M $ with $u$ ranging over $J(\M ):=\{u \in \M: u^2=u\}$.
			\item
			For each $u \in J(\M )$, $u\M $ is a group with
			identity element $u$, where the group operation is the restriction of the semigroup operation on $S$.
			\item
			All the groups $u\M $ (for $u \in J(\M )$) are isomorphic, even when we vary the minimal left ideal $\M $.
		\end{enumerate}
	\end{fact}

%%%Krzys: I added a comment on the terminology as a footnote. I would not refer here to Newelski, since I think that he often calls this object  "ideal group" rather than "Ellis group". 
	Applying this to $S:=E(X)$, the isomorphism type of the groups $u\M $ (or just any of these groups) from the above fact is called the {\em Ellis group} of the flow $X$.
%%%Krzys25: I modified the footnote by adding the missing assumption that  $vx_0=x_0$. 
\footnote{This terminology is used by model theorists. Let $\textrm{E}$ be the Ellis group  (in our sense) of the universal $G$-ambit $\beta G$ with neutral element $v$. In topological dynamics, the Ellis group of a pointed minimal $G$-flow $(X,x_0)$ with $vx_0=x_0$ is the subgroup of the elements $\eta \in \textrm{E}$ for which $\eta x_0=x_0$, but we will not use this definition.}
%\footnote{This terminology is used by model theorists. In topological dynamics, the Ellis group of a pointed minimal $G$-flow $(X,x_0)$ is the subgroup of those elements $\eta$ in the Ellis group (in our sense) of the universal $G$-ambit $\beta G$ for which $\eta x_0=x_0$, but we will not use this definition in the paper.}

%%%Krzys25: I added  $I$ as the index set, as suggested by the referee.
	\begin{definition}
		For $B \subseteq E(X)$ and $a\in E(X)$, $a \circ B$ is defined as the set of all points $c \in E(X)$ for which there exist nets $(b_i)_{i\in I}$ in $B$ and $(g_i)_{i\in I}$ in $G$ such that $\lim g_i=a$ and $\lim g_ib_i=c$.
	\end{definition}
Basic properties of $\circ$ are contained in Facts A.25-A.29 of \cite{Rz}. In particular, $a \circ B$ is closed.
	
	Now, choose any minimal left ideal $\M $ of $E(X)$ and an idempotent $u \in \M $.
	
	\begin{definition}\label{Def: tau topology}
		For $A \subseteq u\M $, define $\cl_\tau (A) := (u \circ A) \cap u\M$.
	\end{definition}

For the proofs of the facts listed below see Facts A.30-A.40 in \cite{Rz}.
	
\begin{fact}
$\cl_\tau$ is a closure operator on $u\M$. The topology given by $\cl_\tau$ is called the {\em $\tau$-topology}.
\end{fact}

%%%Krzys: Yes, the terminology "semitopological group" is commonly used. You can easily google it.
\begin{fact}
$u\M$ with the $\tau$-topology is a compact $T_1$ semitopological group (i.e. multiplication is separately continuous) which does not depend (up to topological isomorphism) on the choice of $\M$ and $u \in J(\M)$.
\end{fact}

%%%Krzys25: I removed "(greatest)", as it is not precise (and the referee did not understand it). I added ``topological''. I replaced the last sentence by a new one.
\begin{fact}
$H(u\M): = \bigcap_V \cl_\tau(V)$, where $V$ ranges over the $\tau$-neighborhoods of $u$ in $u\M$, is a $\tau$-closed normal subgroup of $u\M$, and $u\M/H(u\M)$ is a compact Hausdorff topological group. 
%In fact, $u\M/H(u\M)$ is the universal Hausdorff quotient of $u\M$.
Moreover, if $F$ is a subgroup of $u\mathcal{M}$ such that $u\mathcal{M}/F$ is a Hausdorff space, then $H(u\M) \leq F$.
\end{fact}

An ambit is a flow $(G,X,x_0)$ with a distinguished point $x_0 \in X$ with dense orbit.
%%%Krzys25: I referred to  \cite[Proposition I.2.6]{Gla} for universality.
An important classical $G$-flow is the universal $G$-ambit $\beta G$ (see \cite[Proposition I.2.6]{Gla}), i.e. the space of ultrafilters on the Boolean algebra of all subsets of $G$ with the action of $G$ by left translation and the distinguished ultrafilter being the principal ultrafilter of the neutral element. Then the Ellis semigroup $E(\beta G)$ is naturally isomorphic to $(\beta G, *)$, where $*$ is given by $U \in p * q \iff \{ g \in G: g^{-1}U \in q\} \in p$
%%%Krzys25: I added  ``(the isomorphism is recalled below in a more general context)''.
(the isomorphism is recalled below in a more general context). Model theory provides a transparent and very useful formula for $*$. Namely, treat  $G$ as a group definable in $M:=G$ equipped with the full structure 
%%%Krzys25: Added ``of all finite Cartesian powers'', as the ``full structure'' was defined as the structure with predicates for all subsets of all finite Cartesian powers.
(i.e. with predicates for all subsets of all finite Cartesian powers of $G$). Then $\beta G=S_{G,\ext}(M)$ is naturally identified with the space of types $S_G(M)$ and it turns out that $p * q=\tp(ab/M)$, where $b \models q$, $a \models p$, and $\tp(a/M,b)$ is the unique extension of $p$ which is a coheir over $M$. More generally, if we have a group $G$ definable in a structure $M$, then $S_{G,\ext}(M)$ is a $G$-ambit with the action of $G$ by left translation and the distinguished element being the ultrafilter of the neutral element. Identifying $S_{G,\ext}(M)$ with $S_{G,M}(N)$ (where $N \succ M$ is $|M|^+$-saturated), it turns out that the Ellis semigroup $E(S_{G,M}(N))$ is isomorphic to $(S_{G,M}(N),*)$ with $*$ given by $p * q := \tp(ab/N)$, where $b \models q$, $a \models p$, and $\tp(a/N,b)$ is the unique extension of $p$ which is a coheir over $M$ (an isomorphism $(S_{G,M}(N),*) \to E(S_{G,M}(N))$ is given by $p \mapsto l_p$, where $l_p(q):=p*q$). 
%%%Krzys25: I added the next sentence.
For the details see \cite[Section 4]{New}.

%$M$ and an $|M|^+$-saturated extension $N \succ M$, then $G$ acts on $S_{G,M}(N)$ be left translations (where $S_{G,M}(N)$ is the space of all complete types over $N$ concentrated on $G$ and finitely satisfiable in $M$, which is naturally identified with the space $S_{G,\ext}(M)$ of ultrafilters in the Boolean algebra of externally definable subset of $G$) and the Ellis semigroup $E(S_{G,M}(N))$ is isomorphic to $(S_{G,M}(N),*)$ with $*$ is given by $p * q := \tp(ab/N)$, where $b \models q$, $a \models p$ and $\tp(a/N,b)$ is a coheir over $M$.

\section{Generalized definable locally compact model}\label{section: main theorem}

This section is devoted to a new self-contained construction of a generalized definable locally compact model of an arbitrary definable approximate subgroup.
Let us start from the context and precise definition of generalized definable locally compact models.

For a map $f \colon G \to H$ from a group (or even semigroup) $G$ to a group $H$, $\error_r(f) := \{ f(y)^{-1}f(x)^{-1}f(xy): x,y \in G\}$ and $\error_l(f) := \{ f(xy)f(y)^{-1}f(x)^{-1}: x,y \in G\}$. For $C \subseteq H$, we write $f \colon G \to H : C$ if $\error_r(f) \cup \error_l(f) \subseteq C$ and we say that $f$ is a {\em quasi-homomorphism with an error set $C$}. Note that if $C$ is normal in $H$, 
%%%Krzys25: I added ``i.e. closed under conjugation by the elements of $H$''.
i.e. closed under conjugation by the elements of $H$ (which will be the case in our context), then  $\error_r(f) \subseteq C$ if and only if $\error_l(f) \subseteq C$. 
%Also, if  $f \colon G \to H : C$, then $f(e_G) \in C$ and $f(x^{-1}) \in f(x)CC^{-1}$. Sometimes one assumes that $f(e_G)=e_H$, and this will be satisfied in our construction.
%%%Krzys25: I added `` for all $x \in G$''.
Also, if  $f \colon G \to H : C$, then $f(e_G) \in C^{-1}$ and $f(x^{-1}) \in f(x)^{-1}C^{-2}$ for all $x \in G$. Sometimes one assumes that $f(e_G)=e_H$, and this will be satisfied in our construction. It is clear that if $f \colon G \to H : C$, then for any subsets $X,Y \subseteq G$, we have $f[XY]\subseteq f[X]f[Y]C$.

From now on, take the situation and notation described at the end of Subsection \ref{subsection: model theory}.

\begin{definition}\label{definition: gen. loc. comp. model}
A {\em generalized definable locally compact model of $X$} is a quasi-homomorphism $f \colon G \to H:C$ for some symmetric, normal, compact subset $C$ of a locally compact group $H$ such that:

\begin{enumerate}
\item for every compact $V \subseteq H$ there is $i \in \mathbb{N}$ with $f^{-1}[V] \subseteq X^i$;
%%%Krzys25: I added ``(i.e., with compact closure)''.
\item for every $i \in \mathbb{N}$, $f[X^i]$ is relatively compact (i.e., with compact closure) in $H$;
\item there is $l \in \mathbb{N}$ such that
%\begin{enumerate}
for any compact $Z,Y \subseteq H$ with $C^l Y\cap C^l Z = \emptyset$ the preimages $f^{-1}[Y]$ and $f^{-1}[Z]$ can be separated by a definable set, 
%%%Krzys25: I added ``meaning that there is a definable set $D$ such that $f^{-1}[Y] \subseteq D$ and $f^{-1}[Z] \cap D = \emptyset$''.
meaning that there is a definable set $D$ such that $f^{-1}[Y] \subseteq D$ and $f^{-1}[Z] \cap D = \emptyset$.
%\item for any compact $Z,Y \subseteq H$ with $Y C^l \cap Z C^l  = \emptyset$ the preimages $f^{-1}[Y]$ and $f^{-1}[Z]$ can be separated by a definable set.
%\end{enumerate}
\end{enumerate}
If we drop item (3), we get the notion of {\em generalized locally compact model}.
\end{definition}

%%%Krzys25: Reformulated the next remark in order to emphasize that we do not use items (1) and (3) in this remark.
\begin{remark}\label{remark: i=1 is enough}
Item (2) of the above definition is equivalent to the property that $f[X]$ is relatively compact.
%In item (2) of the above definition, it is equivalent to require that $f[X]$ is relatively compact.
\end{remark}

\begin{proof}
We have $f[X^2] \subseteq f[X]^2C$, so, by compactness of $\cl(f[X])$ and $C$, we get $\cl(f[X^2]) \subseteq \cl(f[X])^2C$. More generally, by induction, $\cl(f[X^i]) \subseteq \cl(f[X])^iC^{i-1}$ for all $i \geq 1$, and  since the last set is compact, so is $\cl(f[X^i])$.
\end{proof}

\begin{remark}\label{remark: separation by two sets}
If $f \colon G \to H:C$ is a generalized definable locally compact model of $X$, then there is $l \in \mathbb{N}$ such that for any compact $Z,Y \subseteq H$ with $C^l Y\cap C^l Z = \emptyset$ there are disjoint definable subsets $D_1$ and $D_2$ of some $X^n$ with $f^{-1}[Y]\subseteq D_1$ and $f^{-1}[Z]\subseteq D_2$.
\end{remark}

\begin{proof}
It follows from items (1) and (3) of the definition.
\end{proof}

\begin{fact}\label{fact: preimages of neighborhoods} 
Let $f \colon G \to H:C$ be a generalized locally compact model of $X$.
\begin{enumerate}
\item For every neighborhood $U$ of $e_H$, $f^{-1}[UC]$ is generic in the sense that finitely many left translates of  $f^{-1}[UC]$ cover $X$.
\item For every relatively compact neighborhood $U$ of $e_H$, $Y:=f^{-1}[UC]$ is commensurable with $X$ and $YY^{-1}$ is an approximate subgroup commensurable with $X$. 
\end{enumerate}
\end{fact}

\begin{proof}
(1) Take an open neighborhood $W$ of $e_H$ such at $W^{-1}W \subseteq U$. By compactness of $\cl(f[X])$, we have that $\cl(f[X])$ is covered by finitely many translates $a_1W$, \dots, $a_nW$. 

For every $i \leq n$ with $f^{-1}[a_iW] \ne \emptyset$ choose $g_i \in f^{-1}[a_iW]$. 
%for other $i \leq n$ choose $g_i \in G$ arbitrarily. 
We will show that $X$ is covered by the finitely many translates $g_i f^{-1}[U C]$ for $i \leq n$ such that $f^{-1}[a_iW] \ne \emptyset$.

Consider any $g \in X$; then $g \in f^{-1}[a_iW]$ for some $i \leq n$, i.e. $f(g) \in a_i W$.  Write $g$ as $g_ih$. Then $f(g) =f(g_i)f(h) f(h)^{-1}f(g_i)^{-1}f(g_ih) \in a_i W f(h)C$. Hence, the last set has a nonempty intersection with $a_iW$. 
%%%Krzys25: I added `` as $C$ is symmetric''.
So, as $C$ is symmetric,  $f(h) \in W^{-1}W C \subseteq U C$. Therefore, $h \in f^{-1}[UC]$, and so $g \in g_i  f^{-1}[UC]$.

(2) The fact that finitely many left translates of $f^{-1}[UC]$ cover $X$ follows from (1). The fact that finitely many left translates of $X$ cover $f^{-1}[UC]$ follows from item (1) of Definition \ref{definition: gen. loc. comp. model} and the assumption that $X$ is an approximate subgroup (which clearly implies that $X^i$ is covered by finitely many left translates of $X$ for every every $i \in \mathbb{N}$). 
%%%Krzys25: We added a justification that $Y \subseteq Y^{-1}Y$.
The very final part about $YY^{-1}$ easily follows, as $Y \subseteq YY^{-1} \subseteq X^i$ for some $i$ (the inclusion  $Y \subseteq YY^{-1}$ holds, because $e_G \in Y$, which in turn is true as $f(e_G) \in C^{-1}=C \subseteq UC$).
%The very final part about $YY^{-1}$ easily follows under the assumption that $f(e_G)=e_H$, as then $Y \subseteq Y^{-1}Y \subseteq X^i$ for some $i$. In general, we still have that  $Y^{-1}Y \subseteq X^i$ for some $i$, hence $Y^{-1}Y$ is covered by finitely many left translates of $X$ and so also of $Y$ as $X$ and $Y$ are commensurable. Similarly, since $(Y^{-1}Y)^2 \subseteq X^{2i}$,  we get that $(Y^{-1}Y)^2$ is covered by finitely many left translates of $Y$ and so also of $Y^{-1}Y$. Thus, $Y^{-1}Y$ is an approximate subgroup. Since finitely many left translates of $Y$ cover $X$, also finitely many left translates of $Y^{-1}Y$ cover $X$. So we see that $Y^{-1}Y$ is commensurable with $X$.
\end{proof}

Thus, as mentioned in the introduction, a generalized locally compact model $f \colon G \to H:C$ of $X$ allows us to recover $X$ up to commensurability as the preimage of any compact neighborhood of $C$. 
%%%Krzys25: Added justification, as requested by the referee.
Indeed, if $V$ is a neighborhood of $C$, then, by compactness of $C$ (and local compactness of $H$), there is a compact neighborhood $U$ of $e_H$ such that $UC \subseteq V$. On the other hand, $V \subseteq VC$. Assuming that $V$ is compact, by Fact \ref{fact: preimages of neighborhoods}(2), both $f^{-1}[UC]$ and $f^{-1}[VC]$ are commensurable with $X$, and hence so is $f^{-1}[V]$.

\subsection{Topological dynamics of $\pmb{S_{G,M}(N)}$}\label{subsection: top dyn for S(N)}

%%%Krzys: I added the next sentence. And in the whole paragraph I replaced $A$ by $C$.
Recall that for $N \subseteq C \subseteq \C$ by $S_M(C)$ we denote the space of complete types over $C$ which are finitely satisfiable in $M$, and by $S_{G,M}(C)$ the subspace of $S_M(C)$ consisting of all types concentrated on some $X^n$. 
For a formula $\varphi(x)$ in $L_C$ (with $N \subseteq C \subseteq \C$) such that $\varphi(\C) \subseteq \bar{X}^n$ for some $n \in \mathbb{N}$, we have that $[\varphi(x)]:= \{p \in S_M(C): \varphi(x) \in p\} \subseteq S_{X^n,M}(C)$ is a basic open set in $S_{G,M}(C)$. For any $g \in \bar{G}$, by $\varphi(g^{-1}x)$ [resp. $\varphi(xg^{-1})$] we mean an $L_{C,g}$-formula defining the set $g\varphi(\C)$ [resp. $\varphi(\C)g$]. 
%%%Krzys25: Added explicit formula as requested by the referee.
Note that by the definability of the approximate subgroup $X$, it is clear that the sets $g\varphi(\C)$ and $\varphi(\C)g$ are indeed definable over $C,g$. More precisely, explicitly  $\varphi(g^{-1}x)$ can be chosen to be $(\exists y)(\varphi(y) \wedge \psi(g,y,x))$, where $\psi(t,y,x)$ defines the graph of the group operation on $G$ restricted to $X^n$ for some $n$ such that $\varphi(\C) \subseteq \bar X^n$ and $g \in \bar X^n$. 

%%%Krzys25: I added ``(note that this is a well-defined action)'', as requested by the referee.
The goal of this subsection is to extend the classical theory briefly mentioned in Subsection \ref{subsection: topological dynamics} to the action of $G$ on the locally compact space $S_{G,M}(N)$ by left translation, that is $g \tp(a/N):=\tp(ga/N)$ (note that this is a well-defined action).
%more precisely, identifying $S_{G,M}(N)$ with $S_{G,\ext}(M)$ the action of $G$ is given by left translations.
First of all, this action is by homeomorphisms, because a basis of open sets in $S_{G,M}(N)$ consists of the sets of the form $[\varphi(x)]$ for formulas $\varphi(x)$ in $L_N$ with $\varphi(\C) \subseteq  \bar{X}^n$ for some $n$, and $g[\varphi(x)] =[\varphi(g^{-1}x)]$ is still a basic open set for any $g \in G$.

Define a binary operation $*$ on $S_{G,M}(N)$ by
$$p*q:= \tp(ab/N), \textrm{ where $b \models q$, $a \models p$, and $\tp(a/N,b)$ is a coheir over $M$}.$$

\begin{lemma}\label{lemma: semigroup operation}
 $(S_{G,M}(N),*)$ is a left topological semigroup, that is, $*$ is well-defined, associative, and left continuous.
\end{lemma}

\begin{proof}
%%%Krzys25: I added the first two sentences in the first paragraph of the proof. This is in response to one of the referee's suggestions.
The existence of a pair $(a,b)$ as in the definition of $*$ follows from Fact \ref{fact: uniqueness of coheirs}. The fact that $\tp(ab/N)$ is a coheir over  $M$ follows from the fact that the types $\tp(a/N,b)$ and $\tp(b/N)$ are coheirs over $M$. 
Now, take pairs $(a,b)$ and $(a',b')$ both as in the definition of $*$. Thus, $b' \equiv_N b$, so, by $|N|^+$-saturation of $\C$, we can find $a''$ such that $(a'',b) \equiv_N (a',b')$. Then $\tp(a''/N,b)$ is an extension of $p$ which is a coheir over $M$. Therefore, $\tp(a''/N,b)=\tp(a/N,b)$ by Fact \ref{fact: uniqueness of coheirs}. Hence, $\tp(ab/N) = \tp(a''b/N)=\tp(a'b'/N)$. We have proved that $*$ is well-defined.

To check that $*$ is associative, consider any $p,q,r \in S_{G,M}(N)$ and pick $a \models p$, $b \models q$, and $c \models r$ such that both $\tp(b/N,c)$ and $\tp(a/N,b,c)$ are coheirs over $M$. Then $\tp(a/N,bc)$ is a coheir over $M$, so $abc \models p*(q*r)$. On the other hand, $\tp(a/N,b)$ and $\tp(ab/N,c)$ are both coheirs over $M$, so $abc \models (p*q)*r$. Thus, $p*(q*r) = (p*q)*r$.

It remains to show left continuity of $*$. Fix $q \in S_{G,M}(N)$ and pick $b \models q$. Then $b \in \bar{X}^m$  for some $m$. 
%Consider any open subset $U$ of $S_{G,M}(N)$. The goal is to show that $V:=\{ p\in S_{G,M}(N): p*q \in U\}$ is open. Since $U
Consider any basic open set $U =[\varphi(x)] \subseteq S_{X^n,M}(N)$ for some $n$. The goal is to show that $V:=\{ p\in S_{G,M}(N): p*q \in U\}$ is open. 
%%%Krzys25: Added ``(see Remark \ref{remark: key property of *})''.
It is clear that $V \subseteq S_{X^{n+m},M}(N)$ (see Remark \ref{remark: key property of *}). By Fact \ref{fact: uniqueness of coheirs}, the restriction map $r \colon S_{X^{n+m},M}(N,b) \to S_{X^{n+m},M}(N)$ is a homeomorphism. So it is enough to show that $r^{-1}[V]$ is open.

For any $a$ such that $\tp(a/N,b)$ is a coheir over $M$ we have
$$\tp(a/N,b) \in r^{-1}[V] \iff \tp(ab/N) \in U \iff \C \models \varphi(ab).$$
Therefore, $r^{-1}[V]= [\varphi(xb)]$ is a basic open set in $S_{X^{n+m},M}(N,b)$.
\end{proof}

%%%Krzys25: Added "as a group" as requested by the referee. I also added the second sentence.
Note that $G$ naturally embeds as a group into $S_{G,M}(N)$ via $g \mapsto \tp(g/N)$, which we will be using without mentioning. Also, $gp=\tp(g/N)*p$ for all $g \in G$ and $p \in S_{G,M}(N)$.

\begin{remark}\label{remark: density of G in the space of types}
For every $n$ the set $X^n$ is dense in $S_{X^n,M}(N)$, and $G$ is dense in $S_{G,M}(N)$.
\end{remark}

\begin{proof}
The second part follows from the first. The first part is clear, as for any nonempty basic open set $[\varphi(x)]$ in $S_{X^n,M}(N)$ there is $a \in \varphi(M) \subseteq X^n$.
\end{proof}

For $p \in S_{G,M}(N)$ let $l_p \colon S_{G,M}(N) \to S_{G,M}(N)$ be defined by $l_p(q):=p*q$.
%%%Krzys25: I changed the next sentence, as now the proof of 3.7 is included. The referee really wanted a proof.
%Since the next fact will not be used in the rest of the construction, we leave a proof as an exercise.
Although the next proposition  will not be used in the rest of the construction, we include it for completeness of the whole picture.

%%%Krzys25: I added ``$l$ given by'' and I included a proof.
\begin{proposition} 
The assignment $l$ given by $p \mapsto l_p$ yields an isomorphism between $S_{G,M}(N)$ and the Ellis semigroup $E(S_{G,M}(N))$ defined in the same way as for (compact) flows in Subsection \ref{subsection: topological dynamics}.
\end{proposition}

%%%Krzys25: The proof is added.
\begin{proof}
The fact that for every $p \in S_{G,M}(N)$ the function $l_p$ belongs to $E(S_{G,M}(N))$ follows from left continuity of $*$ and the observation that $p = \lim g_i$ for some net $(g_i)$ of elements of $G$ which holds by Remark \ref{remark: density of G in the space of types}. The fact that $l$  is a homomorphism follows from associativity of $*$: $l_{p*q}(r)=(p*q)*r=p*(q*r)=l_p(q*r)=(l_p \circ l_q)(r)$.

To show that $l$ is onto, consider any $\eta \in E(S_{G,M}(N))$. Then $\eta = \lim \pi_{g_i}$ for some net $(g_i)$ of elements of $G$ (where $\pi_{g_i}(p):=g_ip$ for $p \in S_{G,M}(N)$). Set $p:=\eta(\tp(e/N))=\lim \pi_{g_i}(\tp(e/N)) = \lim \tp(g_i/N)$. By left continuity of $*$, for every $q \in S_{G,M}(N)$ we have $p*q=\lim\tp(g_i/N)*q = \lim \pi_{g_i}(q)=\eta(q)$. So $\eta = l_p$.

Injectivity of $l$ is clear, because $l_p(\tp(e/N))=p*\tp(e/N)=p$. Continuity of $l$ follows trivially from left continuity of $*$.

It remains to show that $l$ is an open map. For that it is enough to show that if $U$ is an open subset of some $S_{X^n,M}(N)$, then $l[U]$ is open in $E(S_{G,M}(N))$. Since $l$ is a bijection and $l_p(\tp(e/N))=p$, we see that for any $\eta \in E(S_{G,M}(N))$ the condition $\eta \in l[U]$ is equivalent to $\eta(\tp(e/N)) \in U$. Thus, $l[U]$ is open.
\end{proof}

The following property of the semigroup operation $*$, which follows immediately from the definition of $*$ and the assumption that $X$ is symmetric,  will play an essential role in the rest of the construction.

\begin{remark}\label{remark: key property of *}
Whenever $q \in S_{X^n,M}(N)$, $r \in S_{X^m,M}(N)$, and $p*q=r$, then $p \in S_{X^{n+m},M}(N)$.
\end{remark}

\begin{lemma}\label{lemma: preparation of existence of minimal ideals}
There exists a left ideal $\M$ of $S_{G,M}(N)$ for which the set $\M \cap S_{X,M}(N)$ is minimal (nonempty).
\end{lemma}

\begin{proof}
By compactness of $S_{X,M}(N)$ and Zorn's lemma, it is enough to show that for every $s \in S_{X,M}(N)$ the set $(S_{G,M}(N)*s) \cap S_{X,M}(N)$ is closed. By  Remark \ref{remark: key property of *}, $(S_{G,M}(N) *s) \cap S_{X,M}(N) = (S_{X^2,M}(N) *s) \cap S_{X,M}(N)$.

Since $r_s \colon  S_{X^2,M}(N) \to S_{X^3,M}(N)$ given by $p \mapsto p*s$ is a continuous map between compact Hausdorff spaces, we get that $S_{X^2,M}(N) *s= r_s[S_{X^2,M}(N)]$ is closed, and so is $(S_{X^2,M}(N) *s) \cap S_{X,M}(N)$.
\end{proof}

\begin{proposition}\label{proposition: existence of minimal ideal}
There exists a minimal left ideal in $S_{G,M}(N)$.
\end{proposition}

\begin{proof}
%%%Krzys: I changed the next sentence as you suggested.
%Take a left ideal $\M$ satisfying the conclusion of Lemma \ref{lemma: preparation of existence of minimal ideals} which is of the form $S_{G,M}(N) * s_0$ for some $s_0 \in S_{X,M}(N)$. 
We can clearly find a left ideal $\M$ as in the conclusion of Lemma \ref{lemma: preparation of existence of minimal ideals} which is of the form $S_{G,M}(N) * s_0$ for some $s_0 \in S_{X,M}(N)$. 
We will show that it is minimal. For that take any $s \in \M$. It is enough to show that $(S_{G,M}(N) *s) \cap S_{X,M}(N) \ne \emptyset$ (as then $s_0 \in (S_{G,M}(N) *s) \cap S_{X,M}(N)$ by the choice of $\M$).

We have that $s \in S_{X^n,M}(N)$ for some $n$; then $s=\tp(b/N)$ for some $b \in \bar{X}^n$.

\begin{clm}
$\bar{X} \cdot b^{-1} \cap G \ne \emptyset$.
\end{clm}

\begin{clmproof}
Since $X$ is an approximate subgroup, $X^n \subseteq Xg_1 \cup \dots \cup Xg_n$ for some $g_1,\dots,g_k \in G$. Hence, $\bar X^n \subseteq \bar X g_1 \cup \dots \cup \bar X g_n$, i.e. $\bar X^n \subseteq \bar X G$. Since $X$ is symmetric, we get that $(\bar X^n)^{-1} \subseteq \bar X^{-1}G$, so $b^{-1} \in \bar X^{-1}G$, that is $\bar X b^{-1} \cap G \ne \emptyset$.
\end{clmproof}

By this claim, $\bar{X} \cdot b^{-1} \cap G$ extends to an ultrafilter 
%%%Krzys25: I removed ``generated by'', because externally definable subsets of $G$ defined after Fact 2.1 as the intersections of $G$ with the sets definable with parameters from $N$ already form a Boolean algebra.
%on the Boolean algebra generated by externally definable subsets of $G$ which is concentrated on $X^{n+1}$. 
on the Boolean algebra of externally definable subsets of $G$ which is concentrated on $X^{n+1}$. This ultrafilter corresponds to a unique $\tp(a/N,b)$ finitely satisfiable in $M$. Then $\tp(a/N)*\tp(b/N) = \tp(ab/N) \in S_{X,M}(N)$, so $(S_{G,M}(N) *s) \cap S_{X,M}(N) \ne \emptyset$.
\end{proof}

\begin{lemma}
Any minimal left ideal of $S_{G,M}(N)$ is closed and intersects $S_{X,M}(N)$.
\end{lemma}

\begin{proof}
Let $\M$ be a minimal left ideal of $S_{G,M}(N)$.
%%%Krzys: I changed the next sentence which hopefully clarifies the argument.
%The fact that it intersects $S_{X,M}(N)$ follows from the proof of Proposition \ref{proposition: existence of minimal ideal}. 
The proof of Proposition \ref{proposition: existence of minimal ideal} shows that any left ideal (in particular $\M$) of $S_{G,M}(N)$ intersects $S_{X,M}(N)$.
To show closedness of $\M$, first note that $\M = S_{G,M}(N) * s$ for some $s \in S_{G,M}(N)$. Of course, $s \in S_{X^n,M}(N)$ for some $n$. By Remark \ref{remark: key property of *}, for every $m \in \mathbb{N}$,  $(S_{G,M}(N)*s) \cap  S_{X^m,M}(N) = (S_{X^{n+m},M}(N) *s) \cap S_{X^m,M}(N)$, and the last set is closed by compactness of $S_{X^{n+m},M}(N)$ and left continuity of $*$.
\end{proof}

From now on, we will often skip writing $*$.

\begin{lemma}
Let $\M$ be an arbitrary minimal left ideal of $S_{G,M}(N)$. Then $J(\M):=\{u \in \M: u^2=u\}$ is nonempty and $\M$ is the union of all $u\M$ with $u$ ranging over $J(\M)$.
\end{lemma}

\begin{proof}
Consider any $p \in \M$. Then $p \in S_{X^n,M}(N)$ for some $n$. By minimality of $\M$, the set $P:=\{ q \in \M: qp=p\}$ is nonempty. 
Thus, by left continuity of $*$ and Remark \ref{remark: key property of *}, $P$ is a nonempty closed subsemigroup of $\M$ contained in $S_{X^{2n},M}(M)$, so it is compact. 
By Zorn's lemma, there exists a minimal closed subsemigroup $K$ of $P$.

Consider any $u \in K$. We will show that $u^2=u$. Then, since $u \in P$, we get $p=up=u(up) \in u\M$, so we will be done.

Let $Q:=\{q \in K: qu =u\}$. By compactness of $K$ and left continuity of $*$, $Ku$ is a (nonempty) closed subsemigroup of $K$, so $Ku=K$ as $K$ is minimal. Hence, $Q \ne \emptyset$. Since $Q$ is a closed subsemigroup of $K$, we get that $Q=K$, in particular $u \in Q$. 
\end{proof}

%%%Krzys25: In fact, the proofs of all the three lemmas below are identical to the classical case (in the third lemma we only have to use 3.8 to say that some set is compact), so I slightly changed the next sentence. In the second sentence "We only prove'' replaced by "We only recall the proof of''.
%The proofs of the next two  lemmas are identical to the proofs in the classical context, and the proof of the third lemma below is an easy elaboration on the proof in the classical context. 
The proofs of the next three  lemmas are essentially identical to the proofs in the classical context.
%We will only prove the first one. 
We only recall the proof of the first one, as the other two are not needed in this work. For the proofs in the classical context see \cite[Fact A.8]{Rz}.

\begin{lemma}
For any minimal left ideal $\M$ of $S_{G,M}(N)$ and $u \in J(\M)$, the set $u\M$ is a group (with $*$ as group operation).
\end{lemma}

\begin{proof}
$u\M$ is clearly closed under $*$, $u \in u\M$ is a neutral element in $u\M$, and $*$ is associative. Now, consider any $p \in u\M$. By minimality of $\M$, there is $q \in \M$ with $qp =u$. Then $(uq)p=u^2=u$. Thus, $u\M$ is a semigroup with left identity and left inverses, and so it is a group.
\end{proof}

\begin{lemma}
For every minimal left ideal $\M$ of $S_{G,M}(N)$ and any distinct $u,v \in J(\M)$, $u\M \cap v\M = \emptyset$.
\end{lemma}

\begin{lemma}
For any minimal left ideals $\M,\mathcal{N}$ of $S_{G,M}(N)$ and $u \in J(\M), v\in J(\mathcal{N})$ the groups $u\M$ and $v\mathcal{N}$ are isomorphic.
\end{lemma}

Therefore, the isomorphism type of all these groups $u\M$ (or just any of these groups separately) can be called the {\em Ellis group} of $S_{G,M}(N)$.

Now, the goal is to equip the Ellis group with a topology, which will be called the $\tau$-topology. We will do it in the same way as in the classical context. Below, for $P \subseteq S_{G,M}(N)$ the closure of $Q$ will be denoted by $\overline{Q}$, while for a subset $Q$ of the Ellis group the closure with respect to the $\tau$-topology will be denoted by $\cl_{\tau}(Q)$. 

\begin{definition}\label{definition: circle operation}
For any $p \in S_{G,M}(N)$ and $Q \subseteq S_{G,M}(N)$ we define $p \circ Q$ as the set of all $r \in S_{G,M}(N)$ for which there are nets $(g_i)_{i \in I}$ in $G$ and $(q_i)_{i \in I}$ in $Q$ such that $\lim_i g_i =p$ and $\lim_ig_iq_i = r$.
\end{definition}

All the easy observations A.25 -- A.35 from \cite{Rz} work with exactly the same proofs for $S_{G,M}(N)$ in place of the Ellis semigroup of a compact flow.  
%%%Krzys25: Added sentences, remark, and proof.
The only place where a slight elaboration is needed is the proof of \cite[A.25(1)]{Rz}, because it uses compactness to pass to a convergent subnet. So let us prove it in our context.

\begin{remark}
For any $Q \subseteq S_{G,M}(N)$ and $p,r \in S_{G,M}(N)$ we have $(p\circ Q)r=p \circ (Qr)$.
\end{remark}

\begin{proof}
The inclusion ($\subseteq$) follows from left continuity of $*$. Namely, consider any $s \in p\circ Q$. Then there are nets $(g_i)_i$ in $G$ and $(q_i)_i$ in $Q$ such that $\lim_i g_i=p$ and $\lim_i g_iq_i = s$. By left continuity of $*$, we get $sr=(\lim_i  g_iq_i)r=\lim_i g_i(q_ir) \in p \circ (Qr)$. 

For the opposite inclusion consider any $s \in p \circ (Qr)$. Then there are nets $(g_i)_i$ in $G$ and $(q_ir)_i$ in $Qr$ (where $q_i \in Q$) such that $\lim_i g_i =p$ and $\lim g_iq_ir = s$.  There are $m,n \in \mathbb{N}$ such that $r \in S_{X^n,M}(N)$ and $s \in S_{X^m,M}(N)$. By Remark \ref{remark: key property of *}, passing to suitable final segments of the nets in question, we can assume that $g_iq_i \in S_{X^{n+m},M}(N)$ for all $i$. Since  $S_{X^{n+m},M}(N)$ is compact, passing to a subnet, we can assume that $\lim_i g_iq_i$ exists. Then $\lim_i g_iq_i \in p \circ Q$ and, by left continuity of $*$, we get $s = \lim_i g_iq_ir = (\lim_i g_iq_i)r \in (p \circ Q)r$.
\end{proof}

%%%Krzys25: Instead of ``In particular, we have'' I wrote ``In particular, the counterparts of Facts A.30 and A.32 from \cite{Rz} yield''.
In particular, the counterparts of Facts A.30 and A.32 from \cite{Rz} yield
\begin{lemma}
		\label{fct:tau_closure}
		\index{clt@$\cl_{\tau}$}
		Given a minimal left ideal $\M\unlhd S_{G,M}(N)$ and idempotent $u\in \M$, the operator $\cl_\tau$ on subsets of $u\M$ given by $\cl_\tau(Q):=(u\M)\cap (u\circ Q)=u(u\circ Q)$ is a closure operator on $u\M$.
	\end{lemma}

Now, fix a minimal left ideal $\M$ of $S_{G,M}(N)$ and $ u \in J(\M)$.

\begin{definition}
		\label{dfn:tau_topology}
		\index{topology!t@$\tau$}
		By the \emph{$\tau$-topology} we mean the topology on the Ellis group $u\M$ given by the closure operator $\cl_\tau$ from Lemma~\ref{fct:tau_closure}.
	\end{definition}

Fact A.33 of \cite{Rz} tells us that the $\tau$-topology on $u\M$ is coarser than the subspace topology inherited from $S_{G,M}(N)$.
The next lemma (see Fact A.35 of \cite{Rz}) yields an important connection between limits in both these topologies.
%the topology inherited from $S_{G,M}(N)$ and in the $\tau$-topology.

	\begin{lemma}
		\label{fct:ulimit}
		If $(a_i)_i$ is a net in $u\M$ converging to $a\in \overline{u\M}$, then $(a_i)_i$ converges to $ua$ in the $\tau$-topology.
	\end{lemma}

%%%Krzys25: New sentence and four lemmas. The first lemma is added, as the referee in several places had a problem with this argument and was suggesting some alternative arguments.

The next four technical lemmas will be very useful throughout the paper.

\begin{lemma}\label{lemma: tau-closure in terms of realizations of types}
Let $Q \subseteq u\M$ be contained in a closed subset of $S_{G,M}(N)$ of the form $[\pi(x)]:=\{p \in S_{G,M}(N): \pi(x) \subseteq p\}$, where $\pi(x)$ is a (partial) type with parameters from $N$. Then, each element $q \in \cl_{\tau}(Q)$ can be written as $\tp(ab/N)$ for some $a,b \in \bar G$ with $a \models u$ and $b \models \pi(x)$.
\end{lemma}

\begin{proof}
%Consider any $q \in \cl_\tau(V)$. Then $q=\lim_i g_iq_i$ for some nets $(g_i)_i$ in $G$ and $(q_i)_i$ in $Q$ with $\lim_ig_i=u$. Consider any formulas $\varphi(x) \in u$, $\psi(x)$ implied by $\pi(x)$, and $\theta (x) \in q$, all with parameters from $N$. Then there is an index $i$ such that $g_i \models \varphi(x)$ and $\theta(x) \in g_iq_i$. Take any $b_i \models q_i$. Then $\C \models \varphi(g_i) \wedge \psi(b_i) \wedge \theta(g_ib_i)$. Hence, by compactness (or rather $|N|^+$-saturation of $\C$), there are $a \models u$ and $b \models \pi(x) $ such that $ab \models q$.
Consider any $q \in \cl_\tau(Q)$. Then $q=\lim_i g_iq_i$ for some nets $(g_i)_i$ in $G$ and $(q_i)_i$ in $Q$ with $\lim_ig_i=u$. Note that $u \in S_{X^m,M}(N)$ and $q \in S_{X^n,M}(N)$ for some $m,n \in \mathbb{N}$. Consider any formulas: $\varphi(x) \in u$ with $\varphi(\C) \subseteq \bar{X}^m$, $\psi(x)$ implied by $\pi(x)$, and $\theta (x) \in q$ with $\theta(\C) \subseteq \bar{X}^n$, all with parameters from $N$. Then there is an index $i$ such that $g_i \models \varphi(x)$ and $\theta(x) \in g_iq_i$. Take any $b_i \models q_i$. Then $\C \models \varphi(g_i) \wedge \psi(b_i) \wedge \theta(g_ib_i)$, which implies that $g_i \in \bar{X}^m$ and $b_i \in \bar{X}^{n+m}$. Hence, by compactness or rather $|N|^+$-saturation of $\C$ (using the assumption that group multiplication restricted to $\bar X^m \times \bar X^{n+m}$ is definable), there are  $a \models u$ (so $a \in \bar{X}^m$) and $b \models \pi(x)$ with $b \in \bar{X}^{n+m}$ such that $ab \models q$.
\end{proof}

%%%Krzys25: The next three lemmas were extracted from the proof of the old Proposition 3.21 (now 3.26), as requested by the referee.
\begin{lemma}\label{lemma: quasi-compactness from the old 3.21}
Let $V \subseteq u\mathcal{M}$ be $\tau$-closed and  contained in $S_{X^n,M}(N)$ for some $n$. Then $V$ is quasi-compact in the $\tau$-topology.  
\end{lemma}

\begin{proof}
We need to show that any net $(p_i)_{i\in I}$ in $V$ has a $\tau$-convergent subnet. By compactness of $S_{X^{n},M}(N)$, the net $(p_i)_{i\in I}$ has a subnet $(q_j)_{j \in J}$ convergent to some $r \in S_{X^{n},M}(N)$ in the usual topology on $S_{X^{n},M}(N)$. By Lemma \ref{fct:ulimit}, $\tau\textrm{-}\lim_j q_j = ur$, and, by $\tau$-closedness of $V$, $ur \in V$.
\end{proof}

\begin{lemma}\label{lemma: tau-cl increases n by m}
Let $V \subseteq u\mathcal{M}$ be contained in $S_{X^n,M}(N)$ for some $n$. Take $m$ such that $u \in S_{X^m,M}(N)$\footnote{In Lemma \ref{lemma: u in F_1}, we will see that $u \in S_{X^2,M}(N)$.\label{note1}}. Then $\cl_{\tau}(V) \subseteq S_{X^{n+m},M}(N)$.
\end{lemma}

\begin{proof}
%Consider any $p \in \cl_\tau(V)$. Then $p=\lim_i g_ip_i$ for some nets $(g_i)_i$ in $G$ and $(p_i)_i$ in $V$ with $\lim_ig_i=u$. So $p=\tp(ab/N)$ for some $a \in \bar X^m$ and $b \in \bar X^{n}$. Hence, $p \in  S_{X^{n+m},M}(N)$.
Consider any $p \in \cl_\tau(V)$. By Lemma \ref{lemma: tau-closure in terms of realizations of types}, $p=\tp(ab/N)$ for some $a \in \bar X^m$ and $b \in \bar X^{n}$. Hence, $p \in  S_{X^{n+m},M}(N)$.
\end{proof}

\begin{lemma}\label{lemma: tau-neighborhood contained in n+m}
Let $q \in u\mathcal{M}$. Take $n,m$ such that $q \in S_{X^n,M}(N)$ and $u \in S_{X^m,M}(N)$\footref{note1}. Then $q$ has a $\tau$-open neighborhood $V$ contained in $S_{X^{n+m},M}(N)$.
\end{lemma}

\begin{proof}
Let 
$$P:= S_{X^{n+m},M}(N)^c \cap u\M,$$
where $S_{X^{n+m},M}(N)^c$ denotes the complement of  $S_{X^{n+m},M}(N)$ in $S_{G,M}(N)$.

\begin{clm}
$S_{X^n,M}(N) \cap \cl_\tau(P) = \emptyset$.
\end{clm}

\begin{clmproof}
%Take any $p \in \cl_\tau(P)$. Then $p=\lim_i g_ip_i$ for some nets $(g_i)_i$ in $G$ and $(p_i)_i$ in $P$ with $\lim_ig_i=u$. So for sufficiently large $i$ we have that $g_i \in X^m$ and $p_i \notin S_{X^{n+m},M}(N)$. Hence, $p=\tp(ab/N)$ for some $a \in \bar X^m$ and $b \notin \bar X^{n+m}$. Therefore, $p \notin S_{X^n,M}(N)$, as required.
Take any $p \in \cl_\tau(P)$. By Lemma \ref{lemma: tau-closure in terms of realizations of types}, $p=\tp(ab/N)$ for some $a \in \bar X^m$ and $b \notin \bar X^{n+m}$. Therefore, $p \notin S_{X^n,M}(N)$, as required.
\end{clmproof}

Let 
$$V:=u\M \setminus \cl_\tau(P).$$ 
By Claim 1 and the above choices, $q \in S_{X^n,M}(N) \cap u\M \subseteq V \subseteq S_{X^{n+m},M}(N)$. In particular,  $V$ is a $\tau$-open neighborhood of $q$.
\end{proof}

\begin{definition}
Let us say that a topological space $P$ is {\em quasi locally compact} if every point $p \in P$ has a neighborhood $U$ whose closure is quasi-compact.
\end{definition}

%%%Krzys25: Added ``in the $\tau$-topology'' as the referee suggested.

\begin{proposition}\label{proposition: quasi local compactness}
The Ellis group $u\M$ is a quasi locally compact $T_1$ space in the $\tau$-topology.
\end{proposition}

%%%Krzys25: The proof is shortened by referring to the above lemmas (which were extracted from the old proof).

\begin{proof}
The fact that it is $T_1$ is easy: $\cl_{\tau}(\{p\}) = u(u \circ \{p\}) = \{u(up)\}= \{p\}$. Quasi local compactness follows from Lemmas \ref{lemma: quasi-compactness from the old 3.21}, \ref{lemma: tau-cl increases n by m}, \ref{lemma: tau-neighborhood contained in n+m}. Namely,
consider any  $q \in u\M$. Then $q \in S_{X^n,M}(N)$ for some $n$. Also, $u \in S_{X^m,M}(N)$ for some $m$. By Lemma  \ref{lemma: tau-neighborhood contained in n+m}, $q$ has a $\tau$-open neighborhood $V$ contained in $S_{X^{n+m},M}(N)$. By Lemma \ref{lemma: tau-cl increases n by m}, $\cl_{\tau}(V) \subseteq S_{X^{n+2m},M}(N)$. Hence, $\cl_{\tau}(V)$ is quasi-compact by Lemma \ref{lemma: quasi-compactness from the old 3.21}.
\end{proof}

\begin{proposition}\label{proposition Ellis group is semitopological}
$u\M$ equipped with the $\tau$-topology is a semitopological group, i.e. group operation is separately continuous.
\end{proposition}

\begin{proof}
The argument from Fact A.36 of \cite{Rz} works without any changes.
\end{proof}

The proof of Fact A.37 of \cite{Rz} applies to our context, so we get that all Ellis groups of $S_{G,M}(N)$ (for varying minimal left ideals $\M$ and idempotents $u \in J(\M)$) are in fact topologically isomorphic. So the Ellis group of $S_{G,M}(N)$ is a well-defined semitopological group associated with $S_{G,M}(N)$.

\begin{definition}
Define $H(u\M)$ as $\bigcap \cl_\tau(V)$ with $V$ ranging over all $\tau$-neighborhoods of $u$.
\end{definition}

\begin{proposition}
$H(u\M)$ is a $\tau$-closed normal subgroup of $u\M$, and $u\M/H(u\M)$ is a locally compact (so Hausdorff) topological group.
\end{proposition}

\begin{proof}
This is an elaboration on the proof of Fact A.40 (so, in fact, Fact A.12) of \cite{Rz}. 

Exactly as in the proof of \cite[Fact A.12]{Rz}, we get that $H(u\M)$ is a $\tau$-closed normal subsemigroup containing $u$. 
Hence, for every $h \in H(u\M)$ both $hH(u\M)$ and $H(u\M)h$ are subsemigroups of $H(u\M)$.
%Thus, using Propositions \ref{proposition: quasi local compactness} and \ref{proposition Ellis group is semitopological} (note that the latter one implies that multiplication on the left or on the right by a fixed element is a homeomorphism), we get that for every $h \in H(u\M)$ both $hH(u\M)$ and $H(u\M)h$ are subsemigroups of $H(u\M)$ which are quasi locally compact $T_1$ spaces.

\begin{clm}
For every $h \in H(u\M)$, both $hH(u\M)$ and $H(u\M)h$ contain an idempotent.
\end{clm}

\begin{clmproof}
Fix $h \in H(u\M)$ and consider $hH(u\M)$ (the case of $H(u\M)h$ is analogous).
%%%Krzys25: Modified references due to the changes above.
%By the proof of Proposition \ref{proposition: quasi local compactness}, there is a $\tau$-neighborhood $V$ of $u$ in $u\M$ such that $\cl_\tau(V) \subseteq S_{X^k,M}(N)$ for some $k$. By the last paragraph of the proof of Proposition \ref{proposition: quasi local compactness}, $\cl_\tau(V)$ is quasi-compact. 
By Lemmas \ref{lemma: tau-cl increases n by m} and \ref{lemma: tau-neighborhood contained in n+m}, there is a $\tau$-neighborhood $V$ of $u$ in $u\M$ such that $\cl_\tau(V) \subseteq S_{X^k,M}(N)$ for some $k$. Then, by Lemma \ref{lemma: quasi-compactness from the old 3.21}, $\cl_\tau(V)$ is quasi-compact.
Hence, $H(u\M)$ is $\tau$-closed,  quasi-compact, and $T_1$. Therefore, by Proposition  \ref{proposition Ellis group is semitopological}  (which implies that multiplication on the left or on the right by a fixed element is a homeomorphism), $hH(u\M)$ is $\tau$-closed, quasi-compact, $T_1$, and the map $hH(u\M) \to hH(u\M)$ given by $s \mapsto ss_0$ is continuous and closed for every $s_0 \in hH(u\M)$. Hence, $hH(u\M)$ contains an idempotent by Fact \ref{Ellis theorem}.
\end{clmproof}

Since the only idempotent in the group $u\M$ is $u$, we conclude from the above claim that $H(u\M)$ is a subgroup of $u\M$. By Proposition \ref{proposition: quasi local compactness}, $u\M$ is quasi locally compact, and so it is {\em weakly quasi locally compact} in the sense that every $p \in u\M$ has a quasi-compact neighborhood. This property is easily seen to be preserved under taking group quotients of semitopological groups, so $u\M/H(u\M)$ is weakly quasi locally compact. 

The last paragraph of the proof of \cite[Fact A.12]{Rz} applies to our context, so $u\M/H(u\M)$ is Hausdorff.

By the last two paragraphs, $u\M/H(u\M)$ is locally compact. On the other hand, since $u\M$ is a semitopological group, so is $u\M/H(u\M)$. Therefore, by Ellis joint continuity theorem \cite[Theorem 2]{Ell57}, we get that $u\M/H(u\M)$ is jointly continuous and inversion is continuous. Thus,  $u\M/H(u\M)$ is a locally compact topological group.
\end{proof}

\subsection{The main theorem}\label{subsection: main theorem}

Recall that we are in the situation and notation described at the end of Subsection \ref{subsection: model theory}. Let $\M$ be a minimal left ideal of $S_{G,M}(N)$ and $u$ an idempotent in $\M$. Let $F \colon G \to u\M$ be given by $F(g):=ugu$ and $\hat{F}\colon S_{G,M}(N) \to u\M$ be the extension of $F$ given by $\hat{F}(p):=upu$. Let $f \colon G \to u\M/H(u\M)$ be given by $f(g):= ugu/H(u\M)$ and $\hat{f} \colon S_{G,M}(N) \to u\M/H(u\M)$ be the extension of $f$ given by $\hat{f}(p) := upu/H(u\M)$. In particular, $f = \pi F$ where $\pi \colon u\M \to u\M/H(u\M)$ is the quotient map.
 
%%%Krzys25: First I added the sentence below in response to the referee's complain. But then I removed it, added the next sentence, and changed the notation form $\cl_\tau$ to $\cl$ in the quotient group (hopefully everywhere).
%Abusing notation, the closure in the quotient topology on $u\mathcal{M}/H(u\M)$ induced by the $\tau$-topology on $u\M$ will also be denoted  by $\cl_{\tau}$.
The group $u\M/H(u\M)$ is always equipped with the quotient topology induced by the $\tau$-topology on $u\M$, and $\cl$ denotes the closure operator in this quotient topology.

%To define a suitable error set $C$, we need to introduce the following sets.
The following sets will play a key role.

\begin{flalign*}
&\:F_n:=  \{ x_1y_1^{-1} \dots x_ny_n^{-1}: x_i, y_i \in \bar{G}\; \textrm{and}\; x_i \equiv_M y_i\; \textrm{for all}\; i\leq n\}\\
%%%Krzys25: I added `` \textrm{ such that } \tp(a/N) \textrm{ is a coheir over } M'' as suggested by the referee.
&\: \tilde{F}_n:=  \{ \tp(a/N) \in S_{G,M}(N): a \in F_n \textrm{ such that } \tp(a/N) \textrm{ is a coheir over } M\}\\
&\: \tilde{F}:=  ((\tilde{F}_7 \cap u\mathcal{M})/H(u\mathcal{M}))^{u\mathcal{M}/H(u\mathcal{M})}\\
&\: C:=  \cl(\tilde{F}) \cup \cl(\tilde{F})^{-1}
\end{flalign*}

Here is the main result, i.e. our version of Hrushovski's \cite[Theorem 4.2]{Hru2}.

\begin{theorem}\label{theorem: main Theorem}
The above function $f$ is a generalized definable locally compact model of $X$ with the  compact, normal, symmetric error set $C$ defined above, which is witnessed by $l=2$ (see Definition \ref{definition: gen. loc. comp. model}). 
Moreover, $f^{-1}[C] \subseteq X^{30}$ and there is a compact neighborhood $U$ of the neutral element in $u\M/H(u\M)$ such that $f^{-1}[U] \subseteq X^{14}$ and $f^{-1}[UC] \subseteq X^{34}$.
\end{theorem}

%Before the proof, we need to show a few lemmas. 
The proof of Theorem \ref{theorem: main Theorem} starts after the proof of Lemma \ref{lemma: main lemma for definability} below.

\begin{lemma}\label{lemma: F_n in X^2n}
%If $a \equiv_M b$ are elements of $\bar G$, then $ab^{-1} \in \bar X^2$.
$F_1= \{xy^{-1}: x,y \in \bar X \textrm{ with } x \equiv_M y\}$. In particular, $F_n \subseteq \bar X^{2n}$ is $M$-type-definable (that is, the set of realizations of a type over $M$), and so $\tilde{F}_n \subseteq S_{X^{2n},M}(N)$ is closed.
\end{lemma}

\begin{proof}
Only ($\subseteq$) requires a proof. Take any $a,b \in \bar G$ with $a \equiv_M b$.  Then $a \in \bar X^n$ for some $n$. Since $X^n \subseteq XS$ for some finite $S \subseteq G$, we have that $\bar X^n \subseteq \bar X S$. So $ac \in \bar X$ for some $c \in S^{-1}$. As $a\equiv_Mb$ and $c \in M$, also $bc \in \bar X$ and $ac \equiv_M bc$. So $ab^{-1}=(ac)(bc)^{-1} \in  \{xy^{-1}: x,y \in \bar X \textrm{ with } x \equiv_M y\}$. The rest easily follows.
\end{proof}

\begin{lemma}\phantomsection\label{lemma: u in F_1}
\begin{enumerate}
\item $u \in \tilde{F}_1 \subseteq S_{X^2,M}(N)$.
\item If $p \in \tilde{F}_n \cap u\M$, then $p^{-1} \in \tilde{F}_{n+1} \cap u\M$.
\item $\hat{F}^{-1}[S_{X^n,M}(N) \cap u\M] \subseteq S_{X^{n+4},M}(N)$.
\item $F^{-1}[S_{X^n,M}(N) \cap u\M] \subseteq X^{n+4}$.
\end{enumerate}
\end{lemma}

\begin{proof}
(1) $u^2=u$ implies that there are $a$ and $b$ realizing $u$ such that $ab \models u$. So $ab \equiv_M b$, hence $a=(ab)b^{-1} \in F_1$. Therefore, $u \in \tilde{F}_1$ which is contained in $S_{X^2,M}(N)$ by Lemma \ref{lemma: F_n in X^2n}.

%%%Krzys25: A few typos spotted by the referee corrected.
(2) Since $pp^{-1}=u$, there are $a \models p$ and $b \models p^{-1}$ such that $ab \models u$. By assumption, $a \in F_n$, and, by (1), $ab \in F_1$. Therefore, $b \in F_{n+1}$.

(3) Take any $p \in \hat{F}^{-1}[S_{X^{n},M}(N)]$, i.e. $upu \in S_{X^{n},M}(N)$. Then $abc = d \in \bar X^{n}$ for some $a \models u$, $b \models p$, and $c\models u$. So $b = a^{-1}dc^{-1} \in \bar X^2 \bar X^{n} \bar X^2 = \bar X^{n+4}$ by (1). Hence, $p \in S_{X^{n+4},M}(N)$.

%(3) Take any $g \in F^{-1}[S_{X^{n},M}(N)]$, i.e. $ugu \in S_{X^{n},M}(N)$. Then $agb = c \in \bar X^{n}$ for some $a \models u$ and $b\models u$. So $g = a^{-1}cb^{-1} \in \bar X^2 \bar X^{n} \bar X^2 = \bar X^{n+4}$ by (1). As $g \in G$, we get $g \in X^{n+4}$.
(4) Take any $g \in F^{-1}[S_{X^{n},M}(N)]$.  By (3), $\tp(g/N) \in S_{X^{n+4},M}(N)$, so $g \in \bar X^{n+4}$. As $g \in G$, we get $g \in X^{n+4}$.
\end{proof}

\begin{lemma}\label{lemma: existence of V}
 There exists a $\tau$-open neighborhood $V$ of $u$ in $u\M$ such that $V\subseteq S_{X^4,M}(N)$.
\end{lemma}

%%%Krzys25: Adjusted references.
\begin{proof}
%By Lemma \ref{lemma: u in F_1}(1), $u \in \tilde{F}_1\subseteq S_{X^2,M}(N)$. So the proof of Proposition \ref{proposition: quasi local compactness} (in which we can take $q:=u$ and $n=m=2$) yields a $\tau$-open neighborhood $V$ of $u$ which is contained in $S_{X^4,M}(N)$.
This is a direct consequence of  Lemmas \ref{lemma: u in F_1}(1) and \ref{lemma: tau-neighborhood contained in n+m}.
\end{proof}

\begin{lemma}\phantomsection\label{lemma: key properties of F_n's}
\begin{enumerate}
\item $(\tilde{F}_7 \cap u\M)^{u\M} \subseteq \tilde{F}_8 \cap u\M \subseteq S_{X^{16},M}(N) \cap u\M$.
\item $\cl_\tau((\tilde{F}_7 \cap u\M)^{u\M}) \subseteq \tilde{F}_9 \cap u\M \subseteq  S_{X^{18},M}(N) \cap u\M$ is quasi-compact.
\item $\cl(\tilde{F}) = \pi[\cl_\tau((F_7 \cap u\M)^{u\M})] \subseteq (\tilde{F}_9 \cap u\M)/H(u\M)$ is compact.
\item $C$ is compact, normal, symmetric, and contained in $(\tilde{F}_{10} \cap u\M)/H(u\M)$.
\end{enumerate}
\end{lemma}

\begin{proof}
(1) Take $p \in \tilde{F}_7 \cap u\M$ and $q \in u\M$. The goal is to show that $qpq^{-1} \in \tilde{F}_8$ (the last inclusion follows from Lemma \ref{lemma: F_n in X^2n}). 

By the definition of $*$, we can find $\alpha \models q$, $\beta \models q^{-1}$, and $a_1,b_1,\dots,a_7,b_7$ with $a_i \equiv_M b_i$ for all $i \leq 7$ and $\tp(\alpha/N,a_{\leq 7},b_{\leq 7},\beta)$ a coheir over $M$ such that $\alpha (\prod_{i \leq 7} a_ib_i^{-1}) \beta \models qpq^{-1}$. We have $\alpha (\prod_{i \leq 7} a_ib_i^{-1}) \beta =  (\prod_{i \leq 7} a_i^\alpha (b_i^\alpha)^{-1}) \alpha\beta$. Now, by Lemma \ref{lemma: u in F_1}(1), $\alpha \beta \models qq^{-1}=u \in \tilde{F}_1$, so $\alpha \beta \in F_1$. On the other hand,  since $\tp(\alpha/M,a_i,b_i)$ is a coheir over $M$ and $a_i \equiv_M b_i$, by Remark \ref{remark: on coheirs}, we get that $a_i^\alpha \equiv_M b_i^\alpha$. Therefore, $(\prod_{i \leq 7} a_i^\alpha (b_i^\alpha)^{-1}) \alpha\beta \in F_8$, so $qpq^{-1} \in \tilde{F}_8$.

(2) By Lemma \ref{lemma: F_n in X^2n}, the sets $F_1$ and $F_8$ are $M$-type-definable. By Lemma \ref{lemma: u in F_1}(1), $u \in \tilde{F}_1$, and, by (1), $(\tilde{F}_7 \cap u\M)^{u\M} \subseteq \tilde{F}_8$. 
%%%Krzys25: Adjusted references in the rest of the proof of (2).
%Thus, using the definition of $\cl_{\tau}$ as in the proof of Claim 2 in the proof of Proposition \ref{proposition: quasi local compactness}, 
Thus, 
%using the definition of $\cl_{\tau}$ as in the proof of Lemma  \ref{lemma: tau-cl increases n by m},
using Lemma \ref{lemma: tau-closure in terms of realizations of types},
we get that $\cl_\tau((\tilde{F}_7 \cap u\M)^{u\M}) \subseteq \tilde{F}_9$ which is contained in  $S_{X^{18},M}(N)$ by Lemma \ref{lemma: F_n in X^2n}. 
%Then quasi-compactness follows from the argument in the last paragraph of the proof of Proposition \ref{proposition: quasi local compactness}.
Then quasi-compactness follows from Lemma \ref{lemma: quasi-compactness from the old 3.21}.

(3) follows from (2) and Hausdorffness of $u\M/H(u\M)$.

(4) Compactness follows from (3) and the fact that $u\M/H(u\M)$ is a topological group. Normality is immediate, also using that $u\M/H(u\M)$ is a topological group. 
%By (3), in order to show that $C \subseteq  (\tilde{F}_{10} \cap u\M)/H(u\M)$, it is enough to observe that whenever $p \in \tilde{F}_{9} \cap u\M$, then $p^{-1} \in \tilde{F}_{10}$. But this easily follows using $qq^{-1}=u \in \tilde{F}_1$.
The inclusion $C \subseteq  (\tilde{F}_{10} \cap u\M)/H(u\M)$ follows from (3) and Lemma \ref{lemma: u in F_1}(2).
\end{proof}

\begin{lemma}\label{lemma: H(uM) form KrPi}
$H(u\M) \subseteq \tilde{F}_3 \cap u\M \subseteq S_{X^6,M}(N) \cap u\M$.
\end{lemma}

\begin{proof}
The second inclusion is by Lemma \ref{lemma: F_n in X^2n}. The first one is essentially contained in the proof of Theorem 0.1(2) of \cite{KrPi}, but we repeat it here for the reader's convenience. 

By Lemma \ref{lemma: u in F_1}(1), $u \in \tilde{F}_1$. By Lemma \ref{lemma: F_n in X^2n}, $F_2$ is $M$-type-definable.
So let $\rho$ be the partial type over $M$ defining $F_{2}$ and closed under conjunction.
Consider any $\varphi(x) \in \rho$. Let $$V:=[\neg \varphi(x)] \cap u{\M},$$ where  $[\neg \varphi(x)]$ is the clopen subset of $S_{G,M}(N)$ consisting of all types containing $\neg \varphi(x)$.

\begin{clm}
\begin{enumerate}
\item $u \notin \cl_{\tau}(V)$.
\item $\cl_\tau (u{\M} \setminus \cl_\tau(V)) \subseteq \cl_\tau(u{\M} \setminus V) \subseteq \widetilde{F}_{3}^\varphi$, where $\widetilde{F}_{3}^\varphi := \{ \tp(ab/N) \in S_{G,M}(N): a \in F_1 \textrm{ and } b \models \varphi(x)\}$.
\end{enumerate}
\end{clm}

\begin{clmproof}
%%%Krzys25: Slightly modified by referring to Lemma 3.21.
(1) Suppose for a contradiction that $u \in \cl_{\tau}(V)$. 
%Using the definition of $\cl_\tau$, 
By Lemma \ref{lemma: tau-closure in terms of realizations of types},
there are $a \models u$ and $b \models \neg \varphi(x)$ with $ab \models u$. Then $a \in F_1$ and $ab \in F_1$, so $b \in F_{2}$, and hence $b\models \varphi(x)$, a contradiction.

(2) We need to check that  $\cl_\tau(u{\M} \setminus V) \subseteq \widetilde{F}_{3}^\varphi$. Consider any $p \in \cl_\tau(u{\M} \setminus V)$. As before, there are $a \models u$ and $b \models \varphi(x)$ such that $ab \models p$. Then $a \in F_1$, so $\tp(ab/N) \in \widetilde{F}_{3}^\varphi$.
\end{clmproof}

Notice that $\bigcap_{\varphi(x) \in \rho}\widetilde{F}_{3}^\varphi = \widetilde{F}_{3}$. So, by the claim,
$$
\begin{array}{ll}
H(u{\M}) = \bigcap \left\{\cl_{\tau}(U): U \; \mbox{$\tau$-neighborhood of} \; u\right\} \subseteq \bigcap_{\varphi(x) \in \rho}\widetilde{F}_{3}^\varphi \cap u{\M} = \widetilde{F}_{3} \cap u{\M},
\end{array}$$
which completes the proof.
\end{proof}

\begin{lemma}\phantomsection\label{lemma: K and X}
\begin{enumerate}
\item Every compact $K \subseteq u\M/H(u\M)$ is contained in $(S_{X^n,M}(N) \cap u\M)/H(u\M)$ for some $n \in \mathbb{N}$ and $\pi^{-1}[K]$ is quasi-compact.
\item Every quasi-compact $K \subseteq u\M$ is contained in $S_{X^n,M}(N)$ for some $n \in \mathbb{N}$.
\end{enumerate}
\end{lemma}

\begin{proof}
(1) Take $V$ from Lemma \ref{lemma: existence of V}. Then $U:=\pi[V] \subseteq (S_{X^4,M}(N) \cap u\M)/H(u\M)$ is open in $u\M/H(u\M)$, so $K$ is covered by finitely many translates of $U$. Since all the translating elements are in some $(S_{X^m,M}(N) \cap u\M)/H(u\M)$, we get that $K \subseteq (S_{X^{m+4},M}(N) \cap u\M)/H(u\M)$. So the $\tau$-closed subset $\pi^{-1}[K]$ of $u\M$ is contained in $S_{X^{m+4},M}(N) H(u\M)$ which in turn is contained in  $S_{X^{m+10},M}(N)$ by Lemma \ref{lemma: H(uM) form KrPi}. 
%Thus, by Claim 2 in the proof of Proportion \ref{proposition: quasi local compactness}, $\cl_{\tau}(V) \subseteq S_{X^{n+8},M}(N)$. 
So the final paragraph of the proof of Proposition \ref{proposition: quasi local compactness} shows that $\pi^{-1}[K]$ is quasi-compact.

(2) follows from (1) and Lemma \ref{lemma: H(uM) form KrPi}, as $\pi[K]$ is compact. 
\end{proof}

The next two lemmas will be needed only in the proof of definability of $f$.

\begin{lemma}\label{lemma: claim in the notes}
If $p=\tp(a/N)$ and $q=\tp(b/N)$ belong to $u\M$ and $a \in F_nb$, then $p \in (\tilde{F}_{n+2} \cap u\M) q$.
\end{lemma}

\begin{proof}
As $qq^{-1}=u$, we can find $b' \models q^{-1}$ such that $bb' \models u$. So $b'=b^{-1}\alpha$ for some $\alpha \models u$. Then $pq^{-1} = \tp(a''b^{-1} \alpha/N)$ for some $a'' \equiv_N a$. As $a \in F_nb$, we have $a'' \in F_nb''$ for some $b'' \equiv_N b$, i.e. $a'' = cb''$ for some $c \in F_n$. Thus, $pq^{-1} = \tp(cb''b^{-1} \alpha/N)$. Since by Lemma \ref{lemma: u in F_1}(1) we know that $\alpha \in F_1$, we conclude that $cb''b^{-1} \alpha \in F_{n+2}$, so $p \in (\tilde{F}_{n+2} \cap u\M) q$.
\end{proof}

\begin{lemma}\label{lemma: main lemma for definability}
%\begin{enumerate}
$\hat{F}[\overline{\hat{F}^{-1}[K]}] \subseteq (\tilde{F}_7 \cap u\M)K$ for every $\tau$-closed, quasi-compact subset $K$ of $u\M$, where $\overline{\hat{F}^{-1}[K]}$ denotes the closure of $\hat{F}^{-1}[K]$ in $S_{G,M}(N)$.
%\item  $\hat{f}[\overline{\hat{f}^{-1}[X]}] \subseteq ((\tilde{F}_7 \cap u\M)/H(u\M)) X$ for every $\tau$-closed subset $X$ of $u\M/H(u\M)$.
%\end{enumerate}
\end{lemma}

\begin{proof}
%Item (2) follows from (1) since $\hat{f} = \pi \circ \hat{F}$. So let us prove (1).
Consider any $p \in \overline{\hat{F}^{-1}[K]}$ and $q=\hat{F}(p)=upu$. Then $q=\tp(\alpha a \beta/N)$ for some $\alpha \models u$, $a \models p$, $\beta \models u$. Also, $p=\lim p_j'$ for some net $(p_j')_j$ from $\hat{F}^{-1}[K]$. Take a net $(g_k')_k$ in $G$ such that $\lim g_k' = u$.

By Lemma \ref{lemma: K and X}(2), $K\subseteq S_{X^n,M}(N)$ for some $n$, so
%Using Lemma \ref{lemma: u in F_1}(1) and the formula for $\hat{F}$, we easily get that $\hat{F}^{-1}[K] \subseteq S_{X^{n+4},M}(N)$. 
Lemma \ref{lemma: u in F_1}(3) implies that $\hat{F}^{-1}[K] \subseteq S_{X^{n+4},M}(N)$. 
%%%Krzys: I wrote "end segment" in place of "final part".
Moreover, since $u \in S_{X^2,M}(N)$, taking an end segment of $(g_k')_k$, we can assume that  $g_k'\in X^2$ for all $k$. Then all $g_k' u p_j'$ belong to $S_{X^{n+8},M}(N)$. By compactness of  $S_{X^{n+8},M}(N)$, we can find subnets $(g_i)_{i\in I}$ of $(g_k')_k$ and $(p_i)_{i\in I}$ of $(p_j')_j$ such that $\lim_i g_iup_i$ exists.

Clearly, $\lim g_i =u$, $\lim p_i=p$, $up_iu \in u \hat{F}^{-1}[K]u =\hat{F}[\hat{F}^{-1}[K]]=K$, and $r:=\lim g_iup_iu =(\lim g_iup_i)u$ exists. Hence, by Definition \ref{definition: circle operation}, we get $r \in u \circ K$, so $ur \in u(u \circ K) =\cl_{\tau}(K) =K$, as $K$ is $\tau$-closed.

Since $ur= u(\lim g_iup_i)u$, by compactness (or rather $|N|^+$-saturation of $\C$), as in Lemma \ref{lemma: tau-closure in terms of realizations of types}, we get $ur=\tp(\gamma \delta \epsilon b \beta/N)$ for some $\gamma, \delta,\epsilon$ realizing $u$ and $b \models p$ (note that $\beta$ can be chosen the same as at the beginning of the proof.)

Put $x:= \alpha a \beta$ and $y:= \gamma \delta \epsilon b \beta$. Then $x=\alpha a b^{-1} \epsilon^{-1} \delta^{-1} \gamma^{-1} y \in F_5 y$, because $a \equiv_M b$ and $\alpha, \epsilon, \delta, \gamma \in F_1$ by Lemma \ref{lemma: u in F_1}(1). Therefore, using Lemma \ref{lemma: claim in the notes}, we get $$q=\tp(x/N) \in (\tilde{F}_7 \cap u\M) \tp(y/N) =(\tilde{F}_7 \cap u\M)ur.$$ 
As we observed above that $ur \in K$, we get $q \in (\tilde{F}_7 \cap u\M)K$.
\end{proof}

\begin{proof}[Proof of Theorem \ref{theorem: main Theorem}]
By Lemma \ref{lemma: key properties of F_n's}(4), we already know that $C$ is compact, normal, and symmetric. Let us divide the proof into numbered parts.

(1) {\em $C$ is an error set of $f$.}

By normality of $C$, it is enough to show that $\error_r(f) \subseteq C$. We will show more, namely that $\error_r(f) \subseteq (\tilde{F}_{3} \cap u\M)/H(u\M)$. For that take any $g,h \in G$ and we need to show that $F(h)^{-1}F(g)^{-1}F(gh) \in \tilde{F}_3 \cap u\M$. The left hand side equals $(uhu)^{-1}(ugu)^{-1}ughu = (uhu)^{-1}(ugu)^{-1}ghu$.

%%%Krzys25: I added ``and similarly for $h$ in place of $g$'' as requested by the referee.
\begin{clm}
$(ugu)^{-1} = \tp(xy^{-1}g^{-1}/N)$ for some $x \equiv_N y$, and similarly for $h$ in place of $g$.
\end{clm}

\begin{clmproof}
Let $\alpha \models u$. Then $g\alpha \models gu$. Let $a \models (ugu)^{-1}$ be such that $\tp(a/N,\alpha)$ is a coheir over $M$. Then $u=(ugu)^{-1}ugu=(ugu)^{-1}gu= \tp(ag\alpha/N)$. Put $x:= ag\alpha$ and $y:=\alpha$. Then $x \equiv_N y$ (as each of these elements realizes $u$) and $a=xy^{-1}g^{-1}$.
\end{clmproof}

By this claim, we conclude that $F(h)^{-1}F(g)^{-1}F(gh) = \tp(zt^{-1}h^{-1} xy^{-1}g^{-1}gh\alpha/N)$ for some $z \equiv_N t$, $x \equiv_Ny$, and $\alpha \models u$. But $zt^{-1}h^{-1} xy^{-1}g^{-1}gh\alpha= zt^{-1} x^{h^{-1}} (y^{h^{-1}})^{-1} \alpha \in F_3$, because $z \equiv_M t$, $x^{h^{-1}} \equiv_M y^{h^{-1}}$ (as $x \equiv_M y$ and $h \in M$), and $\alpha \in F_1$ (by Lemma \ref{lemma: u in F_1}(1)). Therefore, $F(h)^{-1}F(g)^{-1}F(gh) \in \tilde{F}_3 \cap u\M$.

(2) {\em There is a $\tau$-open neighborhood $V$ of $u$ in $u\M$ such that $V\subseteq S_{X^4,M}(N)$. For any such $V$, $U:=\pi[V]$ is an open neighborhood of the neutral element in $u\M/H(u\M)$ and $f^{-1}[U] \subseteq X^{14}$. Thus, $f^{-1}[U] \subseteq X^{14}$ also  holds for $U$ replaced by any (in particular by a compact)  neighborhood of the neutral element in $u\M/H(u\M)$ contained in $U$.}

%By Lemma \ref{lemma: u in F_1}(1), $u \in \tilde{F}_1\subseteq S_{X^2,M}(N)$. So the proof of Proposition \ref{proposition: quasi local compactness} yields a $\tau$-open neighborhood $V$ of $u$ which is contained in $S_{X^4,M}(N)$. 
The existence of $V$ is by Lemma \ref{lemma: existence of V}. Then $U:=\pi[V]$ is an open neighborhood of $u/H(u\M)$ in $u\M/H(u\M)$. By Lemma \ref{lemma: H(uM) form KrPi}, $H(u\M) \subseteq S_{X^6,M}(N)$. So $\pi^{-1}[U] =VH(u\M) \subseteq S_{X^4,M}(N)  S_{X^6,M}(N) \subseteq S_{X^{10},M}(N)$. Hence, $f^{-1}[U]=F^{-1}[\pi^{-1}[U]] \subseteq F^{-1}[S_{X^{10},M}(N)]$, and the last preimage is contained in $X^{14}$ by Lemma  \ref{lemma: u in F_1}(4). 
%So take any $g \in F^{-1}[S_{X^{10},M}(N)]$, i.e. $ugu \in S_{X^{10},M}(N)$. Then $agb = c \in \bar X^{10}$ for some $a \models u$ and $b\models u$. So $g = a^{-1}cb^{-1} \in \bar X^2 \bar X^{10} \bar X^2 = \bar X^{14}$. As $g \in G$, we get $g \in X^{14}$.

(3) {\em For every compact $K \subseteq u\M/H(u\M)$ there is $k\in \mathbb{N}$ with $f^{-1}[K] \subseteq X^k$.}

Consider any compact $K\subseteq u\M/H(u\M)$. By Lemma \ref{lemma: K and X}(1), $K \subseteq (S_{X^n,M}(N) \cap u\M)/H(u\M)$ for some $n$. So 
$$f^{-1}[K] =F^{-1}[\pi^{-1}[K]] \subseteq F^{-1}[S_{X^n,M}(N) H(u\M)] \subseteq F^{-1}[S_{X^{n+6},M}(N)] \subseteq X^{n+10},$$
where the second inclusion follows from Lemma \ref{lemma: H(uM) form KrPi} and the last one by Lemma \ref{lemma: u in F_1}(4).

(4) {\em $f[X^i]$ is relatively compact for every $i \in \mathbb{N}$.}

By Remark \ref{remark: i=1 is enough}, it is enough to show it for $i=1$. We have 
%%%Krzys25: I removed the intersection with $u\M$ as pointed out by the referee.
$$f[X] = \pi[F[X]]\subseteq \pi[uS_{X,M}(N)u] \subseteq \pi[S_{X^5,M}(N) \cap u\M] \subseteq \pi[\cl_\tau(S_{X^5,M}(N) \cap u\M)],$$
where the second inclusion follows from Lemma \ref{lemma: u in F_1}(1).
%%%Krzys25: Adjusted references below. 
%By this lemma and the proof of Claim 2 in the proof of Proposition \ref{proposition: quasi local compactness}, we have $\cl_\tau(S_{X^5,M}(N) \cap u\M) \subseteq S_{X^7,M}(N)$. So the argument in the last paragraph of the proof of Proposition \ref{proposition: quasi local compactness} shows that $\cl_\tau(S_{X^5,M}(N) \cap u\M)$ is quasi-compact. 
By Lemmas \ref{lemma: u in F_1}(1) and  \ref{lemma: tau-cl increases n by m}, we have $\cl_\tau(S_{X^5,M}(N) \cap u\M) \subseteq S_{X^7,M}(N)$. Hence, $\cl_\tau(S_{X^5,M}(N) \cap u\M)$ is quasi-compact by Lemma \ref{lemma: quasi-compactness from the old 3.21}.
Thus,  $\pi[\cl_\tau(S_{X^5,M}(N) \cap u\M)]$ is compact, and so is the closure of $f[X]$.

(5) {\em $f^{-1}[C] \subseteq X^{30}$.}

By Lemma \ref{lemma: key properties of F_n's}(4), $C \subseteq (\tilde{F}_{10} \cap u\M)/H(u\M)$. Hence,
\begin{flalign*}
& f^{-1}[C] =F^{-1}[\pi^{-1}[C]]\subseteq F^{-1}[(\tilde{F}_{10} \cap u\M) H(u\M)] \subseteq F^{-1}[\tilde{F}_{13} \cap u\M] \subseteq \\
&  F^{-1}[S_{X^{26},M}(N) \cap u\M] \subseteq X^{30},
\end{flalign*}
where the second inclusion follows from Lemma \ref{lemma: H(uM) form KrPi}, the third one from Lemma \ref{lemma: F_n in X^2n}, and the last one from Lemma \ref{lemma: u in F_1}(4).

(6) {\em For $U$ from (2) we have $f^{-1}[UC] \subseteq X^{34}$. Thus, the same holds for $U$ replaced by any (in particular by a compact) neighborhood of the neutral element in $u\M/H(u\M)$ contained in $U$.}

We have 
$$f^{-1}[UC] =F^{-1}[\pi^{-1}[UC]] \subseteq F^{-1}[S_{X^4,M}(N)S_{X^{20},M}(N)S_{X^6,M}(N)] = F^{-1}[S_{X^{30},M}(N)]\subseteq X^{34},$$
where the first inclusion follows from the choice of $U$ and Lemmas \ref{lemma: key properties of F_n's}(4) and \ref{lemma: H(uM) form KrPi}, and the last one from Lemma \ref{lemma: u in F_1}(4).

%%%Krzys25: Strengthened item (7) in order to refer to (7) (rather than to the proof of (7)) in the proof of Theorem 4.8.
(7) {\em For any compact $Z,Y \subseteq u\M/H(u\M)$ with $C^2 Y\cap C^2 Z = \emptyset$ there is a formula $\varphi(x)$ over $M$ such that $\hat{f}^{-1}[Y] \subseteq [\varphi(x)]$ and $\hat{f}^{-1}[Z] \subseteq [\neg \varphi(x)]$, and so the preimages $f^{-1}[Y]$ and $f^{-1}[Z]$ are separated by the definable set $\varphi(M)$.}
%{\em For any compact $Z,Y \subseteq u\M/H(u\M)$ with $C^2 Y\cap C^2 Z = \emptyset$ the preimages $f^{-1}[Y]$ and $f^{-1}[Z]$ can be separated by a definable set.}

%%%Krzys25: I added ``and clearly $\tau$-closed'', as it is needed to see that Claim 2 completes the proof.
By Lemma \ref{lemma: K and X}(1), $\pi^{-1}[Y]$ and $\pi^{-1}[Z]$ are quasi-compact and clearly $\tau$-closed. On the other hand, since $(\tilde{F}_7 \cap u\M)/H(u\M) \subseteq C$, we have that 
$$(\tilde{F}_7 \cap u\M)(\tilde{F}_7 \cap u\M) \pi^{-1}[Y] \cap (\tilde{F}_7 \cap u\M)(\tilde{F}_7 \cap u\M)\pi^{-1}[Z]=\emptyset.$$
So the following claim will complete the proof of (7) and of the whole theorem.

%%%Krzys25: Strengthened to give the strengthened version of (7). We also added the missing assumption that $Z,Y$ are $\tau$-closed (needed to apply Lemma 3.38).
\begin{clm}
%For any quasi-compact $Z,Y \subseteq u\M$ such that the sets $(\tilde{F}_7 \cap u\M)Y$ and $(\tilde{F}_5 \cap u\M)(\tilde{F}_7 \cap u\M)Z$ are disjoint the preimages $F^{-1}[Y]$ and $F^{-1}[Z]$ can be separated by a definable set.
For any $\tau$-closed, quasi-compact $Z,Y \subseteq u\M$ such that the sets $(\tilde{F}_7 \cap u\M)Y$ and $(\tilde{F}_5 \cap u\M)(\tilde{F}_7 \cap u\M)Z$ are disjoint the preimages $\hat{F}^{-1}[Y]$ and $\hat{F}^{-1}[Z]$ can be separated by a clopen $[\varphi(x)]\subseteq S_{G,M}(N)$ where $\varphi(x)$ is a formula with parameters from $M$.
\end{clm}

\begin{clmproof}
Let $\rho \colon S_{G,M}(N) \to S_G(M)$ be the restriction map. By the definition of the topologies on type spaces, $\rho$ is a continuous map. We claim that it is enough to show 
$$(*)\;\;\;\;\;\; \rho[\overline{\hat{F}^{-1}[Y]}] \cap \rho[\overline{\hat{F}^{-1}[Z]}] = \emptyset,$$
To see that ($*$) is enough, note that by Lemma \ref{lemma: K and X}(2) and \ref{lemma: u in F_1}(3) both $\overline{\hat{F}^{-1}[Y]}$ and $\overline{\hat{F}^{-1}[Z]}$ are contained in some $S_{X^n,M}(N)$. Hence, by ($*$) and compactness of $S_{X^n,M}(N)$,  $\rho[\overline{\hat{F}^{-1}[Y]}]$ and $\rho[\overline{\hat{F}^{-1}[Z]}]$ are disjoint closed subsets of $S_{X^n}(M)$, and so they can be separated by a basic open set $[\varphi(x)] \subseteq S_{X^n}(M)$ for some formula in $L_M$. 
%%%Krzys25: Adjusted sentence to the strengthened version of Claim 2.
%Then the definable set $\varphi(M)$ separates $F^{-1}[Y]$ and $F^{-1}[Z]$.
Taking the preimages under $\rho$, we see that the clopen $[\varphi(x)] \subseteq S_{G,M}(N)$ separates $\hat{F}^{-1}[Y]$ and $\hat{F}^{-1}[Z]$.

Let us prove ($*$). Suppose it fails, i.e. there are $p \in \overline{\hat{F}^{-1}[Y]}$ and $q \in \overline{\hat{F}^{-1}[Z]}$ such that $\rho(p)=\rho(q)$. So, taking $\alpha \models p$ and $\beta \models q$, we have $\alpha \equiv_M \beta$. 
%%%Krzys25: I added the next sentence. It is needed in the justification in parentheses in the following sentence.
Pick them so that $\tp(\alpha/N,\beta)$ is a coheir over $M$; then $\tp(\alpha,\beta/N)$ is a coheir over $M$.
Next, $\hat{F}(p) = upu = \tp(\gamma_1 \alpha \gamma_2/N)$ and $\hat{F}(q) = uqu = \tp(\gamma_1 \beta \gamma_2/N)$ for some $\gamma_1,\gamma_2 \models u$ (note that we can choose the same $\gamma_1, \gamma_2$ in both formulas: first we chose $\gamma_2 \models u$ such that $\tp(\alpha,\beta/N,\gamma_2)$ is a coheir over $M$, and then $\gamma_1 \models u$ such that $\tp(\gamma_1/N,\alpha,\beta,\gamma_2)$ is a coheir over $M$). Put $x := \gamma_1 \alpha \gamma_2$ and $y:= \gamma_1 \beta \gamma_2$. Using Lemma \ref{lemma: u in F_1}(1), we conclude that  $xy^{-1}= \gamma_1 \alpha \beta^{-1} \gamma_1^{-1} \in F_3$, so $x \in F_3y$. By Lemma \ref{lemma: claim in the notes}, this implies that $\hat{F}(p)=\tp(x/N)  \in (\tilde{F}_5\cap u\M)\tp(y/N) =  (\tilde{F}_5\cap u\M)\hat{F}(q)$. On the other hand, by Lemma \ref{lemma: main lemma for definability}, we have $\hat{F}(p) \in  (\tilde{F}_7\cap u\M)Y$ and $\hat{F}(q) \in  (\tilde{F}_7\cap u\M)Z$. Thus, we conclude that $\hat{F}(p)$ is in the intersection of $(\tilde{F}_7 \cap u\M)Y$ and $(\tilde{F}_5 \cap u\M)(\tilde{F}_7 \cap u\M)Z$, which contradicts the assumption of the claim.
\end{clmproof}
\end{proof}

\subsection{Around the main theorem}

In this subsection, we discuss some improvements or variants of Theorem \ref{theorem: main Theorem}.\\

{\bf Concrete numbers in the statement of Theorem \ref{theorem: main Theorem}.}\\

In \cite[Theorem 4.2]{Hru2}, Hrushovski produced a generalized definable locally compact model $f$ of $X$ with an error set $C$ such that $f^{-1}[C] \subseteq X^{12}$, while in our theorem $f^{-1}[C] \subseteq X^{30}$.

The proof of part (1) inside the proof of Theorem \ref{theorem: main Theorem} shows that $\error_r(f) \subseteq (\tilde{F}_3 \cap u\M)/H(u\M)$. Analogously, one can show that $\error_l(f) \subseteq (\tilde{F}_3 \cap u\M)/H(u\M)$. Therefore, if we dropped the definability requirement from the definition of generalized definable locally compact model (i.e. item (3) of Definition \ref{definition: gen. loc. comp. model}), then we could decrease our error set $C$ by taking $ \tilde{F}:=  ((\tilde{F}_3 \cap u\mathcal{M})/H(u\mathcal{M}))^{u\mathcal{M}/H(u\mathcal{M})}$ in place of $\tilde{F}:=  ((\tilde{F}_7 \cap u\mathcal{M})/H(u\mathcal{M}))^{u\mathcal{M}/H(u\mathcal{M})}$, and setting $C:=  \cl_{\tau}(\tilde{F}) \cup \cl_{\tau}(\tilde{F})^{-1}$ as before. After this modification, our proofs yield $f^{-1}[C] \subseteq X^{22}$ and $f^{-1}[UC] \subseteq X^{26}$.
A question is whether after this modification item (3) of Definition   \ref{definition: gen. loc. comp. model} still holds for some $l$ (maybe greater than 2). 
%The following question is behind it.
By the proof of part (7) in the proof of Theorem \ref{theorem: main Theorem}, it would hold with $l=4$ if the answer to the second question below was positive.

\begin{question}
\begin{enumerate}
\item Does $\widetilde{F}_n * \widetilde{F}_m = \widetilde{F}_{n+m}$?
\item Does $(\widetilde{F}_n \cap u\mathcal{M}) * (\widetilde{F}_m \cap u\mathcal{M}) = \widetilde{F}_{n+m} \cap u\mathcal{M}$?
\end{enumerate}
\end{question}

 And a final question is whether we could use yet smaller $C$ obtained by replacing $\tilde{F}:=  ((\tilde{F}_7 \cap u\mathcal{M})/H(u\mathcal{M}))^{u\mathcal{M}/H(u\mathcal{M})}$ by $\tilde{F}:=  ((\tilde{F}_1 \cap u\mathcal{M})/H(u\mathcal{M}))^{u\mathcal{M}/H(u\mathcal{M})}$. For this $C$ our proof would give us 
 %$f^{-1}[C] \subseteq X^{12}$ (as in Hrushovski's theorem).\\
 $f^{-1}[C] \subseteq X^{18}$.\\

 {\bf Definability over $\pmb{X}$}\\

In \cite[Theorem 4.2]{Hru2}, Hrushovski obtains separation by two sets definable over $X$, while we got separation by a set definable over $M$. However, assuming that our approximate subgroup $X$ is $\emptyset$-definable 
%%%Krzys25: We changed the comment in the parentheses.
%(in particular, the group operation is piecewise $\emptyset$-definable), 
(i.e. all the $X^n$'s and all the restrictions $\cdot|_{X^n \times X^n}$ are $\emptyset$-definable), 
it is not difficult to modify our error set $C$ to get separation by a set definable over $X$, and then we get separation by subsets of some $X^n$ which are definable over $X$ (see Remark \ref{remark: separation by two sets}). We explain the necessary modification of $C$ below.

Notice that, by $\emptyset$-definability of the approximate subgroup $X$, definability over $X$ is equivalent to definability over $G:=\langle X \rangle$.
Let us modify the definition of $C$ by replacing $F_n:=  \{ x_1y_1^{-1} \dots x_ny_n^{-1}: x_i, y_i \in \bar{G}\; \textrm{and}\; x_i \equiv_M y_i\; \textrm{for all}\; i\leq n\}$ by $F_n':= \{ x_1y_1^{-1} \dots x_ny_n^{-1}: x_i, y_i \in \bar{G}\; \textrm{and}\; x_i \equiv_G y_i\; \textrm{for all}\; i\leq n\}$. Then $\tilde{F}_n'$, $\tilde{F}'$, and $C'$ are defined using $F_n'$ in the same way as the corresponding objects without primes are defined at the beginning of Subsection \ref{subsection: main theorem}.

We claim that Theorem \ref{theorem: main Theorem} holds with $C$ replaced by $C'$ with the stronger conclusion that for any compact $Z,Y \subseteq u\M/H(u\M)$ with $C^2 Y\cap C^2 Z = \emptyset$ the preimages $f^{-1}[Y]$ and $f^{-1}[Z]$ can be separated by an $X$-definable set.

%%%Krzys25: Changed paragraph, as some other things also require minor adjustments.
%Since the sets with primes are supersets of the corresponding sets without primes, it is easy to see that the only things to check (in order to apply the whole argument from Subsection  \ref{subsection: main theorem}) are the following:
Since the sets with primes are supersets of the corresponding sets without primes, most of the statements from Subsection  \ref{subsection: main theorem} automatically imply the corresponding statements for sets with primes. However, the proofs of some of the statements require minor adjustments. Let us discuss three such situations. 

\begin{enumerate}
\item $F_1'= \{xy^{-1}: x,y \in \bar X \textrm{ with } x \equiv_G y\} \subseteq \bar X^2$;
\item  $(\tilde{F}_7' \cap u\M)^{u\M} \subseteq \tilde{F}_8' \cap u\M$;
\item ($*$) from the proof of part (7) in the proof of Theorem \ref{theorem: main Theorem} but for $\rho \colon S_{G,M}(N) \to S_G(M)$ replaced by $\rho' \colon S_{G,M}(N) \to S_G(G)$ being the restriction map.
\end{enumerate}

\begin{proof}
(1) The proof of Lemma \ref{lemma: F_n in X^2n} adapts, because the set $S$ from that proof is contained in $G$, and so $c \in G$.

(2) Since all types in $S_{G,M}(N)$ are finitely satisfiable in $G$, the proof of Lemma \ref{lemma: key properties of F_n's} adapts choosing $\alpha \models q$ so that $\tp(\alpha/N,a_{\leq 7},b_{\leq 7},\beta)$ is a coheir over $G$.

(3) Replacing $\rho$ by $\rho'$, the proof of ($*$) works as before (with $\alpha \equiv_G \beta$ in place of $\alpha \equiv_M \beta$).
\end{proof}

{\bf An error set for $\pmb{\hat{f}}$}\\

 Elaborating on the proof of part (1) in the proof of Theorem \ref{theorem: main Theorem}, we obtain the following, where $S_{G,M}(N)$ is equipped with its semigroup structure.
 
 \begin{proposition}\label{proposition: error of hat f}
The function $\hat{f} \colon S_{G,M}(N) \to u\mathcal{M}/H(u\mathcal{M})$ (given by $\hat{f}(p):= upu/H(u\mathcal{M})$) is a quasi-homomorphism with $\error_r(\hat{f}) \cup \error_l(\hat{f}) \subseteq (\widetilde{F}_5 \cap u\mathcal{M})/H(u\mathcal{M})$, and so
$$\hat{C}:= \cl\left( ((\widetilde{F}_5 \cap u\mathcal{M})/H(u\mathcal{M}))^{u\M/H(u\M)}\right) \cup  \cl\left( ((\widetilde{F}_5 \cap u\mathcal{M})/H(u\mathcal{M}))^{u\M/H(u\M)}\right)^{-1}$$
is a compact, normal, symmetric error set of $\hat{f}$.
 \end{proposition}
 
\begin{proof}
We will only explain how to prove that $\error_r(\hat{f}) \subseteq (\widetilde{F}_5 \cap u\mathcal{M})/H(u\mathcal{M})$. The proof that $\error_l(\hat{f}) \subseteq (\widetilde{F}_5 \cap u\mathcal{M})/H(u\mathcal{M})$ is similar. The rest follows as at the beginning of the proof of Theorem  \ref{theorem: main Theorem}, using an obvious variant of Lemma \ref{lemma: key properties of F_n's}(4).

%In fact, we will show a stronger statement that $\error_r(\hat{f}) \subseteq (\widetilde{E}_5 \cap u\mathcal{M})/H(u\mathcal{M})$, where 
%\begin{flalign*}
%&\:E_n:=  \{ x_1y_1^{-1} \dots x_ny_n^{-1}: x_i, y_i \in \bar{G}\; \textrm{and}\; x_i \equiv_N y_i\; \textrm{for all}\; i\leq n\}\\
%&\: \tilde{F}_n:=  \{ \tp(a/N) \in S_{G,M}(N): a \in E_n\}\\
%\end{flalign*}

We need to show that $(uqu)^{-1}(upu)^{-1}pqu \subseteq \tilde{F}_5$ for all $p,q \in S_{G,M}(N)$. We have $pqu= \tp(g'h'\alpha/N)$ for some $g' \models p$, $h' \models q$, and $\alpha \models u$. An obvious extension of Claim 1 in the proof of Theorem  \ref{theorem: main Theorem} yields
\begin{flalign*}
&\: (upu)^{-1} = \tp(xy^{-1}g^{-1}/N) \textrm{ for some } x \equiv_N y \text{ and }  g \models p,\\
&\: (uqu)^{-1} = \tp(zt^{-1}h^{-1}/N) \textrm{ for some } z \equiv_N t \text{ and }  h \models q.
\end{flalign*} 

Looking at the proof of the aforementioned Claim 1, we can choose all the above data so that 
%$\tp(x,y,g/N,g',h',\alpha)$ and $\tp(z,t,h/N,x,y,g,g',h',\alpha)$ are coheirs over $N$. The $\tp(h/N,x,y,g,g')$ is a coheir over $N$ and  $$(uqu)^{-1}(upu)^{-1}pqu =\tp(zt^{-1}h^{-1}xy^{-1}g^{-1}g'h'\alpha/N).$$
$\tp(t/N,x,y,g,g',h',\alpha)$, $\tp(h/N,t,x,y,g,g',h',\alpha)$, $\tp(zt^{-1}h^{-1}/N,x,y,g,g',h',\alpha)$, and $\tp(xy^{-1}g^{-1}/N,g',h',\alpha)$ are all coheirs over $M$. Then,
\begin{flalign*}
&\: (uqu)^{-1}(upu)^{-1}pqu =\tp(zt^{-1}h^{-1}xy^{-1}g^{-1}g'h'\alpha/N)=\\
&\: \tp(zt^{-1}x^{h^{-1}}(y^{h^{-1}})^{-1}(g^{h^{-1}})^{-1}g'^{h^{-1}}h^{-1}h'\alpha /N) \in \tilde{F}_5,
\end{flalign*}
because $z \equiv_M t$, $x^{h^{-1}} \equiv_M y^{h^{-1}}$ (as $x \equiv_M y$ and $\tp(h/M,x,y)$ is a coheir over $M$), $g^{h^{-1}} \equiv_M g'^{h^{-1}}$ (as $g \equiv_M g'$ and $\tp(h/M,g,g')$ is a coheir over $M$), $h \equiv_M h'$, and $\alpha \in F_1$.
\end{proof}
 
%%%Krzys25: Added paragraph.
In fact, all the steps of the proof of  Theorem \ref{theorem: main Theorem} can be stated for $\hat{f}$ in place of $f$ and $S_{X,M}(N)$ in place of $X$ (step (7) was already stated in this way to be able to apply it directly in Section \ref{section: universality}). The proofs are essentially  the same, using Lemma \ref{lemma: u in F_1}(3) instead of Lemma \ref{lemma: u in F_1}(4). Of course, in item (7) we have to work with the same $C$ as in Theorem \ref{theorem: main Theorem} (which clearly contains $\hat{C}$ defined above). So this shows that $\hat{f}$ is something that we could call a ``generalized definable locally compact model of $S_{X,M}(N)$''.\\

%%%Krzys: I have significantly modified and shortened the "subsubsection" below, and changed the title. However, I have not added yet anything about the new goal that you suggested to develop the general theory of actions of topological groups generated by open approximate subgroups on locally compact flows. Regrading this theory, I think the only new thing is to define a suitable replacement of the flow of external types to get a universal topological locally compact ambit with suitable properties (and this should be doable by dividing by infinitesimals as was done for compact flows). Another point is to prove a variant of Theorem 3.35 with a notion of "topological" generalized locally compact model which takes into account topology on $G$. But maybe you had something else in mind. Please send me the text for an addition "subsubsection" on this topic if you still want to include it.

{\bf Expressing the proof in terms of Boolean algebras}\\

%%%Krzys25: The next sentence is added.
The point of this short subsection is to indicate that the above construction could be also performed in two other (related) contexts -- a more set-theoretic and a more combinatorial one -- which are briefly outlined in the next two paragraphs.

Suppose $X$ is an abstract (rather than definable) approximate subgroup $X$ and one is interested in finding a generalized locally compact model of $X$. Taking $M:=G=\langle X \rangle$ equipped with the full structure,
% we have $S_{G,\ext}(M)=S_G(M)= \beta G$ (the space of ultrafilters on the Boolean algebra of all subsets of $G$) and 
$S_{G,M}(N)$ becomes the subspace of $\beta G$ (the space of ultrafilters on the Boolean algebra of all subsets of $G$) which consists of the ultrafilters concentrated on some $X^n$ (for varying $n$). So no model theory is involved in those objects. In this situation, one should be able to completely eliminate model theory from our construction of the generalized locally compact model 
%%%Krzys25: I modified the rest of this paragraph.
%by not using realizations of types, but we find it unnatural and more technical, so we will not do that.
by using the classical formula for the convolution product on $\beta G$ (recalled in the last paragraph of Subsection \ref{subsection: topological dynamics}) in place of the one in terms of realizations of types, but we find it more technical and less intuitive. And this context does not give any meaningful definability property of the resulting generalized locally compact model.

 After the first author's conference talk on Theorem \ref{theorem: main Theorem}, Sergei Starchenko suggested that it could be interesting to modify our construction of the generalized definable locally compact model of a definable approximate subgroup $X$ by replacing the Boolean algebra generated by externally definable subsets of $G$ by a smaller (or even smallest possible) Boolean algebra which does not refer to model theory; then ultrafilters on this algebra would be used in place of complete external types.  One should be able to realize this suggestion by using Newelski's work on $d$-closed $G$-algebras \cite{New1}. Namely, Newelski showed that whenever $\mathcal{A}$ is a $d$-closed $G$-algebra of subsets of $G$, then there is an explicitly given semigroup operation $*$ on the Stone space $S(\mathcal{A})$ which extends the action of $G$ and is left continuous. Now, in our situation of an approximate subgroup $X$ and $G:=\langle X \rangle$, take $\mathcal{A}$ to be the $d$-closure (in the sense of \cite{New1}) of the $G$-algebra generated by all left translates of the sets $X, X^2,X^3,\dots$. Let $S_G(\mathcal{A})$ be the subflow of $S(\mathcal{A})$ which consists of the ultrafilters containing one of the $X^n$'s (for varying $n$). Then the above semigroup operation $*$ on $S(\mathcal{A})$ restricts to a left continuous semigroup operation on $S_G(\mathcal{A})$, and $S_{G}(\mathcal{A})$ is locally compact. One should be able to adapt the theory developed in this paper for $S_{G,M}(N)=S_{G, \ext}(M)$ to $S_G(\mathcal{A})$; in particular, to state and prove a suitable variant of Theorem \ref{theorem: main Theorem} yielding a generalized locally compact model of $X$ which satisfies a version of definability (i.e. item (3) of Definition \ref{definition: gen. loc. comp. model})  in which separation by a definable set is replaced by separation by a set from the $G$-algebra $\mathcal{A}$ 
%%%Krzys25: I replaced ``left and right'' by ``right'' as this is enough (at least is seems so looking at the proof). No,  I resigned from this change, as it seems that in step (7) we need to work with F_3 in the sense of left and right translates. It works after a suitable modification of the error set, just using powers of $X$ (see my red comments in the printed file.)
%(or even from the Boolean algebra of subsets of $G$ generated by left and right translates of $X,X^2,\dots$).
(or even from the Boolean algebra of subsets of $G$ generated by right translates of $X,X^2,\dots$).

\section{Universality}\label{section: universality}

We will prove that the generalized definable locally compact model from Theorem \ref{theorem: main Theorem} is an initial object in a certain category. In particular, this will explain what it means to be a generalized definable locally compact model in terms of factorization through $u\M/H(u\M)$ (with the notation from Section \ref{section: main theorem}).

As in  Section  \ref{section: main theorem}, take the situation and notation as at the end of Subsection \ref{subsection: model theory}.  We introduce the notion of good quasi-homomorphism which will be used to define morphisms in our category.

\begin{definition}\label{definition: good quasi-homomorphism}
%%%Krzys I changed the definition as you suggested.
%Let $f \colon G \to H:S$ be a generalized definable locally compact model of $X$. A {\em good quasi-homomorphism for $f$}
%\footnote{The notion of good quasi-homomorphism depends only on the pair $(H,S)$ (i.e. the target of $f$ with a distinguished error set $S$) rather than on $f$.} 
Let $H$ be a locally compact group and $S$ a compact, normal, symmetric subset of $H$. A {\em good quasi-homomorphism for $(H,S)$}
is a quasi-homomorphism $h \colon H \to L :T$ for some compact, normal, symmetric subset $T$ of a locally compact group $L$ such that:
\begin{enumerate}
\item for every compact $Y \subseteq L$, $h^{-1}[Y]$ is relatively compact in $H$;
\item for every compact $V \subseteq H$, $h[V]$ is relatively compact in $L$;
\item $h[S] \subseteq T^n$ for some $n \in \mathbb{N}$; %(so for every $m\in \mathbb{N}$ there is $n_m\in \mathbb{N}$ with $h[S^m] \subseteq C^{n_m}$);
\item there is  $m \in \mathbb{N}$ such that for any compact $Y,Z \subseteq L$ with $T^{m} Y\cap T^{m} Z = \emptyset$, $S \cl(h^{-1}[Y]) \cap S \cl(h^{-1}[Z]) = \emptyset$.
\end{enumerate} 
%Let $f \colon G \to H:S$ be a generalized definable locally compact model of $X$. By a {\em good quasi-homomorphism for $f$} we mean any good quasi-homomorphism for the pair $(H,S)$.
\end{definition}

%%%Krzys: I adjusted the statement to the change of Def. 4.1.
\begin{remark}\label{remark: on good quasi-homomorphisms}
%Let $f \colon G \to H:S$ be a generalized definable locally compact model of $X$ and 
Let $(H,S)$ be as in Definition \ref{definition: good quasi-homomorphism} and
let $h \colon H \to L :T$ be a good quasi-homomorphism for $(H,S)$. Then:
\begin{enumerate}
\item  for every $m\in \mathbb{N}$ there is $n_m\in \mathbb{N}$ with $h[S^m] \subseteq T^{n_m}$;
\item for every $n \in \mathbb{N}$ there exists $m_n \in \mathbb{N}$ such that for any compact $Y,Z \subseteq L$ with $T^{m_n} Y\cap T^{m_n} Z = \emptyset$ we have $S^n \cl(h^{-1}[Y]) \cap S^n \cl(h^{-1}[Z]) = \emptyset$.
\end{enumerate}
\end{remark}

\begin{proof}
(1) follows by an easy induction from item (3) of Definition \ref{definition: good quasi-homomorphism} and the assumption that $T$ is an error set of $h$.

(2) 
%We prove it by induction on $n$. The base step is item (4) of  Definition \ref{definition: good quasi-homomorphism}. 
By (1) and the assumption that $T$ is an error set of $h$, we get that $h[S^{n-1} h^{-1}[Y]] \subseteq T^{k_n}Y$ and $h[S^{n-1} h^{-1}[Z]] \subseteq T^{k_n}Z$ for some $k_n$. We will show that $m_n: =k_n +m$ works for any $m$ satisfying the conclusion of item (4) of Definition \ref{definition: good quasi-homomorphism}. For that assume that $T^{m_n}Y \cap T^{m_n}Z = \emptyset$. Then $T^m(T^{k_n}Y) \cap T^m(T^{k_n}Z) = \emptyset$ and $T^{k_n}Y$ and $T^{k_n}Z$ are compact. Hence, $S \cl(h^{-1}[T^{k_n}Y]) \cap S \cl(h^{-1}[T^{k_n}Z]) = \emptyset$ by the choice of $m$.  
%%%Krzys25: Added justification in parentheses, as requested by the referee.
%As the last intersection contains $S^{n}\cl(h^{-1}[Y]) \cap S^{n}\cl(h^{-1}[Z])$, we get that $S^n \cl(h^{-1}[Y]) \cap S^n \cl(h^{-1}[Z]) = \emptyset$.
 As the last intersection contains $S^{n}\cl(h^{-1}[Y]) \cap S^{n}\cl(h^{-1}[Z])$ (which follows from the choice of $k_n$ and the obvious inclusion $A\cl(B) \subseteq \cl(AB)$), we get that $S^n \cl(h^{-1}[Y]) \cap S^n \cl(h^{-1}[Z]) = \emptyset$.
\end{proof}

%%%Krzys: I adjusted the definition to the change of Def. 4.1.
\begin{definition}\label{definition: morphism}
Let $f \colon G \to H:S$ and $h \colon G \to L:T$ be definable generalized locally compact models of $X$. 
%%%Krzys25:  I replaced ``where'' by ``and'' as the referee suggested, although ``where'' was also correct.
A {\em morphism} from $f$ to $h$ is a function $\rho \colon H \to L$ which is a good quasi-homomorphism $\rho \colon H \to L : T^k$ for $(H,S)$ and $k \in \mathbb{N}$ is such that $\rho(f(g)) \in h(g)T^k$ for all $g \in G$. The set of morphisms from $f$ to $h$ will be denoted by $\Mor(f,h)$.
\end{definition}

%%%Krzys25: ``generalized definable locally compact models of $X$'' instead of ``definable locally compact models of $X$ in the generalized sense''
\begin{remark}\label{remark: morphisms yield a category}
Morphisms are closed under composition, and so the class of generalized definable locally compact models of $X$ with morphisms forms a category.
\end{remark}

\begin{proof}
Let $f_1 \colon G \to H_1:S_1$, $f_2 \colon G \to H_2:S_2$, and $f_3 \colon G \to H_3:S_3$ be generalized definable locally compact models, and $\rho \in \Mor(f_1,f_2)$, $\delta \in \Mor(f_2,f_3)$. The goal is to show that there is $k$ such that $\error_r(\delta \rho) = \{\delta (\rho(y))^{-1}\delta (\rho(x))^{-1} \delta (\rho(xy)): x,y \in G\} \subseteq S_3^k$ and $\delta(\rho(f_1(g))) \in f_3(g)S_3^k$ for all $g \in G$. 
%%%Krzys: $(H_1,S_1)$ in place of $f_1$ due to the change of Def. 4.1.
Indeed, once we prove it, the fact that $\delta\rho \colon H_1 \to H_3:S_3^k$ is a good quasi-homomorphism  for $(H_1,S_1)$ easily follows using Remark \ref{remark: on good quasi-homomorphisms} which we leave as an exercise.

%%%Krzys: $(H_2,S_2)$ in place of $f_2$ due to the change of Def. 4.1.
Let $k_1$ and $k_2$ be numbers witnessing that $\rho$ and $\delta$ are morphisms, respectively, and let $n_{k_1}$ be the number from Remark \ref{remark: on good quasi-homomorphisms}(1) applied to the good quasi-homomorphism $\delta \colon H_2 \to H_3 : S_3^{k_2}$ for $(H_2,S_2)$.  

Regarding the first part of our goal, we have 
\begin{flalign*}
&\: \delta (\rho(y))^{-1}\delta (\rho(x))^{-1} \delta (\rho(xy)) \in S_3^{k_2} \delta(\rho(x)\rho(y))^{-1}  \delta (\rho(xy)) \subseteq S_3^{k_2+2k_2} \delta((\rho(x)\rho(y))^{-1})\delta (\rho(xy)) \subseteq \\
&\: S_3^{k_2+2k_2+k_2} \delta(\rho(y)^{-1}\rho(x)^{-1} \rho(xy)) \subseteq S_3^{4k_2}\delta[S_2^{k_1}] \subseteq S_3^{4k_2+k_2n_{k_1}}.
\end{flalign*}
%
%%%Krzys: f_3(g) was missing at the end of the formula below.
Regarding the second part of our goal, we have 
\begin{flalign*}
\delta(\rho(f_1(g))) \in \delta[f_2(g)S_2^{k_1}] \subseteq \delta(f_2(g)) \delta[S_2^{k_1}] S_3^{k_2} \subseteq f_3(g) S_3^{k_2+k_2n_{k_1}+ k_2}=f_3(g)S_3^{2k_2 +k_2n_{k_1}}.
\end{flalign*}
We conclude that 
%$k:= \max (k_2+3+k_2n_{k_1}, 2k_2 +k_2n_{k_1})$ works.
$k:= 4k_2+k_2n_{k_1}$ works.
\end{proof}

The obtained category will be later modified to get that the generalized definable locally compact model from Theorem \ref{theorem: main Theorem} is an initial object, as for the above category we will only obtain existence of a morphism and ``approximate uniqueness''. Before going to these main issues, let us make one more basic observation.

\begin{proposition}
%%%Krzys: $(H,S)$ in place of $f$ due to the change of Def. 4.1.
Let $f \colon G \to H:S$ be a generalized definable locally compact model of $X$ and let $h \colon H \to L :T$ be a good quasi-homomorphism for $(H,S)$. Then there is $n \in \mathbb{N}$ such that $h\circ f \colon G \to L: T^n$ is a generalized definable locally compact model of $X$ and $h \in \Mor(f, h\circ f)$.
\end{proposition}

\begin{proof}
%The fact $h\circ f \colon G \to L: T^n$ is a quasi-homomorphism for some $n$ can be proved by a very similar argument to the two-line computation in the proof of Remark \ref{remark: morphisms yield a category}.
The fact $h\circ f \colon G \to L: T^n$ is a quasi-homomorphism for some $n$ follows from:
\begin{flalign*}
&\: h(f(y))^{-1}h(f(x))^{-1}h(f(xy)) \in Th(f(x)f(y))^{-1}h(f(xy)) \subseteq T^3h((f(x)f(y))^{-1})h(f(xy)) \subseteq \\
&\:  T^4h(f(y)^{-1}f(x)^{-1}f(xy)) \subseteq T^4h[S] \subseteq T^{4+n},
\end{flalign*}
where $n$ is a number witnessing item (3) of Definition \ref{definition: good quasi-homomorphism} applied to the good quasi-homomorphism $h$.

To check item (1) of Definition \ref{definition: gen. loc. comp. model} for $h\circ f$, consider any compact $V \subseteq L$. Then $h^{-1}[V]$ is relatively  compact, so $(h\circ f)^{-1}[V]=f^{-1}[h^{-1}[V]] \subseteq X^i$ for some $i$.

To see item (2) of Definition \ref{definition: gen. loc. comp. model}, note that $\cl(f[X])$ being compact implies that $h[\cl(f[X])]$ is relatively compact, and so $\cl(h[f[X]])$ is compact.

%%%Krzys: I added ``choose $l$ witnessing item (3) of Definition \ref{definition: gen. loc. comp. model} for $f$.'' Then I replaced $S^2 \cl(h^{-1}[Y]) \cap S^2 \cl(h^{-1}[Z]) = \emptyset$ by $S^l \cl(h^{-1}[Y]) \cap S^l \cl(h^{-1}[Z]) = \emptyset$.
To see that item (3) of Definition \ref{definition: gen. loc. comp. model} holds for $h \circ f$, choose $l$ witnessing item (3) of Definition \ref{definition: gen. loc. comp. model} for $f$.
Next, choose $n$ so big that $T^n$ is an error set of $h\circ f$ (the existence of such an $n$ was justified at the beginning of the proof) and for any compact $Y,Z \subseteq L$ with $T^{n} Y\cap T^{n} Z = \emptyset$, $S^l \cl(h^{-1}[Y]) \cap S^l \cl(h^{-1}[Z]) = \emptyset$ (the existence of such an $n$ is guaranteed by Remark \ref{remark: on good quasi-homomorphisms}(2)). Then for any  compact $Y,Z \subseteq L$ with $T^{n} Y\cap T^{n} Z = \emptyset$ we have that $(h \circ f)^{-1}[Y]$ and $(h \circ f)^{-1}[Z]$ can be separated by a definable set.

The fact that  $h \in \Mor(f, h\circ f)$ is trivial.
\end{proof}

\begin{theorem}\label{theorem: universality: existence}(Universality of $f \colon G \to u\mathcal{M}/H(u\mathcal{M})$: existence) Let $f \colon G \to  u\mathcal{M}/H(u\mathcal{M}) : C$ be the generalized definable locally compact model of $X$ from Theorem \ref{theorem: main Theorem}. Let $h \colon G \to H :S$ be an arbitrary generalized definable locally compact model of $X$. Then there exists a morphism $\widetilde{h} \in \Mor(f,h)$. 

%More precisely, define $\bar h \colon S_{G,M}(N) \to H$ by picking $\bar h(p)$ arbitrarily from the set $\bigcap_{\varphi(x) \in p|_M} \cl(h[\varphi(G)])$. Next, define $\widetilde{h} \colon u\mathcal{M}/H(u\mathcal{M}) \to H$ by picking $\widetilde{h}(p/ H(u\mathcal{M}))$ arbitrarily from the set $\bar h[pH(u\mathcal{M})]$. Define also $h^* \colon G^* \to H$ by picking $h^*(a)$ from $\bigcap_{\varphi(x) \in \tp(a/M)}  \cl(h[\varphi(G)])$ so that $\bar h(p)=h^*(a)$ for any $a \models p$. Then $h^*,\bar h,\widetilde{h}$ are quasi-homomorphism with the distinguished error set being $C^n$ for some natural number $n$ depending on $k,l$ from Definition \ref{definition: main definition} applied for $h$, and $\widetilde {h} \in \Mor(f,h)$.
More precisely, define $h_M \colon S_{G}(M) \to H$ by picking $h_M(p)$ arbitrarily from the set $\bigcap_{\varphi(x) \in p} \cl(h[\varphi(G)])$. Next, define $h^* \colon \bar G \to H$ by $h^*(a) := h_M(\tp(a/M))$, $\bar h \colon S_{G,M}(N) \to H$ by $\bar h(p):=h_M(p|_M)$, and finally $\widetilde{h} \colon u\mathcal{M}/H(u\mathcal{M}) \to H$ by picking $\widetilde{h}(p/ H(u\mathcal{M}))$ arbitrarily from the set $\bar h[pH(u\mathcal{M})]$. Then $h^*,\bar h,\widetilde{h}$ are quasi-homomorphisms with the distinguished error set being $S^n$ for some $n \in \mathbb{N}$ 
%%%Krzys25: Correction below (actually using an earlier version).
depending only on $l$ from item (3) of Definition \ref{definition: gen. loc. comp. model} applied for $h$, and $\widetilde {h} \in \Mor(f,h)$.
%independent of the choice of $h$, and $\widetilde {h} \in \Mor(f,h)$.
\end{theorem}

\begin{proof}
%%%Krzys: As in the proof of Prop 4.5: In the first version, I assumed that definability of the generalized locally compact model $h$ is witnessed by $l=2$, which was wrong. So now I used arbitrary $l$. Hence, now all the powers of $S$ in the proof below depend on $l$.
The proof is divided into parts. Items (1), (2), (6) below show that  $h^*,\bar h,\widetilde{h}$ are quasi-homomorphisms with suitable error sets. Items (6)-(11) show that $\widetilde {h} \in \Mor(f,h)$.

%%%Krzys: The next sentence is new.
Take $l \in \mathbb{N}$ witnessing item (3) of Definition  \ref{definition: gen. loc. comp. model} for $h$.

(1) {\em $h^*\colon \bar G \to H : S^{4l+1}$.}

We have 
\begin{flalign*}
&\: h^*(ab) \in \bigcap_{\varphi(x) \in \tp(a/M), \psi(x) \in \tp(b/M)} \overline{h[\varphi(G) \cdot \psi(G)]} \subseteq \bigcap_{\varphi(x), \psi(x)} \overline{h[\varphi(G)] h[\psi(G)]S} = \\
&\:  \bigcap_{\varphi(x), \psi(x)} \left(\overline{h[\varphi(G)]} \cdot \overline{h[\psi(G)]}\cdot S\right) = \bigcap_{\varphi(x)} \overline{h[\varphi(G)]} \cdot  \bigcap_{\psi(x)} \overline{h[\psi(G)]} \cdot S,
\end{flalign*}
where $\varphi(x)$ ranges over $\tp(a/M)$ and $\psi(x)$ over $\tp(b/M)$. (The last two equalities follow from compactness of  $\overline{h[\varphi(G)]}$, $\overline{h[\psi(G)]}$, and $S$ for sufficiently small $\varphi(G)$ and $\psi(G)$.) Therefore,
$$(*)\;\;\;\;\;\; h^*(ab) =\alpha \beta \gamma$$ 
for some $\alpha \in \bigcap_{\varphi(x) \in \tp(a/M)} \overline{h[\varphi(G)]}$,  $\beta \in \bigcap_{\psi(x) \in \tp(b/M)} \overline{h[\psi(G)]}$, and $\gamma \in S$.

We claim that 
$$(**)\;\;\;\;\;\; S^l\alpha \cap S^lh^*(a) \ne \emptyset.$$
%%%Krzys: I added "and local compactness of $H$".
Suppose not. By compactness of $S$ and local compactness of $H$, there are compact neighborhoods $F_1$ of $\alpha$ and $F_2$ of $h^*(a)$ such that $S^l F_1 \cap S^l F_2 =\emptyset$. 
%%%Krzys: "by the choice of $l$" in place of "by definability of $h$".
Then, by the choice of $l$, there is a formula $\theta(x) \in L_M$ such that $h^{-1}[F_1] \subseteq \theta(M)$ and $ h^{-1}[F_2] \subseteq G \setminus \theta(M)$. 
If $\theta(x) \in \tp(a/M)$, then $h^*(a) \in \overline{h[\theta(M)]} \subseteq \overline{F_2^c}$ which contradicts the fact that $F_2$ is a neighborhood of $h^*(a)$. If $\neg \theta(x) \in \tp(a/M)$, then $\alpha \in \overline{h[(\neg \theta(M)]} \subseteq \overline{F_1^c}$ which contradicts the fact that $F_1$ is a neighborhood of $\alpha$. So ($**$) has been proved. Analogously,  $S^l\beta \cap S^lh^*(b) \ne \emptyset$. From these two observations and ($*$), we get $h^*(ab) \in S^{2l}h^*(a)S^{2l}h^*(b)S=S^{4l+1}h^*(a)h^*(b)$. Thus, (1) has been proved.

By (1) and the definition of the semigroup operation $*$ on $S_{G,M}(N)$, we immediately get

(2) {\em $\bar h \colon S_{G,M}(N) \to H:S^{4l+1}$.} 

(3) {\em $\bar h(ugu) \in h(g)S^{4(4l+1)}$.}

Item (3) follows by the following computation
%%%Krzys25: $\in$ instead of $=$. And added ``and equation $u^2=u$''.
$$h(g)^{-1}\bar h(ugu) \in h(g)^{-1} \bar h(u) h(g) \bar h(u) S^{2(4l+1)} \subseteq h(g)^{-1}S^{4l+1} h(g) S^{4l+1} S^{2(4l+1)}= S^{4(4l+1)},$$
which uses (2), the fact that $\bar h \!\upharpoonright_{G} = h$ (which follows from the formulas for $\bar h$ and $h_M$), and the equation $u^2=u$.

(4) {\em For every compact $V \subseteq  H$, $\bar h^{-1}[V] \subseteq S_{X^i,M}(N)$ for some $i$, and so $\bar h^{-1}[V] \cap u\M$ is relatively quasi-compact in the $\tau$-topology.}

Choose a compact neighborhood $U$ of $V$. We have $h^{-1}[U] \subseteq X^i$ for some $i$, and we check that $i$ is good. If not, there is $p \in S_{G,M}(N) \setminus S_{X^i,M}(N)$ with $\bar h(p) \in V$. Then $\bar h(p) \in \overline{h[G \setminus X^i]} \subseteq \overline{U^c}$ which is disjoint from $V$ as $U$ is a neighborhood of $V$, a contradiction. This implies that $\bar h^{-1}[V] \cap u\M$ is relatively quasi-compact in the $\tau$-topology 
%%%Krzys25: Referred to the new lemmas.
%by the argument in Claim 2 of the proof of Proposition \ref{proposition: quasi local compactness} and the final paragraph of the proof of that proposition.
by Lemmas \ref{lemma: tau-cl increases n by m} and \ref{lemma: quasi-compactness from the old 3.21}.

(5) {\em $\bar h[S_{X^i,M}(N)]$ is relatively compact for every $i \in \mathbb{N}$.}

This is immediate from the fact that $\bar h[S_{X^i,M}(N)] \subseteq \overline{h[X^i]}$ and $\overline{h[X^i]}$ is compact.

In order to show that $\widetilde {h} \in \Mor(f,h)$, we will use the above observations and the following claims.

\begin{clm}
\begin{enumerate}
\item[(i)] $h^*(a^{-1}) \in h^*(a)^{-1} S^{2(4l+1)}$.
\item[(ii)] $h^*(ab^{-1}) \in S^{3(4l+1)}$ for every $a \equiv_M b$.
\item[(iii)] $h^*[F_n] \subseteq S^{(4n-1)(4l+1)}$.
\end{enumerate}
\end{clm}

\begin{clmproof}
(i) follows from (1). The computation in (ii) is as follows: $h^*(ab^{-1}) \in h^*(a)h^*(b^{-1})S^{4l+1} \subseteq h^*(a)h^*(b)^{-1}S^{3(4l+1)} = h^*(a)h^*(a)^{-1}S^{3(4l+1)}=S^{3(4l+1)}$, where we used (1), (i), and the definition of $h^*$. Finally, (iii) follows from (ii) and (1).
\end{clmproof}

%%%Krzys25: Added ``for all $p \in u\M$'' as suggested by the referee.
\begin{clm}
$\tilde{h} (p/H(u\M)) \in \bar h(p)S^{12(4l+1)}$ for all $p \in u\M$.
\end{clm}

\begin{clmproof}
By the definition of $\tilde{h}$,  $\tilde{h}(p/H(u\M)) = \bar h(q)$ for some $q \in u\mathcal{M}$ and $r \in H(u\M)$ such that $q=pr$. By Lemma \ref{lemma: H(uM) form KrPi}, $H(u\M) \subseteq \tilde{F}_3 \cap u\M$, so, by Claim 1(iii), $\bar h[H(u\M)] \subseteq S^{11(4l+1)}$.  Therefore, by (2), $\tilde{h}(p/H(u\M))=\bar h(q) \in \bar h(p) \bar h(r) S^{4l+1} \subseteq \bar h(p) S^{12(4l+1)}$.
\end{clmproof}

(6) {\em $\tilde{h} \colon u\M/H(u\M) \to H : S^{37(4l+1)}$}

This follows from (2) and Claim 2. Namely, we have: 
$\tilde{h}(p/H(u\M) \cdot q/H(u\M)) = \tilde{h}(pq/H(u\M)) \in \bar h(pq) S^{12(4l+1)} \subseteq \bar{h}(p)\bar{h}(q) S^{13(4l+1)} \subseteq \tilde{h}(p/H(u\M)) \tilde{h}(q/H(u\M))S^{37(4l+1)}$.

(7) {\em $\tilde{h}(f(g)) \in h(g)S^{16(4l+1)}$.}

This follows from (3) and Claim 2. Namely:  $\tilde{h}(f(g)) = \tilde{h}(ugu/H(u\M)) \in \bar h(ugu)S^{12(4l+1)} \in h(g)S^{4(4l+1)}S^{12(4l+1)}= h(g)S^{16(4l+1)}$.

(8) {\em  For every compact $V \subseteq  H$, $\tilde{h}^{-1}[V]$ is relatively compact.}

%This follows from (4). Namely,  $\tilde{h}^{-1}[V] =\pi[\bar{h}^{-1}[V]]$ (where $\pi \colon u\M \to u\M/H(u\M)$ is the quotient map). By (4), $\cl_\tau(\bar h^{-1}[V])$ is quasi-compact. Therefore, $\cl(\tilde{h}^{-1}[V]) =\pi[\cl_\tau(\bar{h}^{-1}[V])]$ is compact.
This follows from (4). Namely, by the definition of $\tilde{h}$, $\tilde{h}^{-1}[V] \subseteq \pi[\bar{h}^{-1}[V]\cap u\M]$ (where $\pi \colon u\M \to u\M/H(u\M)$ is the quotient map). By (4), $\cl_\tau(\bar h^{-1}[V]\cap u\M)$ is quasi-compact, and so $\cl ( \pi[\bar{h}^{-1}[V]\cap u\M])=\pi[\cl_\tau(\bar{h}^{-1}[V]\cap u\M)]$ is compact. Therefore, $\cl (\tilde{h}^{-1}[V])$ being a closed subset of $\cl ( \pi[\bar{h}^{-1}[V]\cap u\M])$ is also compact.

(9) {\em For every compact $V \subseteq u\M/H(u\M)$, $\tilde{h}[V]$ is relatively compact in $H$.}

%%%Krzys: I also referred to Lemma 3.30.
%%%Krzys25: The reference to 3.35 (the old 3.33) was redundant.
By the definition of $\tilde{h}$, $\tilde{h}[V] \subseteq \bar h[\pi^{-1}[V]]$. By Lemma \ref{lemma: K and X},  the set $\pi^{-1}[V]$  is contained in some $S_{X^i,M}(N)$. Thus, by (5),  $\bar h[\pi^{-1}[V]]$ is relatively compact, and so  $\tilde{h}[V]$ is relatively compact, too.

(10) {\em $\tilde{h}[C] \subseteq S^{51(4l+1)}$.}

By Lemma \ref{lemma: key properties of F_n's}(4), $C \subseteq (\tilde{F}_{10} \cap u\M)/H(u\M)$. Hence, using Claim 2, $\tilde{h}[C] \subseteq \bar h[\tilde{F}_{10}]S^{12(4l+1)}$. By Claim 1(iii),  $\bar h[\tilde{F}_{10}] \subseteq S^{39(4l+1)}$. Therefore,  $\tilde{h}[C] \subseteq S^{51(4l+1)}$.

(11) {\em There is  $m \in \mathbb{N}$ such that for any compact $Y,Z \subseteq H$ with $S^{m} Y\cap S^{m} Z = \emptyset$, $C \cl(\tilde{h}^{-1}[Y]) \cap C \cl(\tilde{h}^{-1}[Z]) = \emptyset$}.

By Lemma \ref{lemma: key properties of F_n's}(4), $C \subseteq (\tilde{F}_{10} \cap u\M)/H(u\M)$, so it is enough to find $m$ such that 
$$(!) \;\;\;\;\;\; S^{m} Y\cap S^{m} Z = \emptyset$$
implies  
$$(\dagger)\;\;\;\;\;\; (\tilde{F}_{10} \cap u\M)/H(u\M)\cl(\tilde{h}^{-1}[Y]) \cap  (\tilde{F}_{10} \cap u\M)/H(u\M) \cl(\tilde{h}^{-1}[Z]) = \emptyset.$$
%Since $\tilde{h}^{-1}[Y] = \pi[\bar h^{-1}[Y]]$, $\tilde{h}^{-1}[Z] = \pi[\bar h^{-1}[Z]]$, and both these sets are relatively compact by (8), we deduce using Lemma \ref{lemma: K and X}(1) that  $\cl(\tilde{h}^{-1}[Y])  = \pi[\cl_\tau(\bar h^{-1}[Y])]$ and $\cl(\tilde{h}^{-1}[Z])  = \pi[\cl_\tau(\bar h^{-1}[Z])]$. We conclude that showing ($\dagger$) boils down to showing that $(\tilde{F}_{10} \cap u\M)H(u\M)\cl_\tau(\bar{h}^{-1}[Y]) \cap  (\tilde{F}_{10} \cap u\M)H(u\M) \cl_\tau(\bar{h}^{-1}[Z]) = \emptyset$. 
Since $\tilde{h}^{-1}[Y] \subseteq \pi[\bar h^{-1}[Y] \cap u\M]$, $\tilde{h}^{-1}[Z] \subseteq \pi[\bar h^{-1}[Z]\cap u\M]$, and $\bar h^{-1}[Y] \cap u\M$ as well as $\bar h^{-1}[Z]\cap u\M$ are relatively quasi-compact by (4), we deduce that  $\cl(\tilde{h}^{-1}[Y])  \subseteq \cl( \pi[\bar h^{-1}[Y]\cap u\M]) = \pi[\cl_\tau(\bar h^{-1}[Y] \cap u\M)]$ and $\cl(\tilde{h}^{-1}[Z])  \subseteq \cl( \pi[\bar h^{-1}[Z]\cap u\M]) =\pi[\cl_\tau(\bar h^{-1}[Z]\cap u\M)]$. We conclude that in order to show ($\dagger$), it is enough to show that $(\tilde{F}_{10} \cap u\M)H(u\M)\cl_\tau(\bar{h}^{-1}[Y] \cap u\M) \cap  (\tilde{F}_{10} \cap u\M)H(u\M) \cl_\tau(\bar{h}^{-1}[Z]\cap u\M) = \emptyset$. 
By virtue of Lemma \ref{lemma: H(uM) form KrPi}, 
%%%Krzys25: ``is implied by'' instead of ``boils down to''.
%this boils down to 
this is implied by
$$(\dagger\dagger)\;\;\;\;\;\; (\tilde{F}_{13} \cap u\M)\cl_\tau(\bar{h}^{-1}[Y] \cap u\M) \cap  (\tilde{F}_{13} \cap u\M) \cl_\tau(\bar{h}^{-1}[Z]\cap u\M) = \emptyset.$$ 

%%%Krzys: I removed the the choice of $l$, as $l$ was chosen at the beginning of the proof.
We will show that $m:=56(4l+1)+2l$ works. 
%where $l \in \mathbb{N}$ witnesses that item (3) of Definition \ref{definition: gen. loc. comp. model} applied to $h$ holds. 
Suppose for a contradiction that ($!$) holds for this $m$, whereas ($\dagger\dagger$) fails.

%By Lemma \ref{lemma: K and X}, 
%%%Krzys: I replaced $S_G(N)$ by $S_{G,M}(N)$ twice below. Both versions are OK, by $S_G(N)$ was not explicitly mentioned in the preliminaries.
By (4), $\bar{h}^{-1}[Y]$ and $\bar{h}^{-1}[Z]$ are contained in some $S_{X^n,M}(N)$ which is closed in $S_{G,M}(N)$. 
%%%Krzys25: Modification below with direct reference to Lemma 2.21.
%Thus, by the definition of $\cl_\tau$ and an easy compactness argument, we get that any element $p$ in the intersection from  ($\dagger\dagger$) is of the form 
Hence, $\overline{\bar{h}^{-1}[Y]} = [\pi_1(x)]$ and  $\overline{\bar{h}^{-1}[Z]}=[\pi_2(x)]$ for some partial types $\pi_1(x)$, $\pi_2(x)$ over $N$,  where  $\overline{\bar{h}^{-1}[Y]}$ and  $\overline{\bar{h}^{-1}[Z]}$ are closures computed in $S_{G,M}(N)$. Thus, by Lemma \ref{lemma: tau-closure in terms of realizations of types}, we get that any element $p$ in the intersection from  ($\dagger\dagger$) is of the form 
$$\tp(a_1b_1^{-1} \dots a_{13}b_{13}^{-1} \alpha \beta/N) = \tp(a_1'b_1'^{-1} \dots a_{13}'b_{13}'^{-1} \alpha' \beta'/N)$$
for some $a_i,b_i,a_i',b_i',\alpha,\alpha',\beta,\beta' \in \bar G$ satisfying $a_i \equiv_M b_i$, $a_i' \equiv_M b_i'$, $\alpha \models u$, $\alpha' \models u$, $\tp(\beta/N) \in \overline{\bar{h}^{-1}[Y]}$, $\tp(\beta'/N) \in \overline{\bar{h}^{-1}[Z]}$.
%where  $\overline{\bar{h}^{-1}[Y]}$ and  $\overline{\bar{h}^{-1}[Z]}$ are closures computed in $S_{G,M}(N)$. 
Pick such an element $p$. By Lemma \ref{lemma: u in F_1}(1), it equals
 $$\tp(a_1b_1^{-1} \dots a_{14}b_{14}^{-1} \beta/N) = \tp(a_1'b_1'^{-1} \dots a_{14}'b_{14}'^{-1} \beta'/N)$$ 
 for some $a_{14},b_{14},a_{14}',b_{14}' \in \bar G$ with $a_{14} \equiv_M b_{14}$ and $a_{14}' \equiv_M b_{14}'$.
 
 \begin{clm}
 $h^*(\beta) \in S^{2l} Y$ and $h^*(\beta') \in S^{2l}Z$.
 \end{clm}

%%%Krzys25: = instead of \ne.
\begin{clmproof}
Suppose for a contradiction that  $h^*(\beta) \notin S^{2l} Y$. So $S^l h^*(\beta) \cap S^{l} Y = \emptyset$. Then $S^l U \cap S^l V = \emptyset$ for some compact neighborhoods $U$ of $h^*(\beta)$ and $V$ of $Y$. By the choice of $l$, there is a formula $\theta(x) \in L_M$ such that $h^{-1}[U] \subseteq \theta(G)$ and $h^{-1}[V] \subseteq G \setminus \theta(G)$.

Case 1. $\theta(x) \in \tp(\beta/M)$. Since $\tp(\beta/N) \in \overline{\bar{h}^{-1}[Y]}$, there is $q \in [\theta(x)] \cap \bar h^{-1}[Y]$. Then $ \bar h(q) \in \overline{h[\theta(G)]} \cap Y$. On the other hand, $h[\theta(G)] \subseteq V^c$, which implies that $\overline{h[\theta(G)]} \cap Y = \emptyset$, because $V$ is a neighborhood of $Y$. This is a contradiction.

Case 2. $\neg \theta(x) \in \tp(\beta/M)$. Then $h^*(\beta) \in \overline{h[G \setminus \theta(G)]}$.  On the other hand, $h[G \setminus \theta(G)] \subseteq U^c$, which implies that $h^*(\beta) \notin \overline{h[G \setminus \theta(G)]}$, because $U$ is a neighborhood of $h^*(\beta)$. This is a contradiction.
\end{clmproof}

Using (1), Claim 1(iii), and Claim 3, we get: 
$$\bar h(p) =h^*\left(\prod_{i=1}^{14}a_ib_i^{-1} \beta\right) \in h^*\left(\prod_{i=1}^{14}a_ib_i^{-1}\right)h^*(\beta)S^{4l+1} \subseteq S^{55(4l+1)}S^{2l}YS^{4l+1} =S^{56(4l+1)+2l}Y.$$
Similarly, $\bar h(p) \in S^{56(4l+1)+2l}Z.$
Hence, $S^{56(4l+1)+2l}Y \cap S^{56(4l+1)+2l}Z \ne \emptyset$, which contradicts ($!$) for $m:=56(4l+1)+2l$.
\end{proof}

The usual notion of definable map from a definable subset $D$ of $M$ to a compact space is explained in terms of a factorization of this map through the type space $S_D(M)$ via a continuous map. The notion of definability in item (3) of Definition \ref{definition: gen. loc. comp. model} is less obvious. The next corollary explains it using a factorization through $u\M/H(u\M)$.

\begin{corollary}
%%%Krzys: $(u\mathcal{M}/H(u\mathcal{M}),C)$ in place of $f$ due to the change of Def. 4.1.
A quasi-homomorphism $h \colon G \to H:S$ with a compact, normal, symmetric subset $S$ of a locally compact group $H$ is a generalized definable locally compact model of $X$ if and only if there exists a good quasi-homomorphism $\tilde{h} \colon u\mathcal{M}/H(u\mathcal{M}) \to H :S^m$ for  $(u\mathcal{M}/H(u\mathcal{M}),C)$, for some $m \in \mathbb{N}$ such that $\tilde{h}(f(g)) \in h(g)S^m$ for all $g \in G$ (where $f \colon G \to  u\mathcal{M}/H(u\mathcal{M}) : C$ is the generalized definable locally compact model of $X$ from Theorem \ref{theorem: main Theorem}).
\end{corollary}

\begin{proof}
The implication $(\Rightarrow)$ follows directly from Theorem \ref{theorem: universality: existence}.

$(\Leftarrow)$ We check items (1), (2), (3) of Definition \ref{definition: gen. loc. comp. model} applied to $h$. 

(1) For any compact $V \subseteq H$ the set $S^mV$ is also compact, and so $\tilde{h}^{-1}[S^mV]$ is relatively compact. Thus, $h^{-1}[V] \subseteq f^{-1}[\tilde{h}^{-1}[S^mV]]$ is contained in some $X^i$. 

(2) We know that $f[X]$ is relatively compact, and so $\tilde{h}[f[X]]$ is relatively compact. Hence, $S^m\tilde{h}[f[X]]$ is relatively compact. Since $h[X] \subseteq S^m\tilde{h}[f[X]]$, we conclude that $h[X]$ is relatively compact, too.

(3) We will show that $l:= m+m_2$ works, where $m_2$ is a number witnessing that Remark \ref{remark: on good quasi-homomorphisms}(2) holds for the good quasi-homomorphism $\tilde{h}$. For that take any compact $Y,Z \subseteq H$ with $S^lY \cap S^lZ = \emptyset$. Then $S^{m_2}(S^mY) \cap S^{m_2}(S^mZ) = \emptyset$ and $S^mY$, $S^mZ$ are compact. So, by the choice of $m_2$, 
$$C^2 \cl(\tilde{h}^{-1}[S^mY]) \cap C^2 \cl(\tilde{h}^{-1}[S^mZ]) = \emptyset.$$
Therefore, $f^{-1}[\tilde{h}^{-1}[S^mY]]$ and $f^{-1}[\tilde{h}^{-1}[S^mZ]]$ can be separated by a definable set. 
Since $h^{-1}[Y] \subseteq f^{-1}[\tilde{h}^{-1}[S^mY]]$ and $h^{-1}[Z] \subseteq f^{-1}[\tilde{h}^{-1}[S^mZ]]$, we conclude that $h^{-1}[Y]$ and $h^{-1}[Z]$ can be separated by a definable set.
\end{proof}

\begin{theorem}\label{theorem: universality uniqueness}(Universality of $f \colon G \to u\mathcal{M}/H(u\mathcal{M})$: approximate uniqueness)
Take $f \colon G \to  u\mathcal{M}/H(u\mathcal{M}) : C$ from Theorem \ref{theorem: main Theorem}.  Let $h \colon G \to H :S$ be an arbitrary generalized definable locally compact model of $X$, and $\rho \in \Mor(f,h)$ any morphism.
%%%Krzys: Instead of "Let $\widetilde{h} \in \Mor(f,h)$ be the morphism constructed in Theorem \ref{theorem: universality: existence}." I wrote "Let $\widetilde{h} \in \Mor(f,h)$ be a morphism constructed in Theorem \ref{theorem: universality: existence}.", because it is not uniquely determined.
Let $\widetilde{h} \in \Mor(f,h)$ be a (non uniquely determined) morphism constructed in Theorem \ref{theorem: universality: existence}. 
%%%Krzys: I changed $l$ to $n$ in the statement and in the proof, and added that this $n$ also depends on $l$  in item (3) of Definition  \ref{definition: gen. loc. comp. model} applied to $h$.
Then there is $n \in \mathbb{N}$ 
(depending only on $l$ in item (3) of Definition  \ref{definition: gen. loc. comp. model} applied to $h$, and on $k$ in Definition \ref{definition: morphism} and $m_2$ in Remark \ref{remark: on good quasi-homomorphisms}(2) both applied to $\rho$) 
such that $\rho(p/H(u\M)) \in \tilde{h}(p/H(u\M)) S^n$ for all $p \in u\mathcal{M}$.
\end{theorem}

\begin{proof}
%%%Krzys: I corrected the formula for $n$ (the old $l$) so that it now depends also on $l$.
We will show that $n:=4\max(m_2, k+12(4l+1))$ works. Suppose not, i.e. $\rho(p/H(u\M)) \notin \tilde{h}(p/H(u\M))S^n$ for some $p \in u\M$. Then $S^{\frac{n}{2}} \rho(p/H(u\M)) \cap  S^{\frac{n}{2}} \tilde{h}(p/H(u\M)) = \emptyset$. So we can find a compact neighborhood $V$ of the neutral element in $H$ such that 
$$S^{\frac{n}{2}} \rho(p/H(u\M))V \cap  S^{\frac{n}{2}} \tilde{h}(p/H(u\M))V = \emptyset.$$
Put $V':=VS^{\frac{n}{4}}$. Then 
$$S^{\frac{n}{4}} \rho(p/H(u\M))V' \cap  S^{\frac{n}{4}} \tilde{h}(p/H(u\M))V' = \emptyset,$$
and $\rho(p/H(u\M))V'$ and $\tilde{h}(p/H(u\M))V'$ are compact sets. Since $n/4 \geq m_2$, we get 
$$(*)\;\;\;\;\;\; C^2 \cl (\rho^{-1}[\rho(p/H(u\M))V']) \cap C^2 \cl (\rho^{-1}[\tilde{h}(p/H(u\M))V']) = \emptyset.$$
Put $P:= \cl (\rho^{-1}[\rho(p/H(u\M))V'])$ and $Q:= \cl (\rho^{-1}[\tilde{h}(p/H(u\M))V'])$. 
%%%Krzys25: Now, we referred to part (7) instead of to the proof of part (7).
By part (7) of the proof of Theorem  \ref{theorem: main Theorem}, we conclude from ($*$) that there exists $\theta(x) \in L_M$ such that 
$$(**)\;\;\;\;\;\; \hat{f}^{-1}[P] \subseteq [\theta(x)] \textrm{ and } \hat{f}^{-1}[Q] \subseteq [ \neg\theta(x)].$$

%%%Krzys: Everywhere below 108 replaced by $12(4l+1)$ (using corrected Claim 2 from the proof of Thm. 4.6).
%%%Krzys: I added "(where $\bar h$ is as in Theorem  \ref{theorem: universality: existence})".
Since $p/H(u\M) \in P$ and $\hat{f}(p) =upu/H(u\M)=p/H(u\M)$, we see that $p \in \hat{f}^{-1}[P]$, hence $\theta(x) \in p$, and so $\bar h(p) \in \overline{h[\theta(G)]}$ (where $\bar h$ is chosen as in Theorem  \ref{theorem: universality: existence}). By Claim 2 from the proof of Theorem \ref{theorem: universality: existence}, we conclude that $\tilde{h}(p/H(u\M)) \in \overline{h[\theta(G)]}S^{12(4l+1)}$. So

\begin{flalign*}
&\: \tilde{h}(p/H(u\M)) \in \overline{\rho[f[\theta(G)]] S^k}S^{12(4l+1)} =  \overline{\rho[f[\theta(G)]]}S^{k+12(4l+1)} \subseteq 
\overline{\rho[\hat{f}[[\theta(x)]]]}S^{k+12(4l+1)} \subseteq \\ 
&\: \overline{\rho[Q^c]}S^{k+12(4l+1)} \subseteq \overline{(\tilde{h}(p/H(u\M))V')^c}S^{k+12(4l+1)} =  \overline{(\tilde{h}(p/H(u\M))VS^{\frac{n}{4}})^c}S^{k+12(4l+1)}\subseteq \\
&\: \overline{(\tilde{h}(p/H(u\M))VS^{\frac{n}{4}})^c}S^{\frac{n}{4}},
\end{flalign*}
where the first belonging is by the choice of $k$, the first equality by compactness of $S$\footnote{Observe that for any subsets $A$ and $S$ of a topological group such that $S$ is compact we have $\overline{AS} =\overline{A}S$.}, the first inclusion is obvious, the second follows by ($**$), the next one by the definition of $Q$, the next equality by the definition of $V'$, and the last inclusion since $n/4 \geq k+12(4l+1)$. Thus, $\tilde{h}(p/H(u\M)) \in  \overline{(\tilde{h}(p/H(u\M))VS^{\frac{n}{4}})^c}S^{\frac{n}{4}}$, which is impossible, because it implies that $\tilde{h}(p/H(u\M))VS^{\frac{n}{4}} \cap  (\tilde{h}(p/H(u\M))VS^{\frac{n}{4}})^c \ne \emptyset$.
\end{proof}

To get full uniqueness (i.e. that $f$ is the initial object) we have to modify the notion of morphism.

\begin{definition}
Let $f \colon G \to H:S$ and $h \colon G \to L:T$ be generalized definable locally compact models of $X$. Let $\rho_1,\rho_1' \in \Mor(f,h)$. We say that $\rho_1$ and $\rho_1'$ are {\em equivalent} (symbolically, $\rho_1 \sim \rho_1'$)
if for some $l \in \mathbb{N}$, for every $p \in H$ we have $\rho_1'(p) \in \rho_1(p) T^l$. 
\end{definition}

%We could also consider a finer relation with the additional requirement  $\rho_2(p) \in C^l \rho_1(p)$ for all $p \in H$.

\begin{remark}
$\sim$ is an equivalence relation on $\Mor(f,h)$.
\end{remark}

\begin{proposition}
If $f_i \colon G \to H_i: S_i$ for $i\in \{1,2,3\}$ are generalized definable locally compact models of $X$ and $\rho_1 \sim \rho_1'$ belong to $\Mor(f_1,f_2)$ and $\rho_2 \sim \rho_2'$ belong to $\Mor(f_2,f_3)$, then $\rho_2\rho_1 \sim \rho_2' \rho_1'$. Thus, all generalized definable locally compact models of $X$ with morphisms modulo $\sim$ form a category.
\end{proposition}

\begin{proof}
Let $l_1$ and $l_2$ be witnesses for $\rho_1 \sim \rho_1'$ and $\rho_2 \sim \rho_2'$, that is $\rho_1'(p) \in \rho_1(p) S_2^{l_1}$ and  $\rho_2'(q) \in \rho_2(q) S_3^{l_2}$ for all $p \in H_1$ and $q \in H_2$. Let $k_2'$ be a witness that $\rho_2' \in \Mor(f_2,f_3)$, that is $\rho_2'\colon H_2 \to H_3: S_3^{k_2'}$ and $\rho_2'(f_2(g)) \in f_3(g)S_3^{k_2'}$, and let $n_{l_1}$ be the number from Remark \ref{remark: on good quasi-homomorphisms}(1) obtained for $\rho_2'$. Then, for every $p \in H_1$ we have 
$$\rho_2'(\rho_1'(p)) \in \rho_2'[\rho_1(p)S_2^{l_1}] \subseteq \rho_2'(\rho_1(p))\rho_2'[S_2^{l_1}]S_3^{k_2'} \subseteq \rho_2(\rho_1(p)) S_3^{l_2}S_3^{k_2'n_{l_1}}S_3^{k_2'}=\rho_2(\rho_1(p))S_3^{k_2'n_{l_1}+l_2+k_2'}.$$

Thus, for $\rho_1 \in \Mor(f_1,f_2)$ and $\rho_2 \in \Mor(f_2,f_3)$ we have a well-defined 
$$\rho_1/\!\!\sim \circ \rho_2/\!\!\sim: = (\rho_1 \circ \rho_2)/\!\!\sim.$$
So it is clear that the all generalized definable locally compact models of $X$ with morphisms modulo $\sim$ form a category.
\end{proof}

By Theorems \ref{theorem: main Theorem}, \ref{theorem: universality: existence}, and \ref{theorem: universality uniqueness}, we get the main result of this section.

\begin{corollary}
The generalized definable locally compact model $f \colon G \to u\mathcal{M}/H(u\mathcal{M}):C$ from Theorem \ref{theorem: main Theorem} is the initial object in the category from the last proposition.
\end{corollary}

%%%Krzys: I expanded the final part of this section. I hope it is more transparent now.
We finish with some natural questions which arise in the special case of Theorem \ref{theorem: universality: existence} when $h:=f$, where $f \colon G \to  u\mathcal{M}/H(u\mathcal{M}) : C$ is from Theorem \ref{theorem: main Theorem}. In this special case, the construction described in the second paragraph of Theorem \ref{theorem: universality: existence} yields non uniquely determined functions $\bar f \colon S_{G,M}(N) \to  u\mathcal{M}/H(u\mathcal{M})$ and $\tilde{f} \colon  u\mathcal{M}/H(u\mathcal{M}) \to  u\mathcal{M}/H(u\mathcal{M})$ such that $\tilde{f} \in \Mor(f,f)$. On the other hand, clearly $\id \in \Mor(f,f)$. This leads to

\begin{question}\phantomsection\label{question: bar f = hat f}
\begin{enumerate}
\item Can we choose $\tilde{f}$ by the construction in Theorem \ref{theorem: universality: existence} so that $\tilde{f}=\id$?
\item  Can we choose $\bar{f}$ by the construction in Theorem \ref{theorem: universality: existence} so that $\bar{f}|_{u\M} \colon u\mathcal{M} \to u\mathcal{M}/H(u\mathcal{M})$ is the quotient map?
\item Can we choose $\bar{f}$ by the construction in Theorem \ref{theorem: universality: existence} so that $\bar f (p) = \hat{f}(p) :=upu / H(u\mathcal{M})$ for all $p \in S_{G,M}(N)$?
\end{enumerate}
\end{question} 

By how $\tilde{f}$ is obtained from $\bar f$, we see that a positive answer to (2) implies a positive answer to (1). Since $\hat{f} (p) = p/ H(u\mathcal{M})$ for all $p \in u\M$, we get that a positive answer to (3) implies a positive answer to (2).

In the next section, the example with $X$ being a definable, generic, symmetric subset of the universal cover $\widetilde{\SL_2(\mathbb{R})}$ of $\SL_2(\R)$ will yield a negative answer to (3), but not to (2).

%%%%%%%%%%%%%%%%%%%%%%%%%%%%%%%%%%%%%%%%%%%%%%%%%%%%%%
%%%%%%%%%%%%%%%%%%%%%%%%%%%%%%%%%%%%%%%%%%%%%%%%%%%%%%
\begin{comment}

Applying Theorem \ref{theorem: universality: existence} in the case when $H := u\mathcal{M}/H(u\mathcal{M})$ and $h:=f$, we get $\bar h = \bar f \colon S_{G,M}(N) \to u\mathcal{M}/H(u\mathcal{M})$. But this $\bar f$ is not uniquely determined by the construction in the statement of Theorem \ref{theorem: universality: existence}.

\begin{question}\label{question: bar f = hat f}
Can we choose $\bar f$ by the construction in Theorem \ref{theorem: universality: existence} so that $\bar f (p) = \hat{f}(p) :=upu / H(u\mathcal{M})$ for all $p \in S_{G,M}(N)$? A weaker question is whether we can have it only for $p \in u\mathcal{M}$, i.e. $\bar f (p) = p/ H(u\mathcal{M})$ for all $p \in u\M$.
\end{question}

The example with $X$ being a definable generic subset of the universal cover $\widetilde{\SL_2(\mathbb{R})}$ of $\SL_2(\R)$ will yield a negative answer to the first question, but not to the second (weaker) one. 

\end{comment}
%%%%%%%%%%%%%%%%%%%%%%%%%%%%%%%%%%%%%%%%%%%%%%%%%%%%%%%
%%%%%%%%%%%%%%%%%%%%%%%%%%%%%%%%%%%%%%%%%%%%%%%%%%%%%%%

\section{Compact case}\label{section: compact case}

In this section, we focus on the special case when the definable approximate subgroup $X$ generates a group in finitely many steps. This is equivalent to  $G:=\langle X \rangle$ being a definable group in which $X$ is a definable, generic, symmetric set ($X$ being {\em generic} in $G$ means that finitely many left translates of $X$ cover $G$). Thus, we will consider just this case or, slightly more generally, the case of a definable generic subset $X$ of a definable group $G$ (notice that then $\langle X \rangle$ has finite index in $G$), which is fundamental in model theory. 
%Since this section is addressed mostly to model-theorists, we will not recall basic model-theoretic notions, such as definability of types or the $\dcl$-closure, which will be in use.

In the case when $G=\langle X \rangle$, the group $u\M/H(u\M)$ in the generalized definable locally compact model $f \colon G \to u\M/H(u\M)$ from Theorem \ref{theorem: main Theorem} is compact, which follows from 
%%%Krzys25: Referred to the new Lemma 3.22.
%the last paragraph of the proof of Proposition \ref{proposition: quasi local compactness}, since $u\M \subseteq S_{X^n,M}(M)$ for some $n \in \mathbb{N}$; 
Lemma \ref{lemma: quasi-compactness from the old 3.21}, since $u\M \subseteq S_{X^n,M}(M)$ for some $n \in \mathbb{N}$;  in fact, in this case, all the topological dynamics developed in Subsection \ref{subsection: top dyn for S(N)} boils down to the classical topological dynamics of the compact flow $S_{G,M}(N)$, so  $u\M/H(u\M)$ is compact.

%%%Krzys25: We recalled the definition of $S_{G,M}(N)$ in the context of a definable group $G$.
In the more general context of $X$ being a definable, generic, symmetric subset of a definable group $G$, one can also use the compact group $u\M/H(u\M)$ computed for the compact $G$-flow $S_{G,M}(N):=\{p \in S_G(N): p \textrm{ finitely satisfiable in } M\}$ and adapting (and even simplifying some parts of) the arguments from Subsection \ref{subsection: main theorem}, we conclude with

\begin{theorem}\label{theorem: main theorem for definable G}
The function $f \colon G \to u\M/H(u\M)$ given by $f(g):=ugu/H(u\M)$ has the following properties. 
\begin{enumerate}
\item $f$ is a quasi-homomorphism with compact, normal, symmetric error set $C:=  \cl_{\tau}(\tilde{F}) \cup \cl_{\tau}(\tilde{F})^{-1}$, where:
\begin{flalign*}
&\:F_n:=  \{ x_1y_1^{-1} \dots x_ny_n^{-1}: x_i, y_i \in \bar{G}\; \textrm{and}\; x_i \equiv_M y_i\; \textrm{for all}\; i\leq n\},\\
&\: \tilde{F}_n:=  \{ \tp(a/N) \in S_{G,M}(N): a \in F_n\},\\
&\: \tilde{F}:=  ((\tilde{F}_7 \cap u\mathcal{M})/H(u\mathcal{M}))^{u\mathcal{M}/H(u\mathcal{M})}.
\end{flalign*}
Moreover, $(\tilde{F}_3 \cap u\mathcal{M})/H(u\M)$ is an error set of $f$.
\item $f^{-1}[C] \subseteq X^{30}$.
\item There is a compact neighborhood $U$ of the neutral element in $u\M/H(u\M)$ such that $f^{-1}[U] \subseteq X^{14}$ and $f^{-1}[UC] \subseteq X^{34}$.
\item For any closed $Z,Y \subseteq u\M/H(u\M)$ with $C^2 Y\cap C^2 Z = \emptyset$ the preimages $f^{-1}[Y]$ and $f^{-1}[Z]$ can be separated by a definable set.
\end{enumerate}
\end{theorem} 

Note that if $X$ is definable, generic, but not symmetric, then replacing it by $XX^{-1}$, we get a definable, generic, and symmetric set, and it is clear how to modify items (2) and (3) in this context. So the assumption that $X$ is symmetric is rather minor.

The proof of Fact \ref{fact: preimages of neighborhoods}(1) adapts to

\begin{remark}\label{remark: preimages of neighborhoods}
For every neighborhood $U$ of $u/H(u\M)$ the preimage $f^{-1}[UC]$ is generic in $G$, that is the preimage under $f$ of any neighborhood of $C$ is generic in $G$.
\end{remark}

Theorem \ref{theorem: main theorem for definable G}(3) and Remark \ref{remark: preimages of neighborhoods} can be thought of as a structural result on definable generic subsets of an arbitrary definable group $G$. In concrete examples, this can lead to more precise information on generics. 

%%%Krzys: I added the whole part below on what we do in Subsection 5.1, including the discussion on our weakening of Newelski's conjecture. And I distinguished Subsection 5.1.

In Subsection \ref{subsection: universal cover of SL_2(R)}, we will illustrate it by the universal cover $\widetilde{\SL_2(\R)}$ of $\SL_2(\R)$. Our analysis of $\widetilde{\SL_2(\R)}$ also yields a negative answer to item (3) of Question \ref{question: bar f = hat f}, and a positive answer to item (2) in the special case of definable generics in  $\widetilde{\SL_2(\R)}$. Moreover, our analysis confirms a certain weakening of Newelski's conjecture (that we have had in mind for a while) in the special case of $\widetilde{\SL_2(\R)}$. So we take the opportunity and state this weakened conjecture below.

%%%Krzys25: ``$|M|^+$-saturated'' instead of ``$|N|^+$-saturated''.
Let $G$ a group definable in a structure $M$. Let $N \succ M$ be $|M|^+$-saturated, and $\C \succ N$ a monster model. By $\bar G$ we denote the interpretation of $G$ in $\C$. Let $u\M$ be the Ellis group of the flow $(G,S_{G,M}(N))$, and let $\bar{G}^{00}_M$ be the smallest type-definable over $M$ subgroup of $\bar G$ which has bounded index. Newelski's conjecture says that the group epimorphism $\theta \colon u\M \to \bar G/\bar{G}^{00}_M$ given by $\theta(\tp(a/N)):=a/\bar{G}^{00}_M$ is an isomorphism under suitable assumptions on tameness of the ambient theory \cite{New}. In \cite{GiPePi}, the conjecture was refuted for $G:=\SL_2(\R)$ treated as a group definable in $M:=(\R,+,\cdot)$, where the Ellis group turned out to be $\Z_2$ while $\bar G/\bar{G}^{00}_M$ is trivial. 
%Many counter-examples were then constructed in \cite{GiKr}, e.g. the non-abelian free groups equipped with the full structure. 
On the other hand, the conjecture was confirmed in \cite{ChSi} for definably amenable groups definable in NIP theories. In \cite{KrPi}, we refined Newelski's epimorphism $\theta$ obtaining a sequence of epimorphisms
$$u\M \to u\M/H(u\M) \to \bar G/\bar{G}^{000}_M \to \bar{G}/\bar{G}^{00}_M,$$
where $\bar{G}^{000}_M$ is the smallest bounded index subgroup of $\bar G$ which is invariant under $\Aut(\C/M)$. This leads to many counter-examples to Newelski's conjecture. Namely, whenever $\bar{G}^{000}_M \ne \bar{G}^{00}_M$, then Newelski's conjecture fails; in fact, we proved that then even $u\M/H(u\M) \to \bar G/\bar{G}^{000}_M$ is not an isomorphism. The first example where $\bar{G}^{000}_M \ne \bar{G}^{00}_M$ was found in \cite{CoPi}: $G:=\widetilde{\SL_2(\R)}$ treated as a group definable in the two-sorted structure $M:=((\R,+,\cdot),(\Z,+))$ has this property. Many other examples were then found in \cite{GiKr}, e.g. the non-abelian free groups equipped with the full structure. 
%Another reason for failure of Newelski's conjecture is when $H(u\M)$ is nontrivial, equivalently when $u\M$ is not Hausdorff in the $\tau$-topology. While in general we are able to find examples in which $u\M$ is not Hausdorff, we have not found any such example with NIP.
Another situation in which Newelski's conjecture fails is when $H(u\M)$ is nontrivial, equivalently when $u\M$ is not Hausdorff in the $\tau$-topology. While in general we are able to find examples in which $u\M$ is not Hausdorff, we have not found any such example with NIP. This leads to the following weakening of Newelski's conjecture.

\begin{conjecture}\label{conjecture: refined Newelski's conjecture}
If $M$ has NIP, then $u\M$ is Hausdorff.
\end{conjecture}

From the above discussion, this is true for definably amenable groups definable in NIP theories. It is also true whenever $u\M$ is finite, as the $\tau$-topology is $T_1$ and so Hausdorff when $u\M$ is finite. In Subsection \ref{subsection: universal cover of SL_2(R)}, we will confirm it for $G:=\widetilde{\SL_2(\R)}$ treated as a group definable in the two-sorted structure $M:=((\R,+,\cdot), (\Z,+))$ which clearly has NIP; more precisely, the Ellis group in this case will turn out to be topologically isomorphic to the profinite completion $\hat{\Z}$ of $\Z$.

\subsection{Case study of  $\widetilde{\SL_2(\R)}$}\label{subsection: universal cover of SL_2(R)}

From now on, $M$ is the 2-sorted structure with the sorts $(\R,+,\cdot)$ and $(\Z,+)$, $G:=\SL_2(\R)$, and $\tilde{G}:=\widetilde{\SL_2(\R)}$. So now $\tilde{G}$ will play the role of $G$ from the above discussion.

%%%Krzys25: $(a_1,b_1)(a_2,b_2):= (a_1a_2,b_1+b_2+h(b_1,b_2))$ replaced by $(x,m)(y,n):= (xy,m+n+h(x,y))$ to avoid a clash of notation with coefficients of the matrices below (and correction of arguments of $h$). And added a reference to Asai.
Recall from \cite{Asai} (in particular, see \cite[Theorem 2]{Asai}) that $\tilde{G}$ can be written as  $\SL_2(\R) \times \Z$ with the group operation given by 
%$(a_1,b_1)(a_2,b_2):= (a_1a_2,b_1+b_2+h(b_1,b_2))$, 
$(x,m)(y,n):= (xy,m+n+h(x,y))$, where $h \colon G \times G \to \Z$ is the 2-cocycle defined as follows.
For $c,d \in \R$ put 
$$c(d):= 
\left\{
\begin{array}{ll}
c, & \textrm{if } c \ne 0\\
d, & \textrm{if } c = 0.
\end{array}
\right.
$$
Then for any
$\left(
\begin{array}{ll}
a_1 & b_1\\
c_1 & d_1\\
\end{array}
\right), 
\left(
\begin{array}{ll}
a_2 & b_2\\
c_2 & d_2\\
\end{array}
\right) \in \SL_2(\R)$, writing
$
\left(
\begin{array}{ll}
a_1 & b_1\\
c_1 & d_1\\
\end{array}
\right) \cdot 
\left(
\begin{array}{ll}
a_2 & b_2\\
c_2 & d_2\\
\end{array}
\right) =
\left(
\begin{array}{ll}
a_3 & b_3\\
c_3 & d_3\\
\end{array}
\right)$,
we have 
$$h\left(\left(
\begin{array}{ll}
a_1 & b_1\\
c_1 & d_1\\
\end{array}
\right), 
\left(
\begin{array}{ll}
a_2 & b_2\\
c_2 & d_2\\
\end{array}
\right)\right) :=
\left\{ 
\begin{array}{ll}
1, & \textrm{if } c_1(d_1)>0, c_2(d_2)>0, c_3(d_3)<0,\\
-1, & \textrm{if } c_1(d_1)<0, c_2(d_2)<0, c_3(d_3)>0,\\
0,  & \textrm{otherwise}.
\end{array}
\right. 
$$
From this formula, we see that $h$ is definable in $M$, and so $\tilde{G}$ is definable in $M$.
%%%Krzys: Instead of "the above theorem" I wrote "Theorem 5.1".
It is clear that each set of the form $G \times k\Z$ is a definable generic subset of $\tilde{G}$. Using Theorem \ref{theorem: main theorem for definable G}, we will deduce a weak converse.

%%%Krzys25: 640 instead of 696 due to the improvement of the statement of Cor. 5.6.
\begin{proposition}\label{proposition: structure of generics in the universal cover}
For every definable, generic, symmetric subset $X$ of $\tilde{G}$ there exists a nonzero $k \in \mathbb{N}$ such that $G \times k\Z \subseteq X^{640}$.
\end{proposition}

Besides Theorem \ref{theorem: main theorem for definable G}, we will need a few other ingredients, some of which will be also used in the proof of Proposition \ref{proposition: universal cover yields an answer} below. The proof of Proposition \ref{proposition: structure of generics in the universal cover} is given after the proof of Lemma \ref{lemma: u_G B u_G=u_G}. 

By \cite[Theorem 3.2]{CoPi}, we know that $\tilde{G}$ does not have any definable subgroups of finite index, so for every definable generic subset $X$ of $\tilde{G}$ we have that $\langle X \rangle = \tilde{G}$. Hence, in this situation, Theorem \ref{theorem: main theorem for definable G} is a particular case of Theorem \ref{theorem: main Theorem}.

%%%Krzys25: I referred to the published version of Gismatullin's paper (and corrected the concrete number of the theorem to which we refer).
One of the ingredients will be {\em 12-connectedness} of $\SL_2(\R)$ which follows from \cite[Theorem 7.7]{Gis_published}, which we will  briefly discuss. By a {\em thick} subset of a group $H$ definable in a structure $N$ we mean a definable symmetric subset $Y$ of $H$ for which there exists a positive $m \in \mathbb{N}$ such that for every $g_1,\dots,g_m \in H$ there are $i<j$ with $g_i^{-1}g_j \in Y$.  Note that when $Y$ is a definable generic, then $Y^{-1}Y$ is thick. We will say that $H$ is {\em $n$-connected} if for every (definable) thick subset $Y$ of $H$ we have $Y^n=H$. This is equivalent to saying that for $P$ being the intersection of all $N$-definable thick subsets of $\bar H:=H(\C)$ (where $\C \succ N$ is a monster model) we have $P^n=\bar H$. 
%%%Krzys25: I referred to the published paper (and corrected the concrete number of theorem to which we refer). I also added a justification.
%The next fact is a particular case of  \cite[Theorem 6.5]{Gis}.
The next fact is a particular case of  \cite[Theorem 7.7]{Gis_published}, because $\SL_2(\mathbb{R})$ is perfect and satisfies the assumptions of  \cite[Theorem 7.7]{Gis_published} (i.e., it is $\mathbb{R}$-split since the maximal torus consisting of the diagonal matrices is $\mathbb{R}$-isomorphic to the multiplicative group, and it is semisimple in the sense that it does not have nontrivial, connected, normal abelian subgroups).

\begin{fact}\label{fact: Gismatullin}
$\SL_2(\R)$ is $12$-connected in any structure in which $\SL_2(\R)$ is definable, in particular in $M$.
\end{fact}

Everywhere below $\C \succ M$ is a monster model, and bars are used to denote the interpretations of various objects in $\C$.

%%%Krzys: I removed the redundant $X$.
%%%Krzys25: 11 instead of 12.
\begin{corollary}\label{corollary: from Gismatullin}
%Let $X$ be a definable, generic, symmetric subset of $\tilde{G}$. Then for every $g \in \bar G$ there exists $n \in \{-12,\dots,0,\dots,12\}$ such that $(g,n) \in \bar X^{24}$. 
Let $X$ be a definable, generic, symmetric subset of $\tilde{G}$. Then for every $g \in \bar G$ there exists $n \in \{-11,\dots,0,\dots,11\}$ such that $(g,n) \in \bar X^{24}$. 
\end{corollary}

\begin{proof}
By Fact \ref{fact: Gismatullin}, $\bar G = P^{12}$, where $P$ is the intersection of all $M$-definable thick subsets of $\bar G$. Hence, by 
%%%Krzys25: Referred to the published paper (here and in the final part of the proof).
%\cite[Lemma 1.3(1)]{Gis}, 
\cite[Lemma 2.3(1)]{Gis_published},
$g=\prod_{i=1}^{12} a_i^{-1}b_i$ for some $a_i,b_i \in \bar G$ such that $a_i \Theta_M b_i$, meaning that $(a_i,b_i)$ starts an infinite $M$-indiscernible sequence. Then clearly $((a_i,0), (b_i,0))$ starts an infinite $M$-indiscernible sequence of pairs, i.e. 
$$(*)\;\;\;\;\;\; (a_i,0) \Theta_M (b_i,0).$$ 

The corresponding entries of matrices $a_i$ and $b_i$ have the same sign (because they have the same type), so, using the explicit definition of $h$ recalled above and the formula 
$\left(
\begin{array}{ll}
a & b\\
c & d
\end{array}
\right)^{-1} =
\left(
\begin{array}{rr}
d & -b\\
-c & a
\end{array}
\right)
$ for matrices in $\SL_2(\R)$, one easily checks that $h(a_i^{-1},b_i) = h(a_i,a_i^{-1})$. Hence, 
$$(a_i,0)^{-1}(b_i,0)  = (a_i^{-1}, -h(a_i,a_i^{-1}))(b_i,0) = (a_i^{-1}b_i, h(a_i^{-1},b_i) -h(a_i,a_i^{-1})) = (a_i^{-1}b_i,0).$$
As  $\im(h) = \{-1,0,1\}$, the last thing implies that 
%%%Krzys25: 11 instead of 12.
$$\prod_{i=1}^{12} (a_i,0)^{-1}(b_i,0) \in \left\{\prod_{i=1}^{12} a_i^{-1}b_i \right\} \times \{-11,\dots,0,\dots, 11\}.$$ 
On the other hand, by thickness of $\bar X^2$, ($*$), and 
%\cite[Lemma 1.3(1)]{Gis}, 
\cite[Lemma 2.3(1)]{Gis_published},
%%%Krzys25: 11 instead of 12.
we get that $\prod_{i=1}^{12} (a_i,0)^{-1}(b_i,0) \in \bar X^{24}$. So there exists $n \in \{-11,\dots,0,\dots,11\}$ such that $(g,n) \in \bar X^{24}$.
\end{proof}

The topological dynamics of the $G$-flow $S_G(\R)$ was worked out in \cite{GiPePi}, including the computation of the Ellis group which turns out to be $\Z_2$.
We will also need the topological dynamics of the $\tilde{G}$-flow $S_{\tilde{G}}(M)$ studied in \cite[Section 5]{Jag}. 
%%%Krzys25: Removed sentence, as now we give precise references to this paper following the request of the referee.
%The results below which are stated without references can be found in \cite[Section 5]{Jag}.

%First of all, it is well-known that all types in $S_G(M) =S_G(\R)$ are definable and all types in $S_{\tilde{G}}(M)$ are definable. So $S_G(M)$ and $S_{\tilde{G}}(M)$ coincide with $S_{G,\ext}(M)$ and $S_{\tilde{G},\ext}(M)$, respectively, and hence the Ellis semigroup operation on these sets is given by $p*q =\tp(ab/M)$ for some/any $b \models q$ and $a \models p$ such that $\tp(a/M,b)$ is a coheir over $M$.
%%%Krzys25: I added ``(first time stated in \cite[page 71]{Dries})'' and ``(by stability of $(\mathbb{Z},+)$)''.
First of all, it is well-known that all types in $S(M)$ are definable, because this is true for all types in $S(\R)$ (first time stated in \cite[page 71]{Dries}) and in $S(\Z)$ (by stability of $(\mathbb{Z},+)$) and there is no interaction between the two sorts of $M$. So $S_G(M)$ and $S_{\tilde{G}}(M)$ coincide with $S_{G,\ext}(M)$ and $S_{\tilde{G},\ext}(M)$, respectively, and hence the Ellis semigroup operation on these sets is given by $p*q =\tp(ab/M)$ for some/any $b \models q$ and $a \models p$ such that $\tp(a/M,b)$ is a coheir over $M$.

%%%Krzys25: I added references.
By \cite[Example 5.7]{Jag} (which actually follows from \cite{GiPePi}), the Ellis group of the flow $S_G(M)$ consists of two types $q_0,q_1$, where $q_0:= \tp(A/M)$ and $q_1:=\tp(-A/M)$ for 
$$A:= \left(
\begin{array}{cc}
(1-x)b & (1-x)c -yb^{-1}\\
yb & yc + (1-x)b^{-1}
\end{array}
\right)$$
%%%Krzys25: I added that $x$ is infinitesimal in order to emphasize it (note that it follows from the rest).
where $b >\R$, $c>\dcl(\R,b)$, $x$ positive infinitesimal, $y$ positive infinitesimal  with $(1-x)^2 + y^2 =1$, and $\tp(x,y/M,b,c)$ coheir over $M$ (which implies that $x,y$ are greater than all infinitesimals in $\dcl(\R,b,c)$). Then $q_0$ is the neutral element of the Ellis group $\{q_0,q_1\}$, so an idempotent in a minimal left ideal of $S_G(M)$, and hence we will denote $q_0$ by $u_G$.

%%%Krzys25: Added reference. Since the relevant things are in several places on page 9, it is difficult to give a more precise reference.
By \cite[page 9]{Jag}, the space $S_{\tilde{G}}(M)$ is naturally homeomorphic with $S_G(\R) \times S_{\Z}(\Z)$, and the induced semigroup operation is given by
$$(p,q)*(p',q')= (p*p', q+q'+h(p,p')),$$
where $h(p,p'):= h(a,a')$ for some/any $a \models p$ and $a' \models p'$ such that $\tp(a/M,a')$ is a coheir over $M$, $+$ denotes the semigroup operation on $S_{\Z}(\Z)$ (which is indeed commutative),
%%%Krzys25: Added ``and $h(a,a')$ is identified with $\tp(h(a,a')/\Z)$'', as this confused the referee.
and $h(a,a')$ is identified with $\tp(h(a,a')/\Z)$. From now on, $(S_{\tilde{G}}(M),*)$ will be identified with $(S_G(\R) \times S_{\Z}(\Z),*)$. Since $*$ uses $h$, we will denote this semigroup as  $S_G(\R) \times_h S_{\Z}(\Z)$. As to the semigroup $S_{\Z}(\Z)$, we will interchangeably use additive and multiplicative notation.

%%%Krzys25: Referred to Newelski.
By \cite[Corollary 1.9]{New} and the discussion in the penultimate paragraph on page 68 in \cite{New}, since $(\Z,+)$ is stable, there is a unique minimal left ideal $\M_{\Z}$ and it consists of the generic types. 
%%%Krzys25: Added ``is'' after ``which''.
There is also a unique idempotent $u_{\Z}$ in $\M_{\Z}$ which is the generic type concentrated on the component $\bar \Z^0$ (the intersection of all definable subgroups of $\bar{\Z}$ of finite index), and $u_\Z \M_\Z=\M_\Z$. 
%%%Krzys25: Added two sentences.
In fact, by basic stable group theory, for each coset of $\bar{\Z}^{0}$ there is a unique generic type concentrated on it. Thus, there is a natural isomorphism between $u_{\Z}\M_{\Z}$ and $\bar{\Z}/\bar{Z}^0$ given by $\tp(a/\Z) \mapsto a/\bar{\Z}^0$.

By the explicit formula for $h$ and the idempotency of $u_G$, we get $h(u_G,u_G)=0$. So \cite[Proposition 5.6]{Jag} yields

\begin{fact}\label{fact: basic from Jagiella}
%Let $\M_G$ and $\M_{\Z}$ be minimal left ideals of $S_G(\R)$ and $S_{\Z}(\Z)$, respectively. Let $u_G \in \M_G$ and $u_{\Z} \in \M_{\Z}$ be idempotents.  
Let $\M_G \ni u_G$ be a minimal left ideal of $S_G(\R)$. 
Then:
\begin{enumerate}
\item $\M_{\tilde{G}}: = \M_G \times \M_{\Z}$ is a minimal left ideal of $S_{\tilde{G}}(M)$;
\item $u_{\tilde{G}}:=(u_G,u_{\Z})$ is an idempotent in $\M_{\tilde{G}}$;
\item The Ellis group $u_{\tilde{G}}\M_{\tilde{G}}$ equals $u_G\M_G \times_h \M_{\Z} = u_G\M_G \times_h u_{\Z}\M_{\Z}$.
\end{enumerate}
\end{fact}

%Since $u_{\tilde{G}}\M_{\tilde{G}} \cong \Z_2$ and $u_{\Z}\M_{\Z} \cong \hat{\Z}$ (the profinite completion of $\Z$), the following corollary is deduced in \cite[Example 5.7]{Jag}. 

$f_1 \colon u_G\M_G \to \Z_2$ given by $q_i \mapsto i$ is clearly an isomorphism. 
%%%Krzys25: I changed the next sentence, referring to an earlier discussion.
%Since $(\Z,+)$ is stable, it is well-known that the natural map $f_2 \colon u_{\Z}\M_{\Z} \to \bar \Z/\bar \Z^{0} \cong \hat{\Z}$ given by $\tp(a/\Z) \mapsto a/\bar \Z^0$ is an isomorphism. Thus, the following corollary is deduced in \cite[Example 5.7]{Jag}. 
%%%Krzys25: I removed $\cong \hat{Z}$, and instead added a new, more precise  sentence (the referee wanted to add that this isomorphism is topological).
By the discussion before Fact \ref{fact: basic from Jagiella} (which uses stability of $(\mathbb{Z},+)$), the natural map $f_2 \colon u_{\Z}\M_{\Z} \to \bar \Z/\bar \Z^{0}$ given by $\tp(a/\Z) \mapsto a/\bar \Z^0$ is an isomorphism. Note also that $\bar \Z/\bar \Z^{0}$ is topologically isomorphic to the profinite completion $\hat{\Z}$. Thus, the following corollary is deduced in \cite[Example 5.7]{Jag}. 

\begin{corollary}\label{corollary: jagiella's isomorphism} 
The map $(f_1,f_2) \colon u_{\tilde{G}}\M_{\tilde{G}} \to \Z_2 \times \hat{\Z}$ is an isomorphism, with the group operation on $\Z_2 \times \hat{\Z}$ given by $(x,n)(x',n'):= (x +_2 x', n+n'-xx')$. The target group is moreover topologically isomorphic to $\hat{\Z}$ via the map $(x,n) \mapsto x-2n$.
\end{corollary}

Our next goal is to show that the isomorphism in Corollary \ref{corollary: jagiella's isomorphism} is topological with respect to the $\tau$-topology on $u_{\tilde{G}}\M_{\tilde{G}}$. 
%%%Krzys: I added the next sentence.
This implies that $u_{\tilde{G}}\M_{\tilde{G}}$ is topologically isomorphic to $\hat{\Z}$, so Hausdorff which confirms Conjecture \ref{conjecture: refined Newelski's conjecture} for $\widetilde{\SL_2(\R)}$.

As the $\tau$-topology on $u_G\M_G$ is $T_1$, it is discrete, and so $f_1$ is a topological isomorphism.
%By \cite[Theorem ???]{KrPi} and \cite[Theorem ???]{ChSi}, the natural map from  $u_{\Z}\M_{\Z}$ to $\hat{\Z}$ is  topological isomorphism of groups.
%Since $(\Z,+)$ is stable, it is well-know that the natural map $u_{\Z}\M_{\Z} \to \bar Z/\bar \Z^{0} \cong \hat{Z}$ given by $\tp(a/\Z) \mapsto a/\bar Z^0$ is an isomorphism. Since this map is continuous by \cite[Theorem 0.1]{KrPi}, we get that it is a topological isomorphism (with  $u_{\Z}\M_{\Z}$ equipped with the $\tau$-topology).
Since $f_2$ is an isomorphism which is continuous by \cite[Theorem 0.1]{KrPi}, we get that it is a topological isomorphism (with  $u_{\Z}\M_{\Z}$ equipped with the $\tau$-topology).

The fact that the isomorphism $(f_1,f_2)$ from Corollary \ref{corollary: jagiella's isomorphism} is topological follows immediately from the above paragraph and the next proposition.

\begin{proposition}\label{proposition: product of tau topologies}
The $\tau$-topology on $u_G\M_G \times_h u_{\Z}\M_{\Z}$ is the product of the $\tau$-topologies on  $u_G\M_G$ and $u_{\Z}\M_{\Z}$.
\end{proposition}

Before the proof, let us show a few properties of $h$ which will be used also in the proof of Proposition \ref{proposition: universal cover yields an answer}.

%%%Krzys25: The next paragraph is new.
First, recall that $h$ satisfies the $2$-cocyle formula (see \cite[Lemma 2]{Asai}): $h(g_1,g_2) + h(g_1g_2,g_3)=h(g_1,g_2g_3) + h(g_2,g_3)$ for all $g_1,g_2,g_3 \in \SL_2(\R)$. So the same holds for $g_1,g_2,g_3 \in \SL_2(\bar \R)$. This implies the 2-cocyle formula for types: $h(p_1,p_2) + h(p_1p_2,p_3)=h(p_1,p_2p_3) + h(p_2,p_3)$ for all $p_1,p_2,p_3 \in S_G(\R)$. In order to see it, just consider $g_1 \models p_1, g_2 \models p_2, g_3 \models p_3$ such that $\tp(g_2/M,g_3)$ and $\tp(g_1/M,g_2,g_3)$ are both finitely satisfiable in $M$, and apply the 2-cocycle formula for $g_1,g_2,g_3$.

\begin{lemma}\phantomsection\label{lemma: properties of h}
\begin{enumerate}
\item $h(p,u_G)=0$ for all $p \in S_G(M)$.
\item $h(u_G, \tp(g/M))=0$ for all $g \in G$.
\item $h(u_G,gu_G) =0$ for all $g \in G$.
\end{enumerate}
\end{lemma}
 
 \begin{proof}
 (1) Present $u_G$ as $\tp(A/M)$ for 
 $A:= \left(
\begin{array}{cc}
(1-x)b & (1-x)c -yb^{-1}\\
yb & yc + (1-x)b^{-1}
\end{array}
\right),$
%with $x,y,b,c$ with the properties listed before. 
where  $x,y,b,c$ are as above.
Write $p$ as $\tp(B/M)$ for 
$B=\left( 
\begin{array}{ll}
\alpha & \beta\\
\gamma & \delta
\end{array}
\right)$
so that $\tp(B/M,x,y,b,c)$ is a coheir over $M$.
Then $pu_G =\tp(BA/M)$ and $BA=\left( 
\begin{array}{ll}
c_{11} & c_{12}\\
c_{21} & c_{22}
\end{array}
\right)$
with $c_{21}:=\gamma(1-x)b + \delta yb$.
Since $yb>0$, by the explicit formula for $h$, the only possibility for $h(p,u_G) \ne 0$ would be the case when $\gamma(\delta) >0$ and $c_{21}(c_{22}) <0$. We will show that it never happens, namely $c_{21}>0$ whenever $\gamma(\delta)>0$. So assume that  $\gamma(\delta)>0$.

If $\gamma =0$, then $\delta>0$, and so $c_{21}=\delta yb>0$. So assume  that $\gamma >0$. Suppose for a contradiction that  $c_{21} \leq 0$. Then $\frac{\delta}{\gamma} \leq \frac{x-1}{y} <\R$ which contradicts the assumption that $\tp(\gamma,\delta/\R,x,y)$ is finitely satisfiable in $\R$.

(2) The proof is similar and left to the reader.

%%%Krzys25: Added ``discussed before Lemma \ref{lemma: properties of h}''.
(3) By the 2-cocycle formula discussed before Lemma \ref{lemma: properties of h}, we have $h(u_G,g) + h(u_Gg,u_G) = h(u_G,gu_G) + h(g,u_G)$. So the conclusion follows from (1) and (2).
 \end{proof}

 \begin{proof}[Proof of Proposition \ref{proposition: product of tau topologies}]
Denote by $\mathcal{T}$ the product of the $\tau$-topologies on $u_G\M_G$ and $u_{\Z}\M_{\Z}$. Our goal is to prove that $\tau = \mathcal{T}$.

%%%Krzys: "is closed in $\mathcal{T}$" replaced by "is closed in $\tau$".
($\supseteq$) It  is enough to show that any subbasic closed set of the form $F \times u_{\Z}\M_{\Z}$ or $u_{G}\M_{G} \times E$ (where $F \subseteq u_G\M_G$ and $E \subseteq u_{\Z}\M_{\Z}$ are $\tau$-closed) is closed in $\tau$.

First, consider $F \times u_{\Z}\M_{\Z}$. Take any $a \in \cl_{\tau}(F \times u_{\Z}\M_{\Z})$. Then $a=\lim (g_i,n_i)(f_i,a_i)$, where $((g_i,n_i))_i$ is a net from $\tilde{G}$ converging to $u_{\tilde{G}}=(u_G,u_{\Z})$ and $(f_i,a_i) \in F \times u_{\Z}\M_{\Z}$. As $(g_i,n_i)(f_i,a_i)=(g_if_i,n_i+a_i+h(g_i,f_i))\in \{g_if_i\} \times \M_{\Z}$, we get that $a \in \cl_{\tau}(F) \times \M_{\Z} = F \times u_{\Z}\M_{\Z}$, as required.

Now, consider $u_{G}\M_{G} \times E$. Take any $a \in \cl_{\tau}(u_{G}\M_{G} \times E)$. Then $a=\lim (g_i,n_i)(f_i,a_i)$, where $((g_i,n_i))_i$ is a net from $\tilde{G}$ converging to $u_{\tilde{G}}=(u_G,u_{\Z})$ and $(f_i,a_i) \in u_{G}\M_{G} \times E$. 

\begin{clm}
$h(g_i,f_i) =0$ for $i$ large enough.
\end{clm}

\begin{clmproof}
Since $\lim g_i =u_G$ is the type over $M$ of a matrix with positive left bottom entry, there is $i_0$ such that the left bottom entry of $g_i$ is positive for all $i>i_0$. Consider any $i>i_0$. If $f_i=u_G$, then $h(g_i,f_i) =0$ by Lemma \ref{lemma: properties of h}(1). If $f_i=q_1$, then the left bottom entry of any matrix realizing $f_i$ is negative, so  $h(g_i,f_i) =0$ by the explicit formula for $h$.
\end{clmproof}

%%%Krzys25: I added ``for $i$ large enough''.
By this claim, $(g_i,n_i)(f_i,a_i)=(g_if_i,n_i+a_i+h(g_i,f_i))=(g_if_i,n_i+a_i)$ for $i$ large enough. So $a \in \cl_\tau(u_G\M_G) \times \cl_\tau(E) = u_G\M_G \times E$, as required.

($\subseteq$) Consider a $\tau$-closed $A \subseteq u_{\tilde{G}}\M_{\tilde{G}}$. We need to show that it is closed in $\mathcal{T}$. So take any $a =(a_1,a_2)\in \cl_{\mathcal{T}}(A)$. There are are nets $(a_{1,i})_i \subseteq u_G\M_G$ and $(a_{2,i})_i\subseteq u_\Z\M_\Z$ $\tau$-converging to $a_1$ and $a_2$, respectively, with $(a_{1,i},a_{2,i}) \in A$ for all $i$. 
%%%Krzys: $S_{\Z}(M)$ instead of $\M_{\Z}(M)$. 
Passing to subnets, we can assume that the nets $(a_{1,i})_i$ and $(a_{2,i})_i$ converge in the usual topologies on $S_G(M)$ and $S_{\Z}(M)$ to some $b_1$ and $b_2$, respectively. 
%%%Krzys: I wrote additively $u_{\Z}+b_2 =a_2$ in place of $u_{\Z}b_2 =a_2$
By Lemma \ref{fct:ulimit} and Hausdorffness of $u_G\M_G$ and $u_{\Z}\M_{\Z}$, we get $u_Gb_1=a_1$ and $u_{\Z}+b_2 =a_2$. Approximating $u_G$ by elements of $G$ and $u_{\Z}$ by elements of $\Z$, using left continuity of the semigroup operations and the fact that the actions of $G$ on $S_G(M)$ and of $\Z$ on $S_{\Z}(M)$ are continuous, passing to subnets, we can assume that there are nets $(g_i)_i$ in $G$ and $(n_i)_i$ in $\Z$ converging to $u_G$ and $u_{\Z}$, respectively, and such that $\lim g_ia_{1,i} = a_1$ and $\lim n_i+a_{2,i}=a_2$ (in the usual topology on type spaces). Then $\lim (g_i,n_i) =(u_G,u_{\Z})$. On the other hand,  by Claim 1, $h(g_i,a_{1,i}) =0$ for sufficiently large $i$'s, and so
$$(g_i,n_i)(a_{1,i},a_{2,i}) = (g_ia_{1,i}, n_i+ a_{2,i}+h(g_i,a_{1,i}))=(g_ia_{1,i}, n_i+ a_{2,i})$$
for sufficiently large $i$'s.
Hence, $\lim (g_i,n_i)(a_{1,i},a_{2,i}) =(a_1,a_2)$. Therefore, $a \in \cl_{\tau} (A)=A$.
\end{proof}

The next lemma follows by an elementary matrix computation.

\begin{lemma}\label{lemma: formula for u_G p u_G}
For every $\tp(B/M) \in S_G(M)$ with 
$B=\left( 
\begin{array}{ll}
\alpha & \beta\\
\gamma & \delta
\end{array}
\right),$
writing $u_G=\tp(A/M)=\tp(A'/M)$ for
$ A:=\left(
\begin{array}{cc}
(1-x)b & (1-x)c -yb^{-1}\\
yb & yc + (1-x)b^{-1}
\end{array}
\right),$
$ A':=\left(
\begin{array}{cc}
(1-x')b' & (1-x')c' -y'b'^{-1}\\
y'b' & y'c' + (1-x')b'^{-1}
\end{array}
\right)$
with the elements satisfying the requirements described before and such that $\tp(B/M,A)$ and $\tp(A'/M,B,A)$ are coheirs over $M$, we have that $u_G\tp(B/M)  u_G = \tp(C/M)$ with the left bottom entry of $C$ equal to  $y'b'(\alpha(1-x)b + \beta yb) + (y'c' +(1-x')b'^{-1})(\gamma(1-x)b +\delta yb)$.
\end{lemma}

\begin{lemma}\phantomsection\label{lemma: u_G B u_G=u_G}
\begin{enumerate}
\item For $B=\left(
\begin{array}{rr}
\alpha & \beta\\
\gamma & \delta
\end{array}
\right) \in G$ we have that $u_G\tp(B/M)u_G=u_G=q_0$ if $\gamma >0$, and $u_G\tp(B/M)u_G= q_1$ if $\gamma<0$.

\item $u_G 
\left(
\begin{array}{rr}
0 & -1\\
1 & 0
\end{array}
\right)
u_G = u_G$.

\item  $u_G 
\tp\left(\left(
\begin{array}{rr}
-1 & 0\\
\gamma & -1
\end{array}
\right)/M\right)
u_G = q_1$ for all positive infinitesimals $\gamma$.
\end{enumerate}
\end{lemma}

\begin{proof}
First, let $B=\left(
\begin{array}{rr}
\alpha & \beta\\
\gamma & \delta
\end{array}
\right) \in \bar G$ with $\gamma >0$. 
%Pick $x,y,b,c,x',y',b',c'$ and matrices $A$ and $A'$ as in Lemma \ref{lemma: formula for u_G p u_G}. 
Pick $x,y,b,c,x',y',b',c'$ and matrices $A$ and $A'$ as in Lemma \ref{lemma: formula for u_G p u_G}, satisfying additionally that $\tp(B/M,x,y,b,c)$ and $\tp(x',y',b',c'/M,x,y,b,c,\alpha,\beta,\gamma,\delta)$ are coheirs over $M$. 
Observe that $\gamma(1-x)b +\delta yb >0$. Indeed, it is clear if $\delta \geq 0$. If $\delta<0$, then it is equivalent to $-\frac{\gamma}{\delta} > \frac{y}{1-x}$ which is true as $\frac{y}{1-x}$ is a positive infinitesimal, $-\frac{\gamma}{\delta}>0$, and $\tp(\gamma,\delta/M,x,y)$ is a coheir over $M$. Let $d$ be the left bottom entry of $A'BA$. Hence, by Lemma \ref{lemma: formula for u_G p u_G}, we conclude that $d>0$ if $\alpha(1-x)b + \beta yb\geq 0$. In the case when $\alpha(1-x)b + \beta yb<0$, we have that $d>0$ if and only if $\frac{b'}{c'} < (1+ \frac{1-x'}{y'b'c'}) (- \frac{\gamma(1-x) + \delta y}{\alpha(1-x)+\beta y}) =  (1+ \frac{1-x'}{y'b'c'}) (- \frac{\gamma + \delta \frac{y}{1-x}}{\alpha+\beta \frac{y}{1-x}}) =:\zeta$.

(1) Since $u_G = \tp(A/M)$ and $q_1 = \tp(-A/M)$, replacing $B$ by $-B$, we see that it is enough to consider the case when $\gamma >0$ and to show that then $d>0$. By the above consideration, this boils down to showing that $\frac{b'}{c'} <  \zeta$ if $\alpha(1-x)b + \beta yb<0$.  By the assumption of (1), $\alpha,\beta,\gamma,\delta \in \R$.  Thus, $\alpha(1-x)b + \beta yb<0$ implies that $\alpha<\beta \frac{y}{x-1}$ which is infinitesimal, so $\alpha \leq 0$. Now, if $\alpha=0$, then $\zeta >\R$, so $\zeta > \frac{b'}{c'}$ (as $\frac{b'}{c'}$ is infinitesimal). 
%%%Krzys25: I added more details (a justification of the inequality $\zeta >-\frac{\gamma}{2\alpha}$ which the referee did not check).
%If $\alpha<0$, then $\zeta> -\frac{\gamma}{2\alpha}> \frac{b'}{c'}$, as $-\frac{\gamma}{2\alpha}$ is a positive real number.
If $\alpha<0$, then $\zeta> - \frac{\gamma + \delta \frac{y}{1-x}}{\alpha+\beta \frac{y}{1-x}} > -\frac{\gamma}{2\alpha}> \frac{b'}{c'}$, where the first inequality follows from the fact that $\frac{1-x'}{y'b'c'}>0$ and $- \frac{\gamma + \delta \frac{y}{1-x}}{\alpha+\beta \frac{y}{1-x}}>0$, the last one from the fact that $-\frac{\gamma}{2\alpha}$ is a positive real number and $b'/c'$ infinitesimal, and the middle one is easily seen to be equivalent to $-\gamma \alpha >(2 \delta \alpha -\gamma \beta)\frac{y}{1-x}$ which is obviously true  as $-\gamma \alpha$ is a positive real and $(2 \delta \alpha -\gamma \beta)\frac{y}{1-x}$ is infinitesimal.

(2) is a particular case of (1).

(3) %In the first paragraph of the proof, we can pick $x,y,b,c,x',y',b',c'$ satisfying additinally that $\tp(B/M,x,y,b,c)$ and $\tp(x',y',b',c'/M,x,y,b,c,\alpha,\beta,\gamma,\delta)$ are coheirs over $M$. 
%%%Krzys25: Added ``$d \leq 0$; equivalently that''.
%It is enough to see that $\zeta \leq \frac{b'}{c'}$. 
It is enough to see that $d \leq 0$; equivalently that $\zeta \leq \frac{b'}{c'}$. 
%We have $\zeta = (1+ \frac{1-x'}{y'b'c'}) (\gamma -  \frac{y}{1-x})$ which is a positive infinitesimal (as $\frac{1-x'}{y'b'c'}$, $\gamma$ and  $\frac{y}{1-x}$ are positive infinitesimals, and $\tp(\gamma/M,x,y)$ is a coheir over $M$). 
We have $\zeta = (1+ \frac{1-x'}{y'b'c'}) (\gamma -  \frac{y}{1-x}) < 2(\gamma -  \frac{y}{1-x})$ which is a positive infinitesimal (as $\frac{1-x'}{y'b'c'}$, $\gamma$ and  $\frac{y}{1-x}$ are positive infinitesimals, and $\tp(\gamma/M,x,y)$ is a coheir over $M$). 
The conclusion follows from the fact that $\frac{b'}{c'}$ is a positive infinitesimal and $\tp(b',c'/M,x,y, \gamma)$ is a coheir over $M$. 
%%%%%%%%%%%%%%%%%%%%%%%%%%%%%%%%%%%%%%%%%%%%%%%%%%%%%%%%%%%
%%%%%%%%%%%%%%%%%%%%%%%%%%%%%%%%%%%%%%%%%%%%%%%%%%%%%%%%%%%
\begin{comment}
By Lemma \ref{lemma: formula for u_G p u_G}, 
$u_G 
\left(
\begin{array}{rr}
0 & -1\\
1 & 0
\end{array}
\right)
u_G = \tp(C/M)$, where the left bottom entry of $C$ equals $-y'b'yb + (y'c' +(1-x')b'^{-1})(1-x)b=:d$. It is enough to show that this entry is positive (as the only other element in $u_G\M_G$ is the type of a matrix with a negative left bottom entry). We see that $d>0$ if and only if $\frac{b'}{c'} < (1+ \frac{1-x'}{y'b'c'})\frac{1-x}{y}$. And the last inequality holds, because $c' > \dcl(\R,b')$ implies that $\frac{b'}{c'}$ is infinitesimal, whereas  $(1+ \frac{1-x'}{y'b'c'})\frac{1-x}{y} > \frac{1-x}{y} > \R$ as $x,y$ are positive infinitesimals and $x',y',b',c'$ are positive.
\end{comment}
%%%%%%%%%%%%%%%%%%%%%%%%%%%%%%%%%%%%%%%%%%%%%%%%%%%%%%%%%%%%%%%%
%%%%%%%%%%%%%%%%%%%%%%%%%%%%%%%%%%%%%%%%%%%%%%%%%%%%%%%%%%%%%%%%%
\end{proof}

We have now all the tools to prove Proposition \ref{proposition: structure of generics in the universal cover}.

\begin{proof}[Proof of Proposition \ref{proposition: structure of generics in the universal cover}]
Let $B:= \left( 
\begin{array}{rr}
0 & -1\\
1 & 0
\end{array}
\right)$. Then $B^2= -I$ and $B^4=I$. So, by the explicit formula for the 2-cocycle $h$, we have $h(B,B)=1$ and $h(B^2,B^2)=-1$. Hence, working in $\tilde{G}$, we get 
$$(*)\;\;\;\;\;\; (B,0)^{4}=(I, 2h(B,B) +h(B^2,B^2))=(I,1).$$

%By Fact \ref{corollary: jagiella's isomorphism} and the comments following it, $u_{\tilde{G}}\M_{\tilde{G}}$ is topologically identified with the product $\Z_2 \times \hat{\Z}$ via the isomorphism $(f_1,f_2)$. 
%%%Krzys: I simplified the next sentence.
%By Corollary \ref{corollary: jagiella's isomorphism} and Proposition \ref{proposition: product of tau topologies}, $u_{\tilde{G}}\M_{\tilde{G}}$ is topologically isomorphic to $\Z_2 \times \hat{\Z}$ (with the appropriate group operation), so Hausdorff. Hence, $H(u_{\tilde{G}}\M_{\tilde{G}})$ is trivial.
As observed after Corollary \ref{corollary: jagiella's isomorphism},  $u_{\tilde{G}}\M_{\tilde{G}} \cong \hat{\Z}$ is Hausdorff, so $H(u_{\tilde{G}}\M_{\tilde{G}})$ is trivial.
%%%Krzys25: I added ``By the explicit formula for $f_2$ and the fact that $f_2(u_{\Z})=0$,''. In order to shorten formulas, we were previously writing $n$ instead of $\tp(n\Z)$ or instead of $n/\bar{\Z}^0$, but now we wrote the precise expressions (to avoid confusions).
By the explicit formula for $f_2$ and the fact that $f_2(u_{\Z})=0$, we also have that $f_2(u_\Z+ \tp(n/\Z) +u_\Z)=n/\bar{\Z}^0$ for $n \in \Z$. Take $f \colon \tilde{G} \to u_{\tilde{G}}\M_{\tilde{G}}$ from Theorem \ref{theorem: main theorem for definable G}. Using Lemma \ref{lemma: properties of h},
\begin{flalign*}
&\: f((g,n)):=(u_G,u_\Z)(g,n)(u_G,u_\Z) =(u_G g u_G, u_\Z +\tp(n/\Z)+u_\Z+ h(g,u_G)+h(u_G,gu_G))=\\
&\: (u_Ggu_G,u_\Z + \tp(n/\Z) + u_\Z).
\end{flalign*}
%%%Krzys25: I changed the next sentence referring also to 5.9 and $f_2$.
%Hence, by Theorem \ref{theorem: main theorem for definable G}(3), there is 
%Hence,  using Theorem \ref{theorem: main theorem for definable G}(3), Proposition  \ref{proposition: product of tau topologies} and the topological isomorphism $f_2$, we get a positive $k \in \mathbb{N}$ such that $\{g \in G: u_G g u_G =u_G\} \times k\mathbb{Z} \subseteq X^{14}$. 
\begin{clm}
There exists a positive $k \in \mathbb{N}$ such that $\{g \in G: u_G g u_G =u_G\} \times k\mathbb{Z} \subseteq X^{14}$. 
\end{clm}

\begin{clmproof}
 Bearing in mind Proposition 5.9, Theorem 5.1(3) yields a set $U$ with $f^{-1}[U] \subseteq X^{14}$ which is of the form $U_1 \times U_2$, where $U_1$ is a neighborhood of $u_G$ and $U_2$ is a neighborhood of $u_\Z$. Since $f_2$ is a homeomorphism, we get that $f_2[U_2]$ is a neighborhood of the neutral element in $\mathbb{\hat{Z}}$. On the other hand, a basis of open neighborhoods of the neutral element in $\hat{\mathbb{Z}}$ consists of the clopen subgroups $G_k:=\varprojlim_n k\mathbb{Z}/n\mathbb{Z}$ for $k \in \mathbb{N} \setminus \{0\}$. So there is $k$ such that $G_k \subseteq f_2[U_2]$. Then $U_2 \supseteq f_2^{-1}[G_k] = \{\tp(a/\mathbb{Z}) \in u_{\mathbb{Z}}\mathcal{M}_{\mathbb{Z}}: a \in k\bar{\mathbb{Z}}\}= u_{\mathbb{Z}} + \{\tp(a/\mathbb{Z}): a \in k\bar{\mathbb{Z}}\} + u_{\mathbb{Z}} \supseteq u_{\mathbb{Z}}+\{\tp(a/\mathbb{Z}): a \in k\mathbb{Z}\} + u_{\mathbb{Z}}$ (note here that under the natural identification of $\hat{Z}$ with $\bar{\mathbb{Z}}/\bar{\mathbb{Z}}^0$, the group $G_k$ becomes $k\bar{\mathbb{Z}}/\bar{\mathbb{Z}}^0$). Thus, $\{ u_G\} \times (u_\Z+k\Z+u_\Z) \subseteq U$. So $\{g \in G: u_Ggu_G =u_G\} \times k\Z \subseteq f^{-1}[U] \subseteq X^{14}$, as required.
\end{clmproof}

By Claim 1 and Lemma \ref{lemma: u_G B u_G=u_G}(2), $\{B\} \times k\Z \subseteq X^{14}$. 
%Therefore, using ($*$), we conclude that $\{I\} \times (1+k\Z) \subseteq X^{14 \cdot 4}=X^{56}$, 
%%%Krzys: ``$\{ I\}$'' instead of ``$I$''.
%%%Krzys25: ``11'' instead of ``12'' due to the improvement of the statement of Cor. 5.6. And in consequence $56 \cdot 11 =616$ in place of $56 \cdot 12=672$, and so $616 +24 =640$ instead of $672+24 =696$.
%and so $\{I\} \times (\{-12,\dots,0, \dots, 12\} +k\Z) \subseteq X^{56\cdot 12}=X^{672}$. Using Corollary \ref{corollary: from Gismatullin}, this implies that $G \times k\Z \subseteq X^{672+24}=X^{696}$.
%%%Krzys25: Added `` as $X$ is symmetric''.
%and so, as $X$ is symmetric,  $\{I\} \times (\{-11,\dots,0, \dots, 11\} +k\Z) \subseteq X^{56\cdot 11}=X^{616}$. Using Corollary \ref{corollary: from Gismatullin}, this implies that $G \times k\Z \subseteq X^{616+24}=X^{640}$.
Therefore, using ($*$), we conclude that $\{I\} \times (1+k\Z) \subseteq X^{14 \cdot 4}=X^{56}$. Indeed: $X^{14 \cdot 4}\supseteq (\{B\} \times k\mathbb{Z}  )^4 \supseteq (B,0)^3 \cdot (\{B\} \times k\mathbb{Z})  = (B,0)^4 \cdot (\{I\} \times  k\mathbb{Z}) = (I,1) \cdot (\{I\} \times  k\mathbb{Z}) = \{I\} \times (1 +k\mathbb{Z})$, where the first equality follows since $\{B\}\times k\mathbb{Z} =(B,0) \cdot (\{I\} \times k\mathbb{Z})$ (which we have as $h(B,I)=0$), the second equality follows from $(*)$, and the last one from the fact that $h(I,I)=0$.
Thus, as $X$ is symmetric, we obtain $\{I\} \times (\{-11,\dots,0, \dots, 11\} +k\Z) \subseteq X^{56\cdot 11}=X^{616}$. Using Corollary \ref{corollary: from Gismatullin}, this implies that $G \times k\Z \subseteq X^{616+24}=X^{640}$.
\end{proof}

One could also prove more directly (without using topological dynamics) a version of Proposition \ref{proposition: structure of generics in the universal cover}
%%%Krzys25: 640 instead of 696.
 with a bigger number in place of 640,  but we will not do that, as our point was to illustrate by a non-trivial example  how Theorem \ref{theorem: main theorem for definable G} leads to a better understanding of generics in definable groups.

%%%Krzys: I extended the next sentence.
Finally, we give a negative answer to Question \ref{question: bar f = hat f}(3), and a positive answer to Question \ref{question: bar f = hat f}(2) in the particular case when $X$ is a definable, generic, symmetric subset of $\tilde{G}$. First, let us describe the context. Let $X$ be a definable, generic, symmetric subset of $\tilde{G}$. 
%%%Krzys25:   ``$X$ is an approximate subgroup definable in $M$'' instead of ``$X$ is a definable in $M$ approximate subgroup''.
Then $X$ is an approximate subgroup definable in $M$ and, as discussed after Proposition \ref{proposition: structure of generics in the universal cover}, $\tilde{G}=\langle X \rangle$. 
%%%Krzys: Instead of ``we have $S_{\tilde{G},M}(N)=S_{\tilde{G}}(M)$ (as discussed above).'' I wrote ``the flows  $S_{\tilde{G},M}(N)$ and $S_{\tilde{G}}(M)$ are identified (as discussed above)''.
By definability of types in $S(M)$, the flows  $S_{\tilde{G},M}(N)$ and $S_{\tilde{G}}(M)$ are identified (as discussed above). 
%%%Krzys: I shortened the next sentence.
%By Corollary \ref{corollary: jagiella's isomorphism}  and Proposition \ref{proposition: product of tau topologies}, $u_{\tilde{G}}\M_{\tilde{G}}$ is topologically isomorphic to $\Z_2 \times \hat{\Z}$ (with the approriate group operation), so Hausdorff. 
%As we explained in the above proof, $u_{\tilde{G}}\M_{\tilde{G}}$ is topologically isomorphic to $\Z_2 \times \hat{\Z}$, and so $H(u_{\tilde{G}}\M_{\tilde{G}})$ is trivial.
We have already proved that $H(u_{\tilde{G}}\M_{\tilde{G}})$ is trivial.
%Hence, $H(u_{\tilde{G}}\M_{\tilde{G}})$ is trivial. 
%%%Krzys: I changed the rest of the paragraph as you suggested.
%Let $f \colon \tilde{G} \to u_{\tilde{G}}\M_{\tilde{G}}$ be the generalized definable locally compact model from Theorem \ref{theorem: main Theorem}. Applying the construction from Theorem \ref{theorem: universality: existence} to $h:=f$, we obtain a function $\bar f = f_M \colon S_{\tilde{G}}(M) \to u_{\tilde{G}}\M_{\tilde{G}}$ extending $f$. On the other hand, we have the function $\hat{f} \colon S_{\tilde{G}}(M) \to u_{\tilde{G}}\M_{\tilde{G}}$ given by $\hat{f}(p) := u_{\tilde{G}}p u_{\tilde{G}}$ which also extends $f$.
Let $f \colon \tilde{G} \to u_{\tilde{G}}\M_{\tilde{G}}$ be the generalized definable locally compact model from Theorem \ref{theorem: main Theorem} and let  $\bar f = f_M \colon S_{\tilde{G}}(M) \to u_{\tilde{G}}\M_{\tilde{G}}$ be as discussed before Question \ref{question: bar f = hat f}. Note that $\bar f$ extends $f$. On the other hand, we have the function $\hat{f} \colon S_{\tilde{G}}(M) \to u_{\tilde{G}}\M_{\tilde{G}}$ given by $\hat{f}(p) := u_{\tilde{G}}p u_{\tilde{G}}$ which also extends $f$.

\begin{proposition}\label{proposition: universal cover yields an answer}
The function $\bar f$ is uniquely determined by the construction from  Theorem \ref{theorem: universality: existence} and continuous, whereas $\hat{f}$ is not continuous, and so $\bar f \ne \hat{f}$. However, $\bar f |_{u_{\tilde{G}}\M_{\tilde{G}}} = \hat{f} |_{u_{\tilde{G}}\M_{\tilde{G}}} = \id$.
\end{proposition}

\begin{proof}
%By Fact \ref{fact: basic from Jagiella} and Lemma \ref{lemma: properties of h}(1,3), identifying $u_{\Z}\M_{\Z}$ with $\hat{\Z}$, we have
%\begin{align*}
%&\: f((g,n)):= u_{\tilde{G}} (g,n) u_{\tilde{G}} =(u_G,0)(g,n)(u_G,0)= (u_G g u_G, n +h(g,u_G)+h(u_G,gu_G))=\\
%&\: (u_G g u_G,n).
%\end{align*}
Identifying $u_{\Z}\M_{\Z}$ with $\hat{\Z}$ via $f_2$, the second displayed computation in the proof of Proposition \ref{proposition: structure of generics in the universal cover} yields 
$$f((g,n))= (u_G g u_G,n).$$

%
%%%Krzys: I changed the next sentence, as much less was needed here.
It follows from Lemma \ref{lemma: formula for u_G p u_G} and definability of types in $S(M)$ that the sets $\{g \in G: u_G g u_G =u_G\}$ and $\{g \in G: u_G g u_G \ne u_G\}$ are both definable.
%By definability of types in $S(M)$ and the explicit formula for $h$, we easily deduce (using Lemma \ref{lemma: formula for u_G p u_G} expanded by an explicit formula for the right bottom entry of $C$) that the sets $\{g \in G: u_G g u_G =u_G\}$ and $\{g \in G: u_G g u_G \ne u_G\}$ are definable. 
On the other hand, the function $\Z \to \hat{\Z}$ given by $n \mapsto n$ is definable in the sense that the preimages of any two disjoint closed subsets of $\hat{\Z}$ can be separated by a definable set. 

All of this together with Proposition \ref{proposition: product of tau topologies} implies that $f\colon \tilde{G} \to u_{\tilde{G}}\M_{\tilde{G}}$ is a definable map. Therefore, by Lemma 3.2 of \cite{GePePi2} and its proof, $\bar{f} =f_M$ is uniquely determined by the construction  from  Theorem \ref{theorem: universality: existence},
%%%Krzys25: ``, and $\bar f$ is continuous'' instead of ``and continuous''.
and $\bar f$ is continuous.

Pick a positive infinitesimal $\gamma$. Let 
$B:=\left(
\begin{array}{rr}
-1 & 0\\
\gamma & -1
\end{array}
\right).
$
Choose any net $(g_i)_i$ of elements of $G$ converging to $p:=\tp(B/M)$. Then the left bottom entries of the matrices $g_i$ are positive for all $i>i_0$ for some $i_0$. So, by Lemma \ref{lemma: u_G B u_G=u_G}(1), $u_G g_i u_G =u_G$ for all $i>i_0$. On the other hand, by Lemma \ref{lemma: u_G B u_G=u_G}(3), $u_G p u_G =q_1$. 
%%%Krzys I removed more precise formulas based on Lemma 5.10, as they were not needed here.
%Therefore, identifying $u_{\Z}\M_{\Z}$ with $\hat{\Z}$ and using Lemma \ref{lemma: properties of h}, we get
%$$\hat{f}((p,0))=(u_G p u_G, h(p,u_G)+h(u_G,pu_G))=(q_1,h(u_G,pu_G))$$
%and
%$$\hat{f}((g_i,0))=(u_G g_i u_G, h(g_i,u_G)+h(u_G,g_iu_G))=(u_G,0)$$
%for all $i>i_0$.
Therefore, $\hat{f}((p,0)) \in \{u_G p u_G\} \times u_\Z \M_\Z = \{q_1\} \times u_\Z \M_\Z$ and $\hat{f}((g_i,0)) \in \{u_G g_i u_G\} \times u_\Z \M_\Z = \{u_G\} \times u_\Z \M_\Z$ for all $i>i_0$.
Since the net $((g_i,0))_i$ tends to $(p,0)$ and $q_1 \ne u_G$, we conclude that $\hat{f}$ is not continuous at $(p,0)$. Hence, $\bar f \ne \hat{f}$ by continuity of $\bar f$. More precisely, since $\bar f |_{\tilde{G}} = \hat{f} |_{\tilde{G}}$, we get that 
%%%Krzys25: $\bar f((p,0)) \ne \hat{f}((p,0))$ instead of $\bar f(p) \ne \hat{f}(p)$.
%$\bar f(p) \ne \hat{f}(p)$.
$\bar f((p,0)) \ne \hat{f}((p,0))$.

It remains to show that  $\bar f |_{u_{\tilde{G}}\M_{\tilde{G}}} = \id$, as directly from the definition of $\hat{f}$ we have $\hat{f} |_{u_{\tilde{G}}\M_{\tilde{G}}} = \id$. Consider any $n \in \hat{\Z}$. Our goal is to show that $\bar f((u_G,n)) = (u_G,n)$ and $\bar f((-u_G,n)) = (-u_G,n)$. We will prove the first equality; the second one can be proved analogously.

Choose any net $((g_i,n_i))_i$ from $\tilde{G}$ converging to $(u_G,n)$. Then the left bottom entry of $g_i$ is positive for all $i>i_0$ for some $i_0$. By Lemma \ref{lemma: u_G B u_G=u_G}(1), $u_G g_i u_G =u_G$ for all $i>i_0$. Therefore, using Lemma \ref{lemma: properties of h}, we get 
$$\bar f((g_i,n_i)) =(u_Gg_i u_G, n_i +h(g_i,u_G)+ h(u_G,g_iu_G))= (u_G,n_i)$$
for all $i>i_0$, so
it clearly tends to $(u_G,n)$. Since the net $((g_i,n_i))_i$ tends to $(u_G,n)$, by continuity of $\bar f$, we conclude that $\bar f ((u_G,n)) = (u_G,n)$. 
\end{proof}

\section*{Acknowledgments}
We would like to thank the referee for their very careful reading of the manuscript and for many suggestions which helped us to improve the presentation.

\printbibliography
\nocite{*}

\end{document}